\documentclass[final,3p,times]{elsarticle}

\usepackage{framed,multirow}

\usepackage{amssymb}
\usepackage{latexsym}

\usepackage{url}
\usepackage{xcolor}
\definecolor{newcolor}{rgb}{.8,.349,.1}

\usepackage{amsmath}
\usepackage{amsfonts}
\usepackage{epsfig}
\usepackage{bigcenter}
\usepackage{cancel}
\usepackage{lineno}
\usepackage{subfloat}
\usepackage{subfig}
\usepackage{enumitem}
\usepackage{hyperref}
\usepackage{cleveref}
\crefname{equation}{}{}
\crefname{figure}{Fig.}{Figs.}
\crefname{appendix}{}{}
\crefname{table}{Tab.}{Tabs.}
\Crefname{ALC@unique}{Line}{Lines} 
\usepackage{tikz}

\def\du{[\![}
\def\df{]\!]}
\def\ol{\overline}

\def\etab{\boldsymbol{\eta}}
\def\psib{\boldsymbol{\psi}}

\def\bxi{{\boldsymbol \xi}}
\def\bcW{\boldsymbol{\cal W}}
\def\p{\mathrm{p}}
\def\T{\mathrm{T}}
\def\s{\mathrm{s}}
\newcommand{\red}[1]{\textcolor{black}{#1}}
\newcommand{\blue}[1]{\textcolor{black}{#1}}

\newcommand{\orange}[1]{\textcolor{black}{#1}}

\newtheorem{lemma}{Lemma}[section]

\newtheorem{theorem}{Theorem}[section]
\newdefinition{remark}{Remark}[section]
\newproof{proof}{Proof}

\DeclareMathAlphabet\mathbfcal{OMS}{cmsy}{b}{n}

\def \LOCALFIGPATHRP {./}

\def \LOCALFIGPATHSBINS {./}
\def \LOCALFIGPATHSBIKT {./}
\def \LOCALFIGPATHRMI {./}

\begin{document}


\begin{frontmatter}

\title{Entropy stable, robust and high-order DGSEM for the compressible multicomponent Euler equations}

\author[rvt1]{Florent Renac\corref{cor1}}
\ead{florent.renac@onera.fr}
\cortext[cor1]{Corresponding author. Tel.: +33 1 46 73 37 44; fax.: +33 1 46 73 41 66.}
\address[rvt1]{DAAA, ONERA, Universit\'e Paris Saclay, F-92322 Ch\^atillon, France}


\begin{abstract}
This work concerns the numerical approximation of a multicomponent compressible Euler system for a fluid mixture in multiple space dimensions on unstructured meshes with a high-order discontinuous Galerkin spectral element method (DGSEM). We first derive an entropy stable (ES) and robust (i.e., that preserves the positivity of the partial densities and internal energy) three-point finite volume scheme using relaxation-based approximate Riemann solvers from Bouchut [Nonlinear stability of finite volume methods for hyperbolic conservation laws and well-balanced schemes for sources, Frontiers in Mathematics, Birkhauser, 2004] and Coquel and Perthame [\textit{SIAM J. Numer. Anal.}, 35 (1998), 2223--2249]. Then, we consider the DGSEM based on collocation of quadrature and interpolation points which relies on the framework introduced by Fisher and Carpenter [\textit{J. Comput. Phys.}, 252 (2013), 518--557] and Gassner [\textit{SIAM J. Sci. Comput.}, 35 (2013), A1233--A1253]. We replace the physical fluxes in the integrals over discretization elements by entropy conservative numerical fluxes [Tadmor, \textit{Math. Comput.}, 49 (1987), 91--103], while ES numerical fluxes are used at element interfaces. We thus derive a two-point numerical flux satisfying the Tadmor's entropy conservation condition and use the numerical flux from the three-point scheme as ES flux. Time discretization is performed with a strong-stability preserving Runge-Kutta scheme. We then derive conditions on the numerical parameters to guaranty a semi-discrete entropy inequality as well as positivity of the cell average of the partial densities and internal energy of the fully discrete DGSEM at any approximation order. The later results allow to use existing limiters in order to restore positivity of nodal values within elements. The scheme also resolves exactly stationary material interfaces. Numerical experiments in one and two space dimensions on flows with discontinuous solutions support the conclusions of our analysis and highlight stability, robustness and high resolution of the scheme.
\end{abstract}


\end{frontmatter}


%
%
\section{Introduction}\label{sec:intro}

\subsection{Discretization of compressible multicomponent flows}

The accurate and robust (i.e., preserving the solution in the set of admissible states) simulation of compressible multicomponent flows with material interfaces is of strong significance in many engineering applications (e.g., combustion in propulsion systems, explosive detonation products) and scientific applications (e.g., flow instabilities, chemical reactions, phase changes). These flows may involve nonlinear waves such as shock and rarefaction waves, contact waves, material interfaces separating different fluids, and their interactions which usually trigger phenomena leading to small scale flow structures. The discussion in this paper focuses on the nonlinear analysis of a high-order discretization of the multicomponent compressible Euler system for a fluid mixture in multiple space dimensions. Numerical approximation of such flows based on interface capturing methods has been the subject of numerous works. Conservative schemes suffer from spurious oscillations at material interfaces due to the violation of pressure equilibrium \cite{abgrall_96}. Non-conservative formulations have been proposed \cite{abgrall_96,abgrall_karni_01,karni_mutlicomp_96} but do not ensure mass conservation of the different species. The ghost-fluid method \cite{fedkiw_etal_99} removes the pressure oscillations but the approximation of the interface is limited to first order, while the front tracking method \cite{glimm_etal_98} requires expensive operations and do not satisfy mass conservation. High-order discretizations have also been considered to resolve the broad range of scales usually present in these flows. Though not exhaustive, we refer to the works on finite differences \cite{latini_etal_07,kawai_terashima_11,mohaved_johnsen_13,capuano_etal_18}, finite volume methods \cite{xiong_etal_12,johnsen_etal_06,houim_kuo_11,vilar_etal_16,gouasmi_etal_20}, or discontinuous Galerkin (DG) methods \cite{h_de_frahan_etal_15,lv_ihme_14} and references therein. We here consider the discontinuous Galerkin spectral element method (DGSEM) based on collocation between interpolation and quadrature points \cite{despres98,kopriva_gassner10}. 

\subsection{Entropy conservative and entropy stable flux differencing schemes}\label{sec:intro_ECEF}

Using diagonal norm summation-by-parts (SBP) operators and the entropy conservative (EC) numerical fluxes from Tadmor \cite{tadmor87}, semi-discrete EC finite-difference and spectral collocation schemes have been derived in \cite{fisher_carpenter_13,carpenter_etal14} for nonlinear conservation laws and the DGSEM falls into this general framework of conservative elementwise flux differencing schemes. Entropy stable (ES), also known as entropy dissipative, DGSEM for the compressible Euler equations on hexahedral \cite{winters_etal_16} and triangular \cite{chen_shu_17} meshes have been proposed by using the same framework. The particular form of the SBP operators allows to take into account the numerical quadratures that approximate integrals in the numerical scheme compared to other techniques that require their exact evaluation to satisfy the entropy inequality \cite{jiang_shu94,hiltebrand_mishra_14}. The DGSEM thus provides a general framework for the design of ES schemes for nonlinear systems of conservation laws. An ES DGSEM for the discretization of general nonlinear hyperbolic systems in nonconservative form has been introduced in \cite{renac19} and applied to two-phase flow models \cite{renac19,coquel_etal_DGSEM_BN_21}. Numerical experiments in \cite{winters_etal_16,wintermeyer_etal_17,chen_shu_17,renac19,coquel_etal_DGSEM_BN_21} highlight the benefits on stability of the computations, though this not guarantees to preserve neither the entropy stability at the fully discrete level, nor robustness of the numerical solution. A fully discrete ES DGSEM has been proposed in \cite{friedrich_etal_19}, while ES and robust DGSEM have been proposed in \cite{despres98,renac17b} for the compressible Euler equations. 

A common way to design ES numerical fluxes in the sense of Tadmor \cite{tadmor87} is due to \cite{Roe_2006,ismail_roe_09} and consists in deriving EC numerical fluxes to which one adds upwind-type dissipation \cite{tadmor03,fjordholm_etal_12,carpenter_etal14,winters_etal_16,winters_etal_17,renac19,coquel_etal_DGSEM_BN_21,chandrashekar13,gouasmi_etal_20}. Nevertheless, this may cause difficulties to derive further properties of the scheme, such as preservation of invariant domains or bound preservation because the numerical flux may involve intricate operations of its arguments \cite{ismail_roe_09,chandrashekar13}. These properties are however important for the robustness and accuracy of the scheme. These fluxes are also non Lipschitz continuous which is required to apply existing results \cite{bouchut_04,hll_83,perthame_shu_96} that we will use in this work. Besides, for smooth solutions or large scale oscillations around discontinuities the DG approximation is known to become less sensitive to the numerical flux as the scheme accuracy increases \cite{qui_et-al06,renac15a}, but this is no longer the case when small scale flow features are involved and the numerical flux has a strong effect on their resolution \cite{moura_etal_17,chapelier_etal_14}. One has therefore to pay a lot of attention in the design of the numerical fluxes at interfaces. In the present work, we first design an ES and robust (i.e., that preserves positivity of the partial densities and internal energy) three-point scheme and then use it as a building block to define an ES, robust and accurate DGSEM. The multicomponent Euler system is therefore required to possess a convex entropy, which prevents the use of some models in the literature \cite{abgrall_96,karni_mutlicomp_96,ALLAIRE2002577}. \orange{The model is in conservative form and will be discretized by using a conservative scheme and will thus be prone to spurious oscillations at moving material interfaces separating phases with different thermodynamic properties. This drawback can be detrimental to applications involving fluids with highly nonlinear equations of state in which the coefficients depend on the thermodynamic variables as, for instance, in shock driven combustion problems, dynamics of explosive detonation products, etc. Let us stress however that for some non-genuinely multi-dimensional methods, spurious oscillations may appear even in non-conservative methods when the material interface is not aligned with the mesh making this property not useful in practice. Likewise, many research and industrial codes consider discretely conservative schemes and yet constitute relevant tools for the analysis: see, e.g., the codes THYC from EDF \cite{code_THYC_EDF} and FLICA from CEA for the simulation of water-vapor flows in nuclear power plants (see \cite[Sec.~2]{Ambroso2009738} and references therein), or the CHARME solver in the CEDRE software from ONERA for energetics and propulsion applications \cite{code_CEDRE}. Moreover, the present method will be shown} to satisfy other desirable properties (entropy stability, robustness, high-order accuracy) an to well reproduce the physical mechanisms of shock-material interface interaction problems on two-dimensional unstructured grids.


%
\subsection{Contributions of this work}

The thermodynamic properties of multicomponent flows depend on the mass fractions of the different phases which makes difficult the design of an ES and robust three-point scheme. An ES flux has been proposed in \cite{gouasmi_etal_20} that combines an EC flux with dissipation and was applied to high-order TecNO schemes \cite{fjordholm_etal_12}, but this prevent the derivation of a provably robust scheme. In \cite{dellacherie_03} a relaxation technique is applied to the multicomponent Euler system which allows the use of monocomponent ES schemes for each component. However, this technique does not hold for the separated fluid mixture in Eulerian coordinates under consideration in this work. Likewise, the use of simple wave solvers such as the HLL \cite{hll_83}, Roe \cite{Roe_1981}, Rusanov \cite{Rusanov1961}, or relaxation \cite{bouchut_04} schemes require the estimation from above of the maximum wave speeds in the Riemann problem for which fast estimates such as the two-rarefaction approximation \cite[Ch.~9]{toro_book}, \red{the iterative algorithm from \cite{guermond_popov_16} (see also the review \cite{TORO_bound_wave_speed_2020} and references therein)} do not exist to the best of our knowledge. Here we consider the energy relaxation technique introduced in \cite{coquel_perthame_98} for the approximation of the monocomponent compressible Euler equations with general equation of states. The method allows the design of ES and robust numerical schemes by using classical numerical fluxes for polytropic gases. The work in \cite{renac_etal_ES_noneq_21} extended this technique to the compressible multicomponent Euler equations for a gas mixture in thermal non-equilibrium and derived a general way to define an ES finite volume scheme from any scheme for the polytropic gas dynamics. We here follow this approach to derive an ES for our model from a scheme for the polytropic gas dynamics. However we will show that the entropy for the relaxation system is not strictly convex due to the closure for the fluid mixture in our model, which in turn prevents to derive a general framework for designing robust and ES fluxes. We here overcome this difficulty by using the approximate Riemann solver based on pressure relaxation from \cite{bouchut_04}. Relaxation schemes circumvent the difficulties in the treatment of nonlinearities associated to the equation of state by approximating the nonlinear system with a consistent linearly degenerate (LD) enlarged system with stiff relaxation source terms \cite{Jin_Xin_95,coquel_etal_01,chalons_coulombel08}. 


This ES numerical flux is used in the DGSEM at mesh interfaces, while the EC numerical flux from \cite{gouasmi_etal_20} is used within the discretization elements resulting in a semi-discrete ES scheme. We prove that the discrete scheme with explicit time integration exactly captures stationary contacts and stationary material interfaces at interpolation points. We further derive conditions on the time step to keep positivity of the cell-averaged partial densities and internal energy. This is achieved by extending the framework introduced in \cite{zhang2010positivity} for Cartesian meshes to unstructured meshes with straight-sided quadrangles. We indeed formulate the DGSEM for the cell-averaged solution as a convex combination of positive quantities under a CFL condition by using results from \cite{perthame_shu_96}.  We finally apply a posteriori limiters \cite{zhang2010positivity} that extend the positivity to nodal values within elements. 

The paper is organized as follows. \Cref{sec:HRM_model} presents the multicomponent compressible Euler system under consideration and some of its properties. We derive the EC two-point numerical flux and recall some properties of three-point schemes in \cref{sec:fluxes-FV-schemes}. We derive an ES relaxation-based three-point scheme in \cref{sec:relax_flux} that we then use as building block for the DGSEM introduced in \cref{sec:DG_discr}. We analyze the properties of the fully discrete DGSEM in \cref{sec:scheme_properties}. The results are assessed by numerical experiments in one and two space dimensions in \cref{sec:num_xp} and concluding remarks about this work are given in \cref{sec:conclusions}.

%
%
\section{Model problem}\label{sec:HRM_model}

%

Let $\Omega\subset\mathbb{R}^d$ be a bounded domain in $d$ space dimensions, we consider the IBVP described by the multicomponent compressible Euler system for a fluid mixture with $n_c$ components. This model is used for instance to simulate water-vapor flows where the vapor is in saturation state \cite{Ambroso2009738,Ambroso_etal_HEM_HRM_2007}. The problem reads
\begin{subequations}\label{eq:HRM_PDEs}
\begin{align}
 \partial_t{\bf u} + \nabla\cdot{\bf f}({\bf u}) &= 0, \quad \mbox{in }\Omega\times(0,\infty), \label{eq:HRM_PDEs-a} \\
 {\bf u}(\cdot,0) &= {\bf u}_{0}(\cdot),\quad\mbox{in }\Omega, \label{eq:HRM_PDEs-b}
\end{align}
\end{subequations}

\noindent with some boundary conditions to be prescribed on $\partial\Omega$ (see \cref{sec:num_xp}). Here
\begin{equation*}
 {\bf u} = \begin{pmatrix} \rho {\bf Y} \\ \rho \\ \rho{\bf v} \\ \rho E \end{pmatrix}, \quad
 {\bf f}({\bf u}) = \begin{pmatrix} \rho{\bf Y}{\bf v}^\top \\ \rho{\bf v}^\top \\ \rho{\bf v}{\bf v}^\top+\mathrm{p}{\bf I} \\(\rho E+\mathrm{p}){\bf v}^\top \end{pmatrix},
\end{equation*}

\noindent denote the conserved variables and the convective fluxes with ${\bf Y}=(Y_1,\dots,Y_{n_c-1})^\top$ the mass fractions of the $n_c-1$ first components with $\sum_{i=1}^{n_c}Y_i=1$; $\rho$, ${\bf v}$ in $\mathbb{R}^d$, and $E$ are the density, velocity vector, and total specific energy of the mixture, respectively. The mixture quantities are defined from quantities of the $n_c$ components through
\begin{equation}\label{eq:def-rho-rhoE-p}
 \sum_{i=1}^{n_c}\alpha_i=1, \quad \rho = \sum_{i=1}^{n_c}\alpha_i\rho_i=\rho\sum_{i=1}^{n_c}Y_i, \quad \rho E = \sum_{i=1}^{n_c}\alpha_i\rho_iE_i, \quad \mathrm{p} = \sum_{i=1}^{n_c}\alpha_i\mathrm{p}_i(\rho_i,e_i),
\end{equation}

\noindent where $E_{i}=e_i+\tfrac{{\bf v}\cdot{\bf v}}{2}$ and $\alpha_i=\rho Y_i/\rho_i$ is the void fraction of the i$th$ component. Note that we have $\rho E=\rho e+\rho \tfrac{{\bf v}\cdot{\bf v}}{2}$, where $\rho e=\sum_{i=1}^{n_c}\alpha_i\rho_ie_i$ denotes the internal energy of the mixture per unit volume. Equations \cref{eq:HRM_PDEs} are supplemented with polytropic ideal gas equations of states:
\begin{equation}\label{eq:EOS_SG}
 \mathrm{p}_i(\rho_i,e_i) = (\gamma_i-1)\rho_ie_i, \quad e_i = C_{v_i}\mathrm{T}_i, \quad 1\leq i\leq n_c,
\end{equation}

\noindent where $\gamma_i=C_{p_i}/C_{v_i}>1$ is the ratio of specific heats which are assumed to be positive constants of the model. The model assumes thermal and mechanical equilibria: 
\begin{equation}\label{eq:equ_temp_press}
 \mathrm{T}_i(\rho_i,e_i) =: \mathrm{T}({\bf Y},e), \quad
 \mathrm{p}_i(\rho_i,e_i) =: \mathrm{p}({\bf Y},\rho,e), \quad 1\leq i\leq n_c,
\end{equation}

\noindent thus leading to 
\begin{equation}\label{eq:mixture_eos}
 e({\bf Y},\mathrm{T}) = C_v({\bf Y})\mathrm{T}, \quad
 \mathrm{p}({\bf Y},\rho,e) = \big(\gamma({\bf Y})-1\big)\rho e = \rho r({\bf Y})\mathrm{T}({\bf Y},e), 
\end{equation}

\noindent with 
\begin{equation}\label{eq:mixture_r_Cv_Cp}
 r({\bf Y})=C_p({\bf Y})-C_v({\bf Y}), \quad C_p({\bf Y})=\sum_{i=1}^{n_c}Y_i C_{p_i}, \quad C_v({\bf Y})=\sum_{i=1}^{n_c} Y_i C_{v_i},
\end{equation}

\noindent and
\begin{equation}\label{eq:mixture_gamma}
 \gamma({\bf Y})=\frac{C_p({\bf Y})}{C_v({\bf Y})}.
\end{equation}

Note that due to \cref{eq:EOS_SG,eq:equ_temp_press} we have
\begin{equation}\label{eq:partial_densities}
 \rho r({\bf Y}) = \rho_ir_i, \quad r_i:=C_{p_i}-C_{v_i}=(\gamma_i-1)C_{v_i}, \quad 1\leq i\leq n_c,
\end{equation}

\noindent so the model allows to evaluate explicitly the partial densities and so the void fractions:
\begin{equation*}
  \alpha_i = \frac{\rho Y_i}{\rho_i}=\frac{r_i}{r({\bf Y})}Y_i \quad \forall 1\leq i\leq n_c,
\end{equation*}

System \cref{eq:HRM_PDEs-a} is hyperbolic in the direction ${\bf n}$ in $\mathbb{R}^d$ over the set of states
\begin{equation}\label{eq:HRM-set-of-states}
 \Omega^a=\{{\bf u}\in\mathbb{R}^{n_c+d+1}:\; 0\leq Y_{1\leq i\leq n_c}\leq 1, \rho>0, e=E-\tfrac{{\bf v}\cdot{\bf v}}{2}>0\},
\end{equation}

\noindent with eigenvalues $\lambda_1={\bf v}\cdot{\bf n}-c\leq \lambda_2=\dots=\lambda_{n_c+d}={\bf v}\cdot{\bf n}\leq\lambda_{n_c+d+1}={\bf v}\cdot{\bf n}+c$, where $\lambda_{1,n_c+d+1}$ are associated to genuinely nonlinear fields and $\lambda_{2\leq i\leq n_c+d}$ to LD fields. The sound speed reads
\begin{equation}\label{eq:sound_speed}
 c({\bf Y},e) = \sqrt{\gamma({\bf Y})\big(\gamma({\bf Y})-1)e}.
\end{equation}

From \cref{eq:mixture_r_Cv_Cp,eq:mixture_gamma}, we have
\begin{equation}\label{eq:gamma-bound}
 \gamma({\bf Y})\leq\gamma_{max} \quad \forall 0\leq Y_{1\leq i\leq n_c}\leq1, \quad \gamma_{max}:=\max_{1\leq i\leq n_c}\gamma_i.
\end{equation}

Admissible weak solutions to \cref{eq:HRM_PDEs} should satisfy the entropy inequality
\begin{equation}\label{eq:entropy_ineq_cont}
 \partial_t\eta({\bf u}) + \nabla\cdot{\bf q}({\bf u}) \leq 0
\end{equation}

\noindent for the entropy -- entropy flux  pair
\begin{equation}\label{eq:HRM_entropy_pair}
 \eta({\bf u}) =-\rho\mathrm{s}({\bf u}), \quad {\bf q}({\bf u}) =-\rho\mathrm{s}({\bf u}){\bf v}, \quad \mathrm{s}\equiv \sum_{i=1}^{n_c}Y_i\mathrm{s}_i,
\end{equation}

\noindent where the specific partial entropies are defined by the second law of thermodynamics and using \cref{eq:equ_temp_press}:
\begin{equation}\label{eq:2nd_principe}
 \mathrm{T}d\mathrm{s}_i = de_i - \frac{\mathrm{p}}{\rho_i^2}d\rho_i, \quad 1\leq i\leq n_c,
\end{equation}

\noindent and read 

\begin{equation}\label{eq:partial_entropy}
 \mathrm{s}_i(\rho_i,e_i)=C_{v_i}\ln\Big(\tfrac{e_i}{\rho_i^{\gamma_i-1}}\Big)+\mathrm{s}^\infty_{i} \overset{\cref{eq:EOS_SG}}{=} -C_{v_i}\ln\theta - r_i\ln\rho_i + C_{v_i}\ln C_{v_i}+\mathrm{s}^\infty_{i},
\end{equation}

\noindent with $\theta=\tfrac{1}{\mathrm{T}}$ and $\mathrm{s}^\infty_{i}$ an additive constant. Using \cref{eq:partial_densities} and $\T=\tfrac{e_i}{C_{v_i}}=\tfrac{e}{C_{v}({\bf Y})}$, the mixture entropy in \cref{eq:HRM_entropy_pair} becomes

\begin{subequations}\label{eq:mixture_entropy}
\begin{align}
  \mathrm{s}({\bf Y},\tau,e) &= \sum_{i=1}^{n_c}Y_i\Big(C_{v_i}\ln\big(\tfrac{C_{v_i}}{C_v({\bf Y})}e\big) - r_i\ln\big(\tfrac{r({\bf Y})}{r_i}\rho\big) +\mathrm{s}^\infty_{i}\Big) = r({\bf Y})\ln\tau + C_v({\bf Y})\ln e+K({\bf Y}),  \label{eq:mixture_entropy-a}\\
  K({\bf Y}) &= \sum_{i=1}^{n_c}Y_i\Big(C_{v_i}\ln\big(\tfrac{C_{v_i}}{C_v({\bf Y})}\big(\tfrac{r_i}{r({\bf Y})}\big)^{\gamma_i-1}\big)+\mathrm{s}^\infty_{i}\Big),
\end{align}
\end{subequations}

\noindent where $\tau=\tfrac{1}{\rho}$ denotes the covolume of the mixture.

Using \cref{eq:HRM_entropy_pair}, the differential forms \cref{eq:2nd_principe}, and $\theta\big(e_i-\tfrac{\p}{\rho_i}\big)=C_{p_i}$, we get

\begin{align*}
 d(\rho\s) &= \sum_{i=1}^{n_c} \rho_i\theta\Big(de_i-\tfrac{\p}{\rho_i^2}d\rho_i\Big) + \s_id\rho_i = \theta d(\rho e) + \sum_{i=1}^{n_c} (\s_i-C_{p_i})d\rho_i \\
	&= \theta\Big(d(\rho E) - {\bf v}\cdot d(\rho{\bf v}) + \tfrac{{\bf v}\cdot{\bf v}}{2}d\rho\Big) + \sum_{i=1}^{n_c-1} (\s_i-C_{p_i})d\rho_i + (\s_{n_c}-C_{p_{n_c}})\Big(d\rho-\sum_{i=1}^{n_c-1}d\rho_i\Big),
\end{align*}

\noindent so the entropy variables read
\begin{equation}\label{eq:entropy_var}
  \etab'({\bf u}) = \begin{pmatrix} \mathrm{s}_{n_c}-\mathrm{s}_1+C_{p_1}-C_{p_{n_c}} \\ \vdots \\ \mathrm{s}_{n_c}-\mathrm{s}_{n_c-1}+C_{p_{n_c-1}}-C_{p_{n_c}} \\ C_{p_{n_c}}-\mathrm{s}_{n_c}-\tfrac{{\bf v}\cdot{\bf v}}{2}\theta \\ \theta{\bf v} \\ -\theta \end{pmatrix},
\end{equation}

\noindent and the entropy potential is easily obtained:

\begin{align}
 \psib({\bf u}) &:= {\bf f}({\bf u})^\top\etab'({\bf u})-{\bf q}({\bf u}) \nonumber \\
 &= \sum_{i=1}^{n_c-1} (\s_{n_c}-\s_i-C_{p_{n_c}}+C_{p_i})\rho_i{\bf v} + \Big(C_{p_{n_c}}-\s_{n_c}-\tfrac{{\bf v}\cdot{\bf v}}{2}\theta\Big)\rho{\bf v}+ \theta(\rho{\bf v}{\bf v}^\top+\p{\bf I}){\bf v}-\theta(\rho E+\p){\bf v} - \rho\s({\bf u}){\bf v} \nonumber\\
 &= \sum_{i=1}^{n_c}C_{p_i}\rho_i{\bf v} - \rho e\theta{\bf v} = r({\bf Y})\rho{\bf v}. \label{eq:def-potential}
\end{align}

%
%
\section{Two-point numerical fluxes and associated finite volume schemes}\label{sec:fluxes-FV-schemes}

%
%
\subsection{Entropy conservative and entropy stable numerical fluxes}\label{sec:EC-ES-fluxes}

In the following, we design numerical fluxes for the space discretization of \cref{eq:HRM_PDEs}. We adopt the usual terminology from \cite{tadmor87} and denote by {\it entropy conservative} for the pair $(\eta,{\bf q})$ in \cref{eq:entropy_ineq_cont}, a numerical flux ${\bf h}_{ec}$ satisfying
\begin{equation}\label{eq:entropy_conserv_flux}
 \du \etab'({\bf u}) \df \cdot {\bf h}_{ec}({\bf u}^-,{\bf u}^+,{\bf n})  = \du \psib({\bf u})\cdot{\bf n} \df \quad \forall {\bf u}^\pm\in\Omega^a,
\end{equation}

\noindent where $\du a \df = a^+-a^-$ and ${\bf n}$ in $\mathbb{R}^d$ is a unit vector. The flux will also be required to be symmetric in the sense \cite{fisher_carpenter_13,chen_shu_17}:
\begin{equation}\label{eq:symmetric_flux}
 {\bf h}_{ec}({\bf u}^-,{\bf u}^+,{\bf n})={\bf h}_{ec}({\bf u}^+,{\bf u}^-,{\bf n}) \quad \forall {\bf u}^\pm\in\Omega^a.
\end{equation}

Then, a numerical flux ${\bf h}$ will be {\it entropy stable} when
\begin{equation}\label{eq:entropy_stable_flux}
 \du \etab'({\bf u}) \df \cdot {\bf h}({\bf u}^-,{\bf u}^+,{\bf n}) \leq \du \psib({\bf u})\cdot{\bf n} \df \quad \forall {\bf u}^\pm\in\Omega^a.
\end{equation}

Both numerical fluxes ${\bf h}$ and ${\bf h}_{ec}$ are assumed to be consistent:
\begin{equation}\label{eq:consistent_flux}
 {\bf h}_{ec}({\bf u},{\bf u},{\bf n}) = {\bf h}({\bf u},{\bf u},{\bf n}) = {\bf f}({\bf u})\cdot{\bf n} \quad \forall {\bf u}\in\Omega^a,
\end{equation}

\noindent and conservative:
\begin{equation}\label{eq:conservative_flux}
   {\bf h}_{ec}({\bf u}^-,{\bf u}^+,{\bf n}) =-{\bf h}_{ec}({\bf u}^+,{\bf u}^-,-{\bf n}), \quad {\bf h}({\bf u}^-,{\bf u}^+,{\bf n}) =-{\bf h}({\bf u}^+,{\bf u}^-,-{\bf n}) \quad \forall {\bf u}^\pm\in\Omega^a,
\end{equation}

\noindent and the ES numerical flux ${\bf h}$ is further assumed to be Lipschitz continuous in the first two arguments. In this work we use numerical fluxes that satisfy these properties.

An EC flux for the multicomponent Euler equations has been proposed in \cite{gouasmi_etal_20} and we apply their method to derive an EC flux in \cref{sec:appendix_EC_flux} for the sake of comparison with our approach. Below we propose another EC flux that takes into account the particular choice of densities in \cref{eq:HRM_PDEs}. Indeed the choice of the numerical flux is not unique and depends on the choice of variables we use to express the entropy pair and entropy potential \cite{chandrashekar13,ranocha_18}.

\begin{lemma}\label{th:EC_flux}
The following numerical flux is symmetric \cref{eq:symmetric_flux}, consistent \cref{eq:consistent_flux}, and EC \cref{eq:entropy_conserv_flux} for the HRM model \cref{eq:HRM_PDEs} and pair $(\eta,{\bf q})$ in \cref{eq:entropy_ineq_cont}:
\begin{equation}\label{eq:EC_flux}
 {\bf h}_{ec}({\bf u}^-,{\bf u}^+,{\bf n}) = \begin{pmatrix} h_{\rho Y_{1} }({\bf u}^-,{\bf u}^+,{\bf n}) \\ \vdots \\ h_{\rho Y_{n_c-1} }({\bf u}^-,{\bf u}^+,{\bf n}) \\ h_{\rho}({\bf u}^-,{\bf u}^+,{\bf n}) \\ h_{\rho}({\bf u}^-,{\bf u}^+,{\bf n})\ol{{\bf v}} + \tfrac{\ol{\mathrm{p}\theta}}{\ol{\theta}}{\bf n} \\  \displaystyle\sum_{i=1}^{n_c-1}\tfrac{C_{v_i}-C_{v_{n_c}}}{\widehat{\theta}}h_{\rho Y_i}({\bf u}^-,{\bf u}^+,{\bf n}) + \Big(\tfrac{C_{v_{n_c}}}{\widehat{\theta}}+\tfrac{{\bf v}^{-}\cdot{\bf v}^{+}}{2}\Big)h_{\rho}({\bf u}^-,{\bf u}^+,{\bf n}) + \tfrac{\ol{\mathrm{p}\theta}}{\ol{\theta}}\ol{{\bf v}}\cdot{\bf n} \end{pmatrix}, \; \begin{aligned}  &h_{\rho Y_i}({\bf u}^-,{\bf u}^+,{\bf n}) = \tfrac{r_{n_c}(\widehat{\rho_{n_c}}-\widehat{\rho})}{r(\ol{{\bf Y}})-r_{n_c}}\ol{Y_i}\ol{{\bf v}}\cdot{\bf n}, \\ & h_{\rho}(\cdot,\cdot,\cdot) = \blue{\widehat{\rho}\;\ol{{\bf v}}}\cdot{\bf n},\end{aligned}
\end{equation}

\noindent where $\widehat{a}=\tfrac{\du a\df}{\du \ln a\df}$ denotes the logarithmic mean \cite{ismail_roe_09}, $\ol{a}=\tfrac{a^++a^-}{2}$ is the average operator, and $\theta=\tfrac{1}{\T}$.

\end{lemma}

\begin{proof}
Symmetry follows from the symmetry of the logarithmic mean and average operator. Then, observe that $h_{\rho Y_i}({\bf u},{\bf u},{\bf n}) = \tfrac{r_{n_c}(\rho_{n_c}-\rho)}{r({\bf Y})-r_{n_c}}Y_i{\bf v}\cdot{\bf n} \overset{\cref{eq:partial_densities}}{=} \tfrac{r({\bf Y})\rho-r_{n_c}\rho}{r({\bf Y})-r_{n_c}}Y_i{\bf v}\cdot{\bf n} = \rho Y_i {\bf v}\cdot{\bf n}$. Likewise, $h_{\rho E}({\bf u},{\bf u},{\bf n}) = \big(\sum_{i=1}^{n_c}C_{v_i}\T\rho Y_i+\tfrac{1}{2}\rho{\bf v}\cdot{\bf v}+\p\big){\bf v}\cdot{\bf n}=(\rho E+\p){\bf v}\cdot{\bf n}$ thus consistency follows.

From \cref{eq:partial_entropy} we expand $\du\mathrm{s}_i\df = -C_{v_i}\du\ln\theta\df - r_i\du\ln\rho_i\df$ and
\begin{equation}\label{eq:expand_psi}
 \du\psib({\bf u})\cdot{\bf n}\df \overset{\cref{eq:def-potential}}{=} \du r({\bf Y})\rho{\bf v}\cdot{\bf n}\df \overset{\cref{eq:mixture_eos}}{=} \big(\ol{\mathrm{p}\theta}\du {\bf v}\df+\du\rho r({\bf Y})\df\ol{\bf v}\big)\cdot{\bf n} \overset{\cref{eq:partial_densities}}{=} \ol{\mathrm{p}\theta}\du {\bf v}\df\cdot{\bf n} + r_{n_c}\du\rho_{n_c}\df\ol{\bf v}\cdot{\bf n},
\end{equation}

\noindent and using short notations for the flux components in \cref{eq:EC_flux} together with the observation that $\du\ln\rho_i\df = \du\ln\rho_{n_c}+\ln\tfrac{r_{n_c}}{r_i}\df = \du\ln\rho_{n_c}\df$ from \cref{eq:partial_densities}, we get
\begin{align*}
 \du \etab'({\bf u}) \df \cdot {\bf h}_{ec}({\bf u}^-,{\bf u}^+,{\bf n}) - \du\psib({\bf u})\cdot{\bf n}\df
 &\overset{\cref{eq:entropy_var}}{\underset{\cref{eq:expand_psi}}{=}} \sum_{i=1}^{n_c-1}\du (C_{v_i}-C_{v_{n_c}})\ln\theta+r_i\ln\rho_i-r_{n_c}\ln\rho_{n_c}\df h_{\rho Y_i} \\
 &+ \du C_{v_{n_c}}\ln\theta+r_{n_c}\ln\rho_{n_c}-\tfrac{{\bf v}\cdot{\bf v}}{2}\theta\df h_\rho + \du\theta{\bf v}\df\cdot{\bf h}_{\rho {\bf v}}-\du\theta\df h_{\rho E} \\
 &- \ol{\mathrm{p}\theta}\du {\bf v}\df\cdot{\bf n} - r_{n_c}\du\rho_{n_c}\df\ol{\bf v}\cdot{\bf n}\\
 &= \blue{\du\theta\df} \Big( \sum_{i=1}^{n_c-1} \tfrac{C_{v_i}-C_{v_{n_c}}}{\widehat\theta}h_{\rho Y_i} + \big(\tfrac{C_{v_{n_c}}}{\widehat\theta}-\ol{\tfrac{{\bf v}\cdot{\bf v}}{2}}\big)h_\rho + \ol{\bf v}\cdot{\bf h}_{\rho {\bf v}} - h_{\rho E} \Big) \\
 &+ \du\ln\rho_{n_c}\df \Big( \sum_{i=1}^{n_c-1} (r_i-r_{n_c})h_{\rho Y_i} + r_{n_c}h_\rho - r_{n_c}\widehat{\rho_{n_c}}\ol{\bf v}\cdot{\bf n} \Big) + \ol{\theta}\du{\bf v}\df\cdot\Big({\bf h}_{\rho {\bf v}}-h_\rho\ol{\bf v}-\tfrac{\ol{\p\theta}}{\ol{\theta}}{\bf n}\Big). 
\end{align*}

\blue{
Now observe that from \cref{eq:mixture_r_Cv_Cp} we have $r(\ol{\bf Y})=r_{n_c}+\sum_{i=1}^{n_c-1}\ol{Y_i}(r_i-r_{n_c})$, so we obtain 
\begin{equation}\label{eq:expand_r}
 \sum_{i=1}^{n_c-1} (r_i-r_{n_c})h_{\rho Y_i} = \sum_{i=1}^{n_c-1} (r_i-r_{n_c})\tfrac{r_{n_c}(\widehat{\rho_{n_c}}-\widehat{\rho})}{r(\ol{{\bf Y}})-r_{n_c}}\ol{Y_i}\ol{{\bf v}}\cdot{\bf n} = r_{n_c}(\widehat{\rho_{n_c}}-\widehat{\rho})\ol{{\bf v}}\cdot{\bf n} \sum_{i=1}^{n_c-1} \tfrac{(r_i-r_{n_c})\ol{Y_i}}{r(\ol{{\bf Y}})-r_{n_c}} = r_{n_c}(\widehat{\rho_{n_c}}-\widehat{\rho})\ol{{\bf v}}\cdot{\bf n},
\end{equation}
}

\noindent \blue{and using \cref{eq:EC_flux} and then this result we finally get}
\blue{
\begin{equation*}
 \du \etab'({\bf u}) \df \cdot {\bf h}_{ec}({\bf u}^-,{\bf u}^+,{\bf n}) - \du\psib({\bf u})\cdot{\bf n}\df \overset{\cref{eq:EC_flux}}{=} \du\theta\df \Big( -\ol{\tfrac{{\bf v}\cdot{\bf v}}{2}} - \tfrac{{\bf v}^-\cdot{\bf v}^+}{2} + \ol{\bf v}^2\Big)h_\rho + \du\ln\rho_{n_c}\df \Big( \sum_{i=1}^{n_c-1} (r_i-r_{n_c})h_{\rho Y_i} + r_{n_c}(\widehat{\rho}-\widehat{\rho_{n_c}})\ol{\bf v}\cdot{\bf n} \Big)  \overset{\cref{eq:expand_r}}{=} 0,
\end{equation*}
}

\noindent which ends the proof. \qed
\end{proof}


\blue{
\begin{remark}
The EC flux in \cref{eq:EC_flux} requires only three calls of the logarithmic mean for evaluating $\widehat{\rho_{n_c}}$, $\widehat{\rho}$, and $\widehat{\theta}$, compared to $n_c+1$ calls in \cref{eq:EC_flux_gouasmi} derived in \cref{sec:appendix_EC_flux}. It is thus cheaper for $n_c>2$ as this evaluation is computationally expensive. Moreover, as the logarithmic mean requires positive arguments to avoid floating point exceptions, \cref{eq:EC_flux} is also less sensitive to robustness issues.
\end{remark}
}

\blue{
\begin{remark}
 Due to the particular form of the convective terms in the momentum equations, the numerical flux \cref{eq:EC_flux} is formally kinetic energy preserving \cite{chandrashekar13} in the sense of \cite{jameson_KEP_08}: when ignoring boundary conditions, the global kinetic energy budget is only affected by the pressure work, not the transport terms. This property may however fail at the discrete level and numerical results show that the discretization of the pressure plays an important role for the discrete kinetic energy preservation \cite{winters_etal_16}.
\end{remark}
}

\subsection{Entropy stable and robust finite volume schemes}\label{sec:FV-schemes}

We first consider three-point numerical schemes of the form
\begin{align}
  {\bf U}_j^{n+1} &= {\bf U}_j^{n} - \tfrac{\Delta t}{\Delta x}\big({\bf h}({\bf U}_{j}^{n},{\bf U}_{j+1}^{n},{\bf n})-{\bf h}({\bf U}_{j-1}^{n},{\bf U}_{j}^{n},{\bf n})\big), \quad 0<\tfrac{\Delta t}{\Delta x}\max_{j\in\mathbb{Z}}|\lambda({\bf U}_{j}^{n})|\leq\tfrac{1}{2},  \label{eq:3pt-scheme-a} 
\end{align}

\noindent for the discretization of \cref{eq:HRM_PDEs-a} in one space dimension, where ${\bf h}$ is assumed to be consistent, conservative, and Lipschitz continuous. Here ${\bf U}_{j}^{n}$ approximates the averaged solution in the $j$-th cell at time $t^{(n)}$, $\Delta t$ and $\Delta x$ are the time and space steps, and $|\lambda(\cdot)|$ corresponds to the maximum absolute value of the wave speeds.  The scheme \cref{eq:3pt-scheme-a} is said to be {\it entropy stable} for the pair $(\eta,{\bf q})$ in \cref{eq:entropy_ineq_cont} if it satisfies the inequality
\begin{equation}\label{eq:3pt-scheme-ineq}
 \eta({\bf U}_j^{n+1})-\eta({\bf U}_j^{n}) + \tfrac{\Delta t}{\Delta x}\big(Q({\bf U}_{j}^{n},{\bf U}_{j+1}^{n},{\bf n})-Q({\bf U}_{j-1}^{n},{\bf U}_{j}^{n},{\bf n})\big) \leq 0, 
\end{equation}

\noindent with some consistent entropy numerical flux $Q({\bf u},{\bf u},{\bf n}) = {\bf q}({\bf u})\cdot{\bf n}$. 

Such schemes use necessarily ES numerical fluxes \cite[Lemma~2.8]{bouchut_04}. In this work, we found more convenient to prove that \cref{eq:3pt-scheme-a} satisfies \cref{eq:3pt-scheme-ineq} than to prove \cref{eq:entropy_stable_flux} for the numerical flux in \cref{eq:3pt-scheme-a}.

%

Likewise, the scheme \cref{eq:3pt-scheme-a} will be said to be {\it robust} or {\it positive} if the solution remains in the set of states \cref{eq:HRM-set-of-states}: ${\bf U}_{j\in\mathbb{Z}}^{n}$ in $\Omega^a$ implies ${\bf U}_{j\in\mathbb{Z}}^{n+1}$ in $\Omega^a$. By extension the associated numerical flux will be also described as {\it robust} or {\it positive}.

We now consider a numerical scheme for quadrilateral meshes $X_h\subset\mathbb{R}^2$ using the two-point numerical flux in \cref{eq:3pt-scheme-a}:
\begin{equation}\label{eq:2D-FV-scheme}
 {\bf U}_\kappa^{n+1} = {\bf U}_\kappa^{n} - \tfrac{\Delta t}{|\kappa|}\sum_{e\in\partial\kappa} |e|{\bf h}({\bf U}_\kappa^{n},{\bf U}_{\kappa_e^+}^{n},{\bf n}_e) \quad \forall \kappa\in X_h, n\geq 0,
\end{equation}

\noindent where ${\bf n}_e$ is the unit outward normal vector on the edge $e$ in $\partial\kappa$, and $\kappa_e^+$ the neighboring cell sharing the interface $e$ (see \cref{fig:quad-with-subtriangles}). Each element is shape-regular: the ratio of the radius of the largest inscribed ball to the diameter is bounded by below by a positive constant independent of the mesh. We will also use the next result which is an extension to quadrilaterals of results from \cite{perthame_shu_96}. We reproduce the proof in \cref{sec:proof_positive-2D-FVO2-scheme} for the sake of completeness as the original proof in \cite{perthame_shu_96} considered triangular meshes.

\begin{lemma}(Perthame \& Shu \cite[Th.~4]{perthame_shu_96})\label{th:positive-2D-FVO2-scheme}
 Let a three-point numerical scheme of the form \cref{eq:3pt-scheme-a} with a consistent \cref{eq:consistent_flux}, conservative \cref{eq:conservative_flux}, and Lipschitz continuous numerical flux ${\bf h}(\cdot,\cdot,\cdot)$ for the discretization of \cref{eq:HRM_PDEs-a} that satisfies positivity of the solution, ${\bf U}_{j\in\mathbb{Z}}^{n\geq0}$ in $\Omega^a$, under the CFL condition in \cref{eq:3pt-scheme-a}. Then, the scheme
\begin{equation}\label{eq:2D-FVO2-scheme}
 {\bf U}_\kappa^{n+1} = {\bf U}_\kappa^{n} - \tfrac{\Delta t}{|\kappa|}\sum_{e\in\partial\kappa} |e|{\bf h}({\bf U}_{\kappa_e^-}^{n},{\bf U}_{\kappa_e^+}^{n},{\bf n}_e), \quad \sum_{e\in\partial\kappa} \tfrac{|e|}{|\partial\kappa|}{\bf U}_{\kappa_e^-}^{n} = {\bf U}_\kappa^{n}, \quad |\partial\kappa|:=\sum_{e\in\partial\kappa}|e|, 
\end{equation}

\noindent on quadrilateral meshes is also robust, ${\bf U}_{\kappa\in X_h}^{n\geq0}$ in $\Omega^a$, under the condition
\begin{equation}\label{eq:CFL-positive-2D-FVO2-scheme}
  \Delta t \max_{\kappa\in X_h}\frac{|\partial\kappa|}{|\kappa|}\max_{e\in\partial\kappa}\frac{|\partial\kappa_e|}{|e|}\max_{f\in\partial\kappa}|\lambda({\bf U}_{\kappa_f^\pm}^{n})| \leq \frac{1}{2}, 
\end{equation}

\noindent with $\kappa=\cup_{e\in\partial\kappa}\kappa_e$ divided into sub-triangles as in \cref{fig:quad-with-subtriangles}.
\end{lemma}

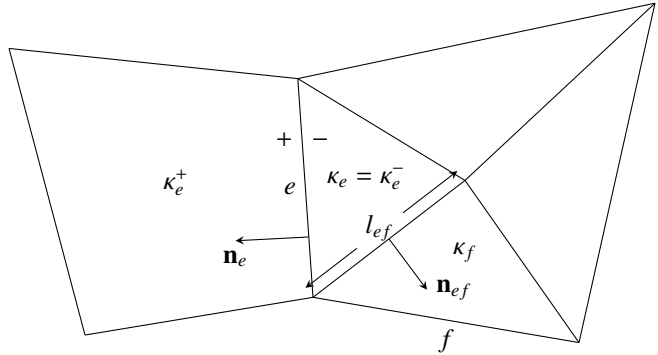
\begin{figure}[ht]
\begin{center}
\begin{tikzpicture}[scale=1.0]
\draw (1.2,1.625) node {$\kappa_e^+$};
\draw (3.7,1.7)   node {$\kappa_e=\kappa_e^-$};
\draw (5,0.7)     node {$\kappa_f$};
\draw (2.9,1.55)  node[left]  {$e$};
\draw (4.75,-0.2) node[below] {$f$};
\draw (2.86,2.20) node[left]  {$+$};
\draw (2.86,2.20) node[right] {$-$};
\draw [>=stealth,<->] (2.9,0.225) -- (4.9,1.775);
\draw (3.875,1.025) node[fill=white] {$l_{ef}$};
\draw [>=stealth,->] (2.94,0.90) -- (1.98,0.85) ;
\draw (1.98,0.85) node[below] {${\bf n}_e$};
\draw [>=stealth,->] (4,0.875) -- (4.5,0.2) ;
\draw (4.5,0.2) node[right] {${\bf n}_{ef}$};
\draw [>=stealth,-] (0.,-0.4) -- (3.0,0.1) ;
\draw [>=stealth,-] (3.,0.1) -- (2.8,3.) ;
\draw [>=stealth,-] (2.8,3.) -- (-1.,3.4) ;
\draw [>=stealth,-] (-1.,3.4) -- (0.,-0.4) ;
\draw [>=stealth,-] (3.,0.1) -- (6.5,-0.5) ;
\draw [>=stealth,-] (6.5,-0.5) -- (7.5,4.) ;
\draw [>=stealth,-] (7.5,4.) -- (2.8,3.) ;
\draw [>=stealth,-] (6.5,-0.5) -- (5.,1.65) ;
\draw [>=stealth,-] (3.0,0.1) -- (5.,1.65) ;
\draw [>=stealth,-] (2.8,3.) -- (5.,1.65) ;
\draw [>=stealth,-] (7.5,4.) -- (5.,1.65) ;
\end{tikzpicture}
\caption{Quadrilateral element $\kappa$ (right) divided into sub-triangles $\kappa=\cup_{e\in\partial\kappa}\kappa_e$ and neighboring element $\kappa_e^+$ sharing the edge $e\in\partial\kappa$. Notations: ${\bf n}_e$ is the unit outward normal vector to $\kappa$ on $e$; ${\bf n}_{ef}$ is the unit normal vector to $\ol{\kappa_e}\cap\ol{\kappa_f}$ oriented from $\kappa_e$ to $\kappa_f$; $l_{ef}=|\ol{\kappa_e}\cap\ol{\kappa_f}|$; and the exponents $^\pm$ correspond to outside and inside traces on $e$ (${\bf n}_{ef}$ and $l_{ef}$ will be used in \cref{sec:proof_positive-2D-FVO2-scheme}).}
\label{fig:quad-with-subtriangles}
\end{center}
\end{figure}

%
%
\section{Entropy stable and robust relaxation-based numerical flux}\label{sec:relax_flux}

We here derive an ES and robust numerical flux for \cref{eq:HRM_PDEs-a} which is given in \cref{sec:bouchut_flux}. \blue{To do so, we will need to combine two successive relaxation approximation steps}. In \cref{sec:nrj_relax_sys,sec:nrj_relax_sys_entropy}, we apply the energy relaxation approximation \cite{coquel_perthame_98} to \cref{eq:HRM_PDEs-a} by following the work in \cite{renac_etal_ES_noneq_21}. The results from \cite{renac_etal_ES_noneq_21} do not apply directly \blue{for two main reasons: (i) the systems differ by their closure law for the pressure and the definition of the mixture, (ii) the entropy for our relaxation system is here not strictly convex as will be shown in \cref{th:var-principle}. The latter reason makes difficult to derive an ES scheme for the energy relaxation system as required by the approach in \cite{coquel_perthame_98,renac_etal_ES_noneq_21}. As a consequence, we cannot apply directly the general framework from \cite{renac_etal_ES_noneq_21} to derive three-point schemes for our model.} We overcome this difficulty in \cref{sec:flux_h_from_H} by introducing a pressure-based relaxation approximation containing only LD fields \cite{bouchut_04} together with a minimization principle on the entropy (see \cref{sec:appendix_relax_sys_for_nrj_relax} for details on this derivation together with a description of the numerical flux). \blue{In the following, we recall the main steps and results of the energy relaxation approximation that will allow to derive the numerical flux for \cref{eq:HRM_PDEs-a}.}

%
\subsection{Energy relaxation system}\label{sec:nrj_relax_sys}

Following the energy relaxation method introduced in \cite{coquel_perthame_98,renac_etal_ES_noneq_21}, we consider the system
\begin{equation}\label{eq:relax-nrj-HRM-sys}
 \partial_t{\bf w}^\epsilon + \nabla\cdot{\bf g}({\bf w}^\epsilon) =-\frac{1}{\epsilon}\big({\bf w}^\epsilon-{\cal M}({\bf w}^\epsilon)\big),
\end{equation}

\noindent with $\epsilon>0$ the relaxation time scale, and also the homogeneous form
\begin{equation}\label{eq:relax-nrj-HRM-sys-homogeneous}
 \partial_t{\bf w} + \nabla\cdot{\bf g}({\bf w}) = 0,
\end{equation}

\noindent where
\begin{equation}\label{eq:relax-nrj-HRM-sys-details}
 {\bf w} = \begin{pmatrix}\rho {\bf Y} \\ \rho \\ \rho {\bf v} \\ \rho E_r \\ \rho e_s \end{pmatrix}, \quad
{\bf g}({\bf w}) = \begin{pmatrix} \rho {\bf Y}{\bf v}^\top \\ \rho {\bf v}^\top \\ \rho {\bf v}{\bf v}^\top+\mathrm{p}_r(\rho,e_r){\bf I} \\ \big(\rho E_r+\mathrm{p}_r(\rho,e_r)\big){\bf v}^\top \\ \rho e_s {\bf v}^\top \end{pmatrix}, \quad
 {\bf w}-{\cal M}({\bf w}) = \begin{pmatrix} 0 \\ 0 \\ 0 \\ 0 \\ -\rho\big(e_s-F({\bf Y},e_r)\big) \\ \rho\big(e_s-F({\bf Y},e_r)\big) \end{pmatrix},
\end{equation}

\noindent $F$ is defined in \cref{eq:maxwellian-equ-hrm}, $E_r$ denotes the total specific energy, $e_r=E_r-\tfrac{{\bf v}\cdot{\bf v}}{2}$ the internal specific energy, and
\begin{equation}\label{eq:EOS-relax}
 \mathrm{p}_r(\rho,e_r) = (\gamma-1)\rho e_r, 
\end{equation}

\noindent where $\gamma$ is defined by
\begin{equation}\label{eq:gamma-max}
 \gamma>\gamma_{max}, \quad \gamma_{max} \overset{\cref{eq:gamma-bound}}{=} \max_{0\leq Y_{1\leq i\leq n_c}\leq1}\gamma({\bf Y}) > 1,
\end{equation}

\noindent and constitutes the subcharacteristic condition for \cref{eq:relax-nrj-HRM-sys} to relax to an equilibrium as $\epsilon\downarrow0$ \cite{coquel_perthame_98}. The set of states for \cref{eq:relax-nrj-HRM-sys} is 
\begin{equation}\label{eq:relax-set-of-states}
 \Omega^r=\big\{{\bf w}\in\mathbb{R}^{d+n_c+2}:\; 0\leq Y_{1\leq i\leq n_c}\leq 1, \rho>0, e_r>0, e_s>0\big\}. 
\end{equation}

Let introduce the operators

\begin{subequations}\label{eq:def_L_P}
\begin{align}
 {\cal L}:\Omega^r\ni{\bf w} &\mapsto {\cal L}{\bf w} = \big(\rho {\bf Y}^\top, \rho, \rho{\bf v}^\top, \rho E_r+\rho e_s\big)^\top \;\in\; \Omega^a,\\
 {\cal P}:\Omega^a\ni{\bf u} &\mapsto {\cal P}({\bf u}) = \Big(\rho {\bf Y}^\top, \rho, \rho{\bf v}^\top, \rho E+\tfrac{\p}{\gamma-1}-\tfrac{\p}{\gamma({\bf Y})-1},\tfrac{\gamma-\gamma({\bf Y})}{\gamma({\bf Y})-1}\tfrac{\p}{\gamma-1}\Big)^\top \;\in\; \Omega^r,
\end{align}
\end{subequations}

\noindent with $\mathrm{p}=\mathrm{p}({\bf Y},\rho,e)$ defined from \cref{eq:mixture_eos} and $\tfrac{\p}{\gamma-1}$ evaluates $\rho e_r$ with equilibrium data. Then, in the limit ${\bf w}=\lim_{\epsilon\downarrow0}{\bf w}^\epsilon$, one formally recovers \cref{eq:HRM_PDEs-a} with
\begin{equation}\label{eq:maxwellian-equ}
 {\bf w} = {\cal M}({\bf w}), \quad {\bf u}={\cal L}{\bf w}, \quad {\bf f}({\bf u}) = {\cal L}{\bf g}({\cal P}({\bf u})),
\end{equation}

\noindent which corresponds to
\begin{equation}\label{eq:maxwellian-equ-hrm}
 E = E_r + e_s, \quad e = e_r + e_s, \quad e_s = F({\bf Y},e_r):=\frac{\gamma-\gamma({\bf Y})}{\gamma({\bf Y})-1}e_r,
\end{equation}

\noindent where $F$ is defined from the consistency relation on the pressure and \cref{eq:mixture_eos}: $\mathrm{p}\big({\bf Y},\rho,e_r+F({\bf Y},e_r)\big) = \mathrm{p}_r(\rho,e_r)$.

\subsection{Entropy for the energy relaxation system}\label{sec:nrj_relax_sys_entropy}

We now define an entropy $\rho\zeta({\bf w})$ for \cref{eq:relax-nrj-HRM-sys} \blue{which will be helpful in deriving an ES scheme for \cref{eq:HRM_PDEs}. The method in \cite{coquel_perthame_98,renac_etal_ES_noneq_21} requires a minimization principle and strict convexity of $\rho\zeta({\bf w})$. We show that only the former property is satisfied here and we will overcome the difficulty associated to the latter in \cref{sec:flux_h_from_H}}. To analyze the entropy, it is convenient to consider the entropy as a function of $\zeta({\bf Y},\tau,e_r,e_s)$ \cite[Ch. 2]{godlewski-raviart} with $\tau=\tfrac{1}{\rho}$ the covolume of the mixture.

Following \cite{renac_etal_ES_noneq_21}, we introduce the functions
\begin{subequations}\label{eq:functions-zeta-R-E}
\begin{align}
 \zeta({\bf Y},\tau,e_r,e_s) &= -\mathrm{s}\big({\bf Y},{\cal T}({\bf Y},\tau,e_r,e_s),{\cal E}({\bf Y},e_s)+e_s\big), \label{eq:functions-zeta-R-E-a} \\ 
{\cal E}({\bf Y},e_s) &= \tfrac{\gamma({\bf Y})-1}{\gamma-\gamma({\bf Y})}e_s, \quad
 {\cal T}({\bf Y},\tau,e_r,e_s) =  \tau\big(\tfrac{\gamma-\gamma({\bf Y})}{\gamma({\bf Y})-1}\tfrac{e_r}{e_s}\big)^{\frac{1}{\gamma-1}} \label{eq:functions-zeta-R-E-b}
\end{align}
\end{subequations}

\noindent where $\mathrm{s}$ is the mixture entropy \cref{eq:mixture_entropy} for \cref{eq:HRM_PDEs-a}, while the function ${\cal E}$ solves $e_s=F({\bf Y},e_r)$ for $e_r$ with $F$ defined in \cref{eq:maxwellian-equ-hrm}. Using \cref{eq:mixture_entropy-a,eq:maxwellian-equ-hrm,eq:functions-zeta-R-E}, we easily obtain

\begin{subequations}\label{eq:zeta-delta-zeta}
\begin{align}
 \zeta({\bf Y},\tau,e_r,e_s) &= -r({\bf Y})\ln\Big(\tau\Big(\tfrac{\gamma-\gamma({\bf Y})}{\gamma({\bf Y})-1}\tfrac{e_r}{e_s}\Big)^{\tfrac{1}{\gamma-1}}\Big) -C_v({\bf Y})\ln\Big(\tfrac{\gamma({\bf Y})-1}{\gamma-\gamma({\bf Y})}e_s+e_s\Big) - K({\bf Y}) \nonumber\\
 &= \blue{ -r({\bf Y})\ln\tau - \tfrac{r({\bf Y})}{\gamma-1}\ln\Big(\tfrac{\gamma-\gamma({\bf Y})}{\gamma({\bf Y})-1}\tfrac{e_r}{e_s}\Big) -C_v({\bf Y})\ln\Big(\tfrac{\gamma-1}{\gamma-\gamma({\bf Y})}e_s\Big) - K({\bf Y}) }\nonumber\\
 &= -r({\bf Y})\ln\tau - \tfrac{r({\bf Y})}{\gamma-1}\ln\Big(\tfrac{\gamma-\gamma({\bf Y})}{\gamma({\bf Y})-1}\tfrac{e_r}{e_s}\Big) - C_v({\bf Y})\ln(e_r+e_s) -C_v({\bf Y})\ln\Big(\tfrac{\gamma-1}{\gamma-\gamma({\bf Y})}\tfrac{e_s}{e_r+e_s}\Big) - K({\bf Y}) \label{eq:zeta-delta-zeta-a}\\
  &= -\mathrm{s}({\bf Y},\tau,e_r+e_s) + \varsigma({\bf Y},e_r,e_s), \label{eq:zeta-delta-zeta-b} \\
 \varsigma({\bf Y},e_r,e_s) &= C_v({\bf Y})\ln\Big(\big(\tfrac{\gamma({\bf Y})-1}{\gamma-\gamma({\bf Y})}\tfrac{e_s}{e_r}\big)^{\frac{\gamma({\bf Y})-1}{\gamma-1}}\Big) + C_v({\bf Y})\ln\big(\tfrac{\gamma-\gamma({\bf Y})}{\gamma-1}\tfrac{e_r+e_s}{e_s}\big). \label{eq:zeta-delta-zeta-c}
\end{align}
\end{subequations}


We now prove the following minimization principle which will guaranty that the entropy $\rho\zeta$ for \cref{eq:relax-nrj-HRM-sys} decreases to a unique global minimum which is solution to the multicomponent Euler system \cref{eq:HRM_PDEs-a}.

\begin{lemma}\label{th:var-principle}
 Under the assumption \cref{eq:gamma-max}, the function $\rho\zeta$ defined by \cref{eq:functions-zeta-R-E} is a (non strictly) convex entropy for \cref{eq:relax-nrj-HRM-sys} that satisfies the following minimization principle:
\begin{equation}\label{eq:var-principle}
 -\mathrm{s}({\bf Y},\tau,e) = \min_{e_r+e_s=e}\{\zeta({\bf Y},\tau,e_r,e_s):\; e_r>0, e_s>0\},
\end{equation}

\noindent and the minimum is reached at a unique global equilibrium which is solution to the multicomponent Euler system \cref{eq:HRM_PDEs-a}.
\end{lemma}

\begin{proof}
From \cref{eq:zeta-delta-zeta-a} and the definition of $K({\bf Y})$ in \cref{eq:mixture_entropy} we rewrite $\zeta$ as
\begin{align}
  \zeta({\bf Y},\tau,e_r,e_s) &= -r({\bf Y})\ln\tau - \tfrac{r({\bf Y})}{\gamma-1}\ln e_r - \tfrac{(\gamma-\gamma({\bf Y}))C_v({\bf Y})}{\gamma-1}\ln e_s  \nonumber\\
	&+ C_v({\bf Y})\ln\big((\gamma-\gamma({\bf Y}))C_v({\bf Y})\big) + r({\bf Y})\ln r({\bf Y}) + \tfrac{r({\bf Y})}{\gamma-1}\ln\tfrac{\gamma({\bf Y})-1}{\gamma-\gamma({\bf Y})} + l({\bf Y}), \label{eq:def_zeta_tau}
\end{align}

\noindent with $l({\bf Y})=\sum_{i=1}^{n_c}Y_i(C_{v_i}\ln C_{v_i}+r_i\ln r_i + \mathrm{s}_i^\infty)$ linear in ${\bf Y}$. To prove that $\rho\zeta({\bf w})$ is convex it is sufficient to prove that $\zeta({\bf Y},\tau,e_r,e_s)$ is convex \cite[Ch. 2]{godlewski-raviart}. Introducing the short notations $\partial_kr\equiv\partial_{Y_k}r({\bf Y})$, $\partial_kC_v\equiv\partial_{Y_k}C_v({\bf Y})$, and $\partial_k\gamma\equiv\partial_{Y_k}\gamma({\bf Y})$, the Hessian of $\zeta$ reads
\begin{equation}\label{eq:hessian_zeta}
  \mathbfcal{H}_\zeta({\bf Y},\tau,e_r,e_s) = 
  \begin{pmatrix} \big(\partial_{kl}^2\zeta\big)_{1\leq k,l < n_c} & \big(\tfrac{-\partial_kr}{\tau}\big)_{1\leq k < n_c} & \big(\tfrac{-\partial_kr}{(\gamma-1)e_r}\big)_{1\leq k < n_c} & \big(\tfrac{\partial_kr-(\gamma-1)\partial_kC_v}{(\gamma-1)e_s}\big)_{1\leq k < n_c} \\
    \big(\tfrac{-\partial_lr}{\tau}\big)_{1\leq l < n_c} & \tfrac{r({\bf Y})}{\tau^2} & 0 & 0 \\
    \big(\tfrac{-\partial_lr}{(\gamma-1)e_r}\big)_{1\leq l < n_c} & 0 & \tfrac{r({\bf Y})}{(\gamma-1)e_r^2} & 0 \\
    \big(\tfrac{\partial_lr-(\gamma-1)\partial_lC_v}{(\gamma-1)e_s}\big)_{1\leq l < n_c} & 0 & 0 & \tfrac{\gamma-\gamma({\bf Y})}{\gamma-1}\tfrac{C_v({\bf Y})}{e_s^2}
 \end{pmatrix}, 
\end{equation}

\noindent where
\begin{equation}\label{eq:dr_dCv_dgam}
  \partial_kr \overset{\cref{eq:mixture_r_Cv_Cp}}{\underset{\cref{eq:mixture_gamma}}{=}} C_v({\bf Y})\partial_k\gamma+(\gamma({\bf Y})-1)\partial_kC_v,
\end{equation}
\noindent and
\begin{align*}
\partial_{k}\zeta &= -\partial_kr\ln\tau - \frac{\partial_kr}{\gamma-1}\ln e_r -\frac{\gamma\partial_kC_v-\partial_kC_p}{\gamma-1}\ln e_s + \partial_kC_v\ln\big((\gamma-\gamma({\bf Y}))C_v({\bf Y})\big) + \frac{\gamma\partial_kC_v-\partial_kC_p}{\gamma-\gamma({\bf Y})} \nonumber\\
 &+ \partial_kr\big(1+\ln r({\bf Y})\big) + \frac{\partial_kr}{\gamma-1}\ln\frac{\gamma({\bf Y})-1}{\gamma-\gamma({\bf Y})} + \frac{r({\bf Y})}{\gamma-1}\Big(\frac{\partial_k\gamma}{\gamma({\bf Y})-1}+\frac{\partial_k\gamma}{\gamma-\gamma({\bf Y})}\Big) + \partial_kl({\bf Y}), \quad 1\leq k < n_c,
\end{align*}
\noindent gives
\begin{equation}\label{eq:d2ZetadYY}
\partial_{kl}^2\zeta = \frac{\partial_kr\partial_lr}{r({\bf Y})} + \frac{\partial_kC_v\partial_lC_v}{C_v({\bf Y})} + \frac{C_v({\bf Y})\partial_k\gamma\partial_l\gamma}{\big(\gamma-\gamma({\bf Y})\big)\big(\gamma({\bf Y})-1\big)}, \quad 1\leq k,l< n_c.
\end{equation}

We now prove that $\mathbfcal{H}_\zeta$ is symmetric positive semi-definite. Let ${\bf x}=(x_{1\leq i<n_c},x_\tau,x_r,x_s)^\top$ in $\mathbb{R}^{n_c+2}$ and use the notation $\sum\equiv\sum_{k=1}^{n_c-1}$, we get
\begin{align}
  \sum_{k,l=1}^{n_c-1} x_k\partial_{kl}^2\zeta x_l &\overset{\cref{eq:d2ZetadYY}}{=} \frac{(\sum x_k\partial_kr)^2}{r({\bf Y})} + \frac{(\sum x_k\partial_kC_v)^2}{C_v({\bf Y})} + \frac{C_v({\bf Y})(\sum x_k\partial_k\gamma)^2}{\big(\gamma-\gamma({\bf Y})\big)\big(\gamma({\bf Y})-1\big)} \nonumber \\
  &\overset{\cref{eq:dr_dCv_dgam}}{=}  \frac{\gamma(\sum x_k\partial_kr)^2}{(\gamma-1)r({\bf Y})} - \frac{\big(\sum x_kC_v({\bf Y})\partial_k\gamma+x_k(\gamma({\bf Y})-1)\partial_kC_v\big)^2}{(\gamma-1)r({\bf Y})} + \frac{(\sum x_k\partial_kC_v)^2}{C_v({\bf Y})} + \frac{C_v({\bf Y})(\sum x_k\partial_k\gamma)^2}{\big(\gamma-\gamma({\bf Y})\big)\big(\gamma({\bf Y})-1\big)} \nonumber \\
  &= \frac{\gamma}{\gamma-1}\frac{(\sum x_k\partial_kr)^2}{r({\bf Y})} + \frac{(\sum x_k(\gamma-\gamma({\bf Y}))\partial_kC_v-x_kC_v({\bf Y})\partial_k\gamma)^2}{\big(\gamma-\gamma({\bf Y})\big)(\gamma-1)C_v({\bf Y})} \nonumber \\
  &\overset{\cref{eq:dr_dCv_dgam}}{=} \frac{\gamma}{\gamma-1}\frac{(\sum x_k\partial_kr)^2}{r({\bf Y})} + \frac{(\sum x_k(\gamma-1)\partial_kC_v-x_k\partial_kr)^2}{\big(\gamma-\gamma({\bf Y})\big)(\gamma-1)C_v({\bf Y})}, \label{eq:SPD_d2ZetadYY}
\end{align}

\noindent so we obtain
\begin{align*}
  {\bf x}^\top\mathbfcal{H}_\zeta{\bf x} &\overset{\cref{eq:hessian_zeta}}{\underset{\cref{eq:SPD_d2ZetadYY}}{=}} \Big(1+\frac{1}{\gamma-1}\Big)\frac{(\sum x_k\partial_kr)^2}{r({\bf Y})} + \frac{(\sum x_k(\gamma-1)\partial_kC_v-x_k\partial_kr)^2}{\big(\gamma-\gamma({\bf Y})\big)(\gamma-1)C_v({\bf Y})} + r({\bf Y})\frac{x_\tau^2}{\tau^2} + \frac{r({\bf Y})}{\gamma-1}\frac{x_r^2}{e_r^2} \\
  & + \frac{\big(\gamma-\gamma({\bf Y})\big)C_v({\bf Y})}{\gamma-1}\frac{x_s^2}{e_s^2} -2\sum x_k\Big(\partial_kr\frac{x_\tau}{\tau} + \frac{\partial_kr}{\gamma-1}\frac{x_r}{e_r} + \frac{(\gamma-1)\partial_kC_v-\partial_kr}{\gamma-1}\frac{x_s}{e_s}\Big) \\
  &= \frac{\big(\sum x_k\partial_kr-r({\bf Y})\tfrac{x_\tau}{\tau}\big)^2}{r({\bf Y})} + \frac{\big(\sum x_k\partial_kr-r({\bf Y})\tfrac{x_r}{e_r}\big)^2}{(\gamma-1)r({\bf Y})} + \frac{\big(\sum x_k((\gamma-1)\partial_kC_v-\partial_kr)-(\gamma-\gamma({\bf Y}))C_v({\bf Y})\tfrac{x_s}{e_s}\big)^2}{(\gamma-\gamma({\bf Y}))(\gamma-1)C_v({\bf Y})}
\end{align*}

\noindent which is non-negative.


Finally, according to \cref{eq:var-principle}, we need to prove that, for all positive $e_r$ and $e_s$ and fixed ${\bf Y}$ and $\tau$, $\varsigma$, defined in \cref{eq:zeta-delta-zeta-c}, is non-negative and vanishes at equilibrium \cref{eq:maxwellian-equ-hrm} that constitutes a global minimum: $-\s\leq\zeta$. Let us rewrite $\varsigma$ as $C_v({\bf Y})\ln\big(f(\alpha,x)\big)$ with $f(\alpha,x)=\tfrac{(1-\alpha)(1+x)}{x}\big(\tfrac{\alpha x}{1-\alpha}\big)^\alpha$, $x=\tfrac{e_s}{e_r}>0$, and $\alpha=\tfrac{\gamma({\bf Y})-1}{\gamma-1}$ in $(0,1)$ from \cref{eq:gamma-max}. We have $\partial_xf(\alpha,x)=\tfrac{1-\alpha}{x^2}(\alpha x+\alpha-1)$, thus $\partial_xf(\alpha,x)<0$ for $0<x<x_{min}:=\tfrac{1-\alpha}{\alpha}$, $\partial_xf(\alpha,x)>0$ for $x>x_{min}$, and $\partial_xf(\alpha,x_{min})=0$. Since $f(\alpha,x_{min})=1$, $\varsigma$ vanishes at the global minimum $\alpha x_{min}=1-\alpha\Leftrightarrow \tfrac{\gamma({\bf Y})-1}{\gamma-1}\tfrac{e_s}{e_r}=1-\tfrac{\gamma({\bf Y})-1}{\gamma-1}$ which indeed corresponds to the equilibrium \cref{eq:maxwellian-equ-hrm}: $e_s=F({\bf Y},e_r)$. This defines the internal energy, $e=e_r+e_s=\tfrac{\gamma-1}{\gamma({\bf Y})-1}e_r$, in \cref{eq:HRM_PDEs} in a unique way so ${\bf u}$ is uniquely defined and the global minimum is unique in $\Omega^a$.\qed
\end{proof}


%
\subsection{Discrete energy relaxation}\label{sec:flux_h_from_H}

\subsubsection{\blue{From a scheme for the energy relaxation system}}


\blue{The derivation of a scheme for \cref{eq:HRM_PDEs-a} with the energy relaxation approximation \cite{coquel_perthame_98,renac_etal_ES_noneq_21} uses} a splitting of the hyperbolic and relaxation operators in \cref{eq:relax-nrj-HRM-sys}. In the first step we consider the following three-point scheme for the homogeneous relaxation system \cref{eq:relax-nrj-HRM-sys-homogeneous}:
\begin{equation}\label{eq:3pt-scheme-relax}
 {\bf W}_j^{n+1}-{\bf W}_j^{n} + \tfrac{\Delta t}{\Delta x}\big({\bf H}({\bf W}_{j}^{n},{\bf W}_{j+1}^{n},{\bf n})-{\bf H}({\bf W}_{j-1}^{n},{\bf W}_{j}^{n},{\bf n})\big) = 0, 
\end{equation}

\noindent with ${\bf H}({\bf w},{\bf w},{\bf n})={\bf g}({\bf w})\cdot{\bf n}$ \blue{and we need to prove that this scheme is ES, i.e.}, under some condition on $\Delta t$, we have

\begin{equation}\label{eq:3pt-scheme-equal-relax}
 \rho\zeta({\bf W}_j^{n+1}) - \rho\zeta({\bf W}_j^{n}) + \tfrac{\Delta t}{\Delta x}\big(Z({\bf W}_{j}^{n},{\bf W}_{j+1}^{n},{\bf n})-Z({\bf W}_{j-1}^{n},{\bf W}_{j}^{n},{\bf n})\big) \leq 0, 
\end{equation}

\noindent with $Z({\bf w},{\bf w},{\bf n})=\rho\zeta{\bf v}\cdot{\bf n}$. Since $\rho\zeta({\bf w})$ is not strictly convex from \cref{th:var-principle}, \cref{eq:3pt-scheme-equal-relax} is usually difficult to prove (for instance one cannot neither define ${\bf w}$ and ${\bf g}$ in \cref{eq:relax-nrj-HRM-sys-homogeneous} as functions of the entropy variables as in \blue{\cite{harten_lax_RCM_81,tadmor87,vila_EC_88,Johnson1990OnTC}, nor derive entropy dissipation estimates as in \cite{coquel_lefloch_MUSCL_96}, see also the review in \cite{hll_83}}). We overcome this difficulty by considering a pressure-based relaxation scheme adapted from \cite[Sec.~2.4]{bouchut_04} that satisfies \cref{eq:3pt-scheme-equal-relax}. For the sake of readability, we derive the numerical scheme in \cref{sec:appendix_relax_sys_for_nrj_relax}: the numerical flux in \cref{eq:3pt-scheme-relax} is defined in \cref{eq:relax_flux_for_nrj_sys}, while we derive \cref{eq:3pt-scheme-equal-relax} in \cref{eq:press-relax-entropy-equality,eq:press-relax-entropy-projection}, and define $Z$ in \cref{eq:3pt-scheme-equal-relax-details}.

\subsubsection{\blue{To a scheme for the multicomponent Euler system}}

We here derive the ES, robust, consistent, and Lipschitz continuous numerical flux for \cref{eq:HRM_PDEs-a} from the pressure-based relaxation numerical flux for the relaxation system \cref{eq:relax-nrj-HRM-sys} with similar properties introduced in \cref{sec:flux_h_from_H} and detailed in \cref{sec:appendix_relax_sys_for_nrj_relax}. \blue{From \cite{coquel_perthame_98,renac_etal_ES_noneq_21},} the numerical flux for \cref{eq:HRM_PDEs-a} reads
\begin{equation}\label{eq:flux_h_from_H}
 {\bf h}({\bf u}^-,{\bf u}^+,{\bf n})={\cal L}{\bf H}\big({\cal P}({\bf u}^-),{\cal P}({\bf u}^+),{\bf n}\big),
\end{equation}

\noindent where the operators ${\cal L}$ and ${\cal P}$ are defined in \cref{eq:def_L_P}. The explicit form of \cref{eq:flux_h_from_H} is given in \cref{sec:bouchut_flux}. \blue{The ${\cal L}$ operator in \cref{eq:flux_h_from_H} consists in adding up the $\rho E_r$ and $\rho e_s$ components of ${\bf H}$ to build the numerical flux for the total energy, $\rho E$, while the ${\cal P}$ operators consist in taking data at equilibrium, i.e., $\p_r=\p({\bf Y},\rho,e)$ and $e_s=\tfrac{\gamma-\gamma({\bf Y})}{\gamma-1}e$ from \cref{eq:maxwellian-equ-hrm}. This last operation is equivalent to applying time discrete instantaneous relaxation and constitutes the second step of the splitting of hyperbolic and relaxation operators \cite{coquel_perthame_98}.}

\blue{
Providing that \cref{eq:gamma-max} holds, the numerical scheme \cref{eq:3pt-scheme-a} and the numerical flux \cref{eq:flux_h_from_H} defined in \cref{eq:relax_flux} have the following properties for which points (iii) and (v) are direct consequences of \cref{th:var-principle} and \cref{eq:3pt-scheme-equal-relax} (see \cite[Th.~4.3]{renac_etal_ES_noneq_21} for details):
}
\blue{
\begin{enumerate}[label=(\roman*)]
 \item the flux \cref{eq:flux_h_from_H} is consistent \cref{eq:consistent_flux} by consistency of ${\bf H}$: we have ${\bf h}({\bf u},{\bf u},{\bf n})={\cal L}{\bf H}\big({\cal P}({\bf u}),{\cal P}({\bf u}),{\bf n}\big)={\cal L}{\bf g}({\cal P}({\bf u}))\cdot{\bf n}={\bf f}({\bf u})\cdot{\bf n}$ from \cref{eq:maxwellian-equ};
 \item the flux \cref{eq:flux_h_from_H} is Lipschitz continuous by composition of Lipschitz continuous functions; 
 \item the scheme \cref{eq:3pt-scheme-a} is ES in the sense \cref{eq:3pt-scheme-ineq} for the pair $(\eta,{\bf q})$ in \cref{eq:HRM_entropy_pair} with $Q({\bf u}^-,{\bf u}^+,{\bf n})=Z\big({\cal P}({\bf u}^-),{\cal P}({\bf u}^+),{\bf n}\big)$;
 \item the flux \cref{eq:flux_h_from_H} is ES in the sense \cref{eq:entropy_stable_flux} since \cref{eq:3pt-scheme-ineq} implies \cref{eq:entropy_stable_flux} \cite[Lemma~2.8]{bouchut_04};
  \item the scheme \cref{eq:3pt-scheme-a} is robust: ${\bf U}_{j\in\mathbb{Z}}^{n\geq0}\in\Omega^a$ providing that ${\bf U}_{j\in\mathbb{Z}}^{0}\in\Omega^a$. 
\end{enumerate}
}

\subsection{Two-point numerical flux for \cref{eq:HRM_PDEs-a}}\label{sec:bouchut_flux}

We here give details on the numerical flux \cref{eq:flux_h_from_H} for the multicomponent Euler equations \cref{eq:HRM_PDEs-a}. The flux follows from applying \cref{eq:flux_h_from_H} to the numerical flux \cref{eq:relax_flux_for_nrj_sys,eq:bouchut-relax-ARS-xt} which is shown in \cref{sec:appendix_press_relax_flux} to lead to an ES and robust scheme \cref{eq:3pt-scheme-relax}. We thus obtain an ES and robust scheme \cref{eq:3pt-scheme-a,eq:3pt-scheme-ineq} and an ES flux \cref{eq:entropy_stable_flux} for \cref{eq:HRM_PDEs-a} under the condition
\begin{equation}\label{eq:3-pt-CFL_cond}
 \frac{\Delta t}{\Delta x}\max_{j\in\mathbb{Z}}|\lambda({\bf U}_j^{n})| < \frac{1}{2}, \quad |\lambda({\bf u})| := |{\bf v}\cdot{\bf n}|+\tfrac{a}{\rho},
\end{equation}

\noindent where $a$ denotes the Lagrangian sound speed and will be defined in \cref{eq:wave-estimate}. The numerical flux reads
\begin{equation}\label{eq:relax_flux}
 {\bf h}({\bf u}^-,{\bf u}^+,{\bf n}) = {\bf f}\big(\bcW^{r}(0;{\bf u}^-,{\bf u}^+,{\bf n})\big)\cdot{\bf n},
\end{equation}

\noindent where the Riemann solver $\bcW^r(\cdot;{\bf u}_L,{\bf u}_R,{\bf n})$ is used to approximate the solution to \cref{eq:HRM_PDEs} in the direction ${\bf n}$ with initial data, ${\bf u}_0(x) = {\bf u}_L$ if $x:={\bf x}\cdot{\bf n}<0$ and ${\bf u}_0(x) = {\bf u}_R$  if $x>0$, and reads
\begin{equation}\label{eq:relax-ARS-xt}
 \bcW^{r}(\tfrac{x}{t};{\bf u}_L,{\bf u}_R,{\bf n}) = \left\{ \begin{array}{ll}  {\bf u}_L, & \tfrac{x}{t} < S_L, \\
 {\bf u}_L^\star, & S_L<\tfrac{x}{t}<u^\star,\\
 {\bf u}_R^\star, & u^\star<\tfrac{x}{t}<S_R,\\
 {\bf u}_R, & S_R < \tfrac{x}{t}, \end{array} \right.
\end{equation}

\noindent where ${\bf u}_L^\star=(\rho_L^\star {\bf Y}_L^\top,\rho_L^\star,\rho_L^\star {\bf v}_L^{\star\top},\rho_L^\star E_L^\star)^\top$, ${\bf u}_R^\star=(\rho_R^\star {\bf Y}_R^\top,\rho_R^\star,\rho_R^\star{\bf v}_R^{\star\top},\rho_R^\star E_R^\star)^\top$, and

\begin{subequations}\label{eq:sol-PR-relax}
\begin{align}
  {\bf v}_L^\star &= {\bf v}_L+(u^\star-u_L){\bf n}, \quad {\bf v}_R^\star = {\bf v}_R+(u^\star-u_R){\bf n},  \label{eq:sol-PR-relax-a} \\
  u^\star  &= \frac{a_Lu_L+a_Ru_R+\mathrm{p}_L-\mathrm{p}_R}{a_L+a_R}, \quad \mathrm{p}^\star = \frac{a_R\mathrm{p}_L+a_L\mathrm{p}_R+a_La_R(u_L-u_R)}{a_L+a_R},  \\
  \tau_L^\star &= \tau_L + \frac{u^\star-u_L}{a_L},\quad \tau_R^\star = \tau_R + \frac{u_R-u^\star}{a_R}, \\
  E_L^\star &= E_L - \frac{\mathrm{p}^\star u^\star - \mathrm{p}_Lu_L}{a_L},\quad E_R^\star = E_R - \frac{\mathrm{p}_Ru_R-\mathrm{p}^\star u^\star}{a_R},
\end{align}
\end{subequations}

\noindent where $u_X={\bf v}_X\cdot{\bf n}$ and $\mathrm{p}_X=\mathrm{p}({\bf Y}_X,\rho_X,e_X)$ defined by \cref{eq:mixture_eos} for $X=L,R$; \cref{eq:sol-PR-relax-a} corresponds to a decomposition into normal, $u_X$, and tangential, ${\bf v}_X-u_X{\bf n}$, components of the velocity vector.

The wave speeds in \cref{eq:relax-ARS-xt} are evaluated from $S_L=u_L-a_L/\rho_L$ and $S_R=u_R+a_R/\rho_R$ where the approximate Lagrangian sound speeds \cite{bouchut_04} are defined by
\begin{subequations}\label{eq:wave-estimate}
\begin{align}
 \left\{ \begin{array}{rcl} \tfrac{a_L}{\rho_L} &=& c_\gamma(\rho_L,\mathrm{p}_L) + \tfrac{\gamma+1}{2}\Big(\tfrac{\mathrm{p}_R-\mathrm{p}_L}{\rho_Rc_\gamma(\rho_R,\mathrm{p}_R)}+u_L-u_R\Big)^+ \\
\tfrac{a_R}{\rho_R} &=& c_\gamma(\rho_R,\mathrm{p}_R) + \tfrac{\gamma+1}{2}\Big(\tfrac{\mathrm{p}_L-\mathrm{p}_R}{a_L}+u_L-u_R\Big)^+
\end{array}\right., & \quad \mbox{if } \mathrm{p}_R \geq \mathrm{p}_L, \\
 \left\{ \begin{array}{rcl} \tfrac{a_R}{\rho_R} &=& c_\gamma(\rho_R,\mathrm{p}_R) + \tfrac{\gamma+1}{2}\Big(\tfrac{\mathrm{p}_L-\mathrm{p}_R}{\rho_Lc_\gamma(\rho_L,\mathrm{p}_L)}+u_L-u_R\Big)^+ \\
\tfrac{a_L}{\rho_L} &=& c_\gamma(\rho_L,\mathrm{p}_L) + \tfrac{\gamma+1}{2}\Big(\tfrac{\mathrm{p}_R-\mathrm{p}_L}{a_R}+u_L-u_R\Big)^+
\end{array}\right., & \quad \mbox{else,}  
\end{align}
\end{subequations}

\noindent where $(\cdot)^+=\max(\cdot,0)$ denotes the positive part and $c_\gamma(\rho,\mathrm{p})=\sqrt{\gamma\mathrm{p}/\rho}$ with $\gamma$ defined by \cref{eq:gamma-max}.

%
%
\section{DGSEM formulation}\label{sec:DG_discr}

The DG method consists in defining a semi-discrete weak formulation of problem \cref{eq:HRM_PDEs}. The domain is discretized with a shape-regular mesh $X_h\subset\mathbb{R}^d$ consisting of nonoverlapping and nonempty cells $\kappa$ and we assume that it forms a partition of $\Omega$. By ${\cal E}_h$ we define the set of interfaces in $X_h$. For the sake of clarity, we introduce the DGSEM in two space dimensions $d=2$, the extension (resp. restriction) to $d=3$ (resp. $d=1$) being straightforward. The present analysis is restricted to meshes with straight-sided cells and to infinite domains though bounded domains will be considered in \cref{sec:num_xp}.

\subsection{Numerical solution} We look for approximate solutions in the function space of discontinuous polynomials ${\cal V}_h^p=\{\phi\in L^2(X_h):\;\phi|_{\kappa}\circ{\bf x}_\kappa\in{\cal Q}^p(I^2)\; \forall\kappa\in X_h\}$, where ${\cal Q}^p(I^2)$ denotes the space of functions over the master element $I^2:=\{\bxi=(\xi,\eta):\;-1\leq\xi,\eta\leq1\}$ formed by tensor products of polynomials of degree at most $p$ in each direction. Each physical element $\kappa$ is the image of $I^2$ through the mapping ${\bf x}={\bf x}_\kappa(\bxi)$. Likewise, each edge in ${\cal E}_h$ is the image of $I=[-1,1]$ through the mapping ${\bf x}={\bf x}_e(\xi)$. The approximate solution to \cref{eq:HRM_PDEs} is sought under the form
\begin{equation*}
 {\bf u}_h({\bf x},t)=\sum_{0\leq i,j\leq p}\phi_\kappa^{ij}({\bf x}){\bf U}_\kappa^{ij}(t) \quad \forall{\bf x}\in\kappa,\, \kappa\in X_h,\, \forall t\geq0,
\end{equation*}

\noindent where $({\bf U}_\kappa^{ij})_{0\leq i,j\leq p}$ are the degrees of freedom (DOFs) in the element $\kappa$. The subset $(\phi_\kappa^{ij})_{0\leq i,j\leq p}$ constitutes a basis of ${\cal V}_h^p$ restricted onto the element $\kappa$ and $(p+1)^2$ is its dimension. 

Let $(\ell_k)_{0\leq k\leq p}$ be the Lagrange interpolation polynomials in one space dimension associated to the Gauss-Lobatto nodes over $I$, $\xi_0=-1<\xi_1<\dots<\xi_p=1$: $\ell_k(\xi_l)=\delta_{k,l}$, $0\leq k,l \leq p$, with  $\delta_{k,l}$ the Kronecker symbol. In this work we use tensor products of these polynomials and of Gauss-Lobatto nodes (see \cref{fig:stencil_2D}):
\begin{equation}\label{eq:lag_basis}
 \phi_\kappa^{ij}({\bf x})=\phi_\kappa^{ij}({\bf x}_\kappa(\bxi))=\ell_i(\xi)\ell_j(\eta), \quad 0\leq i,j\leq p.
\end{equation}

\noindent which satisfy the following relation at quadrature points $\bxi_{i'j'}=(\xi_{i'},\xi_{j'})$ in $I^2$:
\begin{equation*}
 \phi_\kappa^{ij}({\bf x}_\kappa^{i'j'})=\delta_{i,i'}\delta_{j,j'}, \quad 0\leq i,j,i',j' \leq p, \quad {\bf x}_\kappa^{i'j'}:={\bf x}_\kappa(\bxi_{i'j'}),
\end{equation*}

\noindent so the DOFs correspond to the point values of the solution: ${\bf U}_\kappa^{ij}(t)={\bf u}_h({\bf x}_\kappa^{ij},t)$.

\begin{figure}[ht]
\begin{center}
\begin{tikzpicture}[scale=0.7]
\draw (3.,0.1)  node {$\bullet$};
\draw (2.8,3)   node {$\bullet$};
\draw (0,0)     node {$\bullet$};
\draw (-1.,3.4) node {$\bullet$};
\draw (1.2,1.625) node {$\kappa=\kappa^-$};
\draw (5.,1.65)   node {$\kappa^+$};
\draw (2.8,3.)    node[above] {$e$};
\draw (2.86,2.20) node[left]  {${\bf u}_h^-$};
\draw (2.86,2.20) node[right] {${\bf u}_h^+$};
\draw (0.83,0.03) node {$\bullet$};
\draw (2.17,0.07) node {$\bullet$};
\draw (2.94,0.90) node {$\bullet$};
\draw (2.86,2.20) node {$\bullet$};
\draw (-0.28,0.94) node {$\bullet$};
\draw (-0.723,2.46) node {$\bullet$};
\draw (1.75,3.11) node {$\bullet$};
\draw (0.05,3.29) node {$\bullet$};
\draw (0.61,0.93) node {$\bullet$};
\draw (0.27,2.39) node {$\bullet$};
\draw (2.05,0.91) node {$\bullet$};
\draw (1.87,2.27) node {$\bullet$};
\draw [>=stealth,->] (2.94,0.90) -- (3.9,0.95) ;
\draw (3.9,0.95) node[below right] {${\bf n}_e$};
\draw [>=stealth,-] (0.,0.) -- (3.0,0.1) ;
\draw [>=stealth,-] (3.,0.1) -- (2.8,3.) ;
\draw [>=stealth,-] (2.8,3.) -- (-1.,3.4) ;
\draw [>=stealth,-] (-1.,3.4) -- (0.,0.) ;
\draw [>=stealth,-] (3.,0.1) -- (6.5,-0.5) ;
\draw [>=stealth,-] (6.5,-0.5) -- (7.5,4.) ;
\draw [>=stealth,-] (7.5,4.) -- (2.8,3.) ;
\end{tikzpicture}
\caption{Inner and outer elements, $\kappa^-$ and $\kappa^+$, for $d=2$; definitions of traces ${\bf u}_h^\pm$ on the interface $e$ and of the unit outward normal vector ${\bf n}_e$; positions of quadrature points in $\kappa^-$ and on $e$ for $p=3$.}
\label{fig:stencil_2D}
\end{center}
\end{figure}
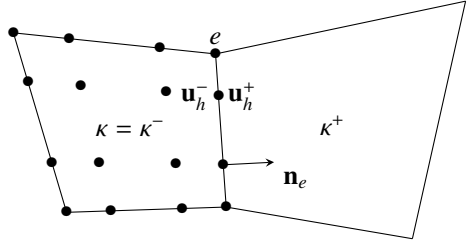

The integrals over elements and faces are approximated by using the Gauss-Lobatto quadrature rules so the quadrature and interpolation nodes are collocated:
\begin{equation}\label{eq:GaussLobatto_quad}
 \int_{\kappa}f({\bf x})dV \simeq \sum_{0\leq i,j\leq p} \omega_i\omega_j J_\kappa^{ij}f({\bf x}_\kappa^{ij}), \quad \int_{e} f({\bf x})dS \simeq \sum_{0\leq k \leq p} \omega_k \tfrac{|e|}{2}f({\bf x}_e^{k}),
\end{equation}

\noindent with $\omega_i>0$, ${\bf x}_\kappa^{ij}$, and ${\bf x}_e^{k}$ the weights and nodes of the quadrature rules, and $J_\kappa^{ij}=\det\big(\partial_\bxi{\bf x}_\kappa(\bxi_{ij})\big)>0$. 

Finally, let define the cell-averaged operator, for instance for the numerical solution:
\begin{equation}\label{eq:cell-average}
 \langle{\bf u}_h\rangle_\kappa(t) :=  \sum_{1\leq i,j\leq p}\omega_i\omega_j\tfrac{J_\kappa^{ij}}{|\kappa|}{\bf U}_\kappa^{ij}(t).
\end{equation}

\noindent where $|\kappa|$ is evaluated through numerical quadrature so the weights satisfy
\begin{equation}\label{eq:2d-weights-cond}
 \sum_{1\leq i,j\leq p}\omega_i\omega_j\frac{J_\kappa^{ij}}{|\kappa|}=1. 
\end{equation}

\subsection{Space discretization}

The semi-discrete form of the DGSEM in space of problem \cref{eq:HRM_PDEs} starts from the following problem: for $t>0$ find ${\bf u}_h$ in $({\cal V}_h^p)^{n_c+d+1}$ such that

\begin{equation}\label{eq:semi-discr_var_form}
  \sum_{\kappa\in X_h}\int_{\kappa} v_h \big(\partial_t{\bf u_h} + \nabla\cdot{\bf f}({\bf u}_h)\big) dV - \sum_{e\in{\cal E}_h}\int_{e}\du v_h\df{\bf h}({\bf u}_h^-,{\bf u}_h^+,{\bf n}_e)-\du v_h{\bf f}({\bf u}_h)\df\cdot{\bf n}_e dS = 0  \quad \forall v_h\in{\cal V}_h^p,
\end{equation}

\noindent where $\du v_h\df=v_h^+-v_h^-$ denotes the jump operator and $v_h^\pm({\bf x })=\lim_{\varepsilon\downarrow0}v_h\big({\bf x }\pm\varepsilon{\bf n}_e({\bf x })\big)$ are the traces of $v_h$ at a point ${\bf x}$ on an interface $e$ in ${\cal E}_h$ and ${\bf n}_e({\bf x })$ denotes the unit normal vector to $e$ at ${\bf x }$ and pointing from $\kappa^-$ to $\kappa^+$. The ES relaxation-based numerical flux \cref{eq:relax_flux} is used to define ${\bf h}(\cdot,\cdot,\cdot)$. Substituting $v_h$ for the Lagrange interpolation polynomials \cref{eq:lag_basis} and using the Gauss-Lobatto quadrature rules \cref{eq:GaussLobatto_quad} to approximate the volume and surface integrals, \cref{eq:semi-discr_var_form} becomes: for all $\kappa\in X_h$, $0\leq i,j\leq p$, and $t>0$, we have
\begin{align*} 
 \omega_i\omega_j J_\kappa^{ij}\frac{d{\bf U}_\kappa^{ij}}{dt} &+ \omega_i\omega_j J_\kappa^{ij} \Big(\sum_{k=0}^p D_{ik}{\bf f}({\bf U}_\kappa^{kj})\nabla\xi(\bxi_{ij})+\sum_{k=0}^p D_{jk}{\bf f}({\bf U}_\kappa^{ik})\nabla\eta(\bxi_{ij})\Big)  \\
 &+ \omega_i\Big(J_{e}^{ip}\delta_{jp}{\bf d}({\bf x}_\kappa^{ip},t)+J_{e}^{i0}\delta_{j0}{\bf d}({\bf x}_\kappa^{i0},t)\Big) + \omega_j\Big(J_{e}^{pj}\delta_{ip}{\bf d}({\bf x}_\kappa^{pj},t)+J_{e}^{0j}\delta_{i0}{\bf d}({\bf x}_\kappa^{0j},t)\Big) = 0, 
\end{align*}

\noindent where $D_{ik} = \ell_k'(\xi_i)$, ${\bf d}({\bf x},t)={\bf h}\big({\bf u}_h^{-}({\bf x},t),{\bf u}_h^{+}({\bf x},t),{\bf n}_e({\bf x})\big)-{\bf f}\big({\bf u}_h^{-}({\bf x},t)\big)\cdot{\bf n}_e({\bf x})$ at ${\bf x}$ on $e$, ${\bf u}_h^\pm({\bf x},t)$ denote the traces of the numerical solution on $e$ (see \cref{fig:stencil_2D}), and $J_{e}^{ip}:=\det(\partial_\xi{\bf x}_e)_{{\bf x}_\kappa^{ip}}=\tfrac{|e|}{2}$ where ${\bf x}_\kappa^{ip}$, $0\leq i\leq p$, uniquely identify $e$. 


As explained in the introduction, the volume integral in the above equation is modified so as to satisfy an entropy balance \cite{fisher_carpenter_13,wintermeyer_etal_17}: the physical fluxes are replaced by EC numerical fluxes \cref{eq:entropy_conserv_flux}, while the metric terms are modified to achieve conservation of the scheme over elements. The semi-discrete ES scheme thus reads
\begin{equation}\label{eq:semi-discr_DGSEM}
 \omega_i\omega_j J_\kappa^{ij}\frac{d{\bf U}_\kappa^{ij}}{dt} + {\bf R}_\kappa^{ij}({\bf u}_h) = 0 \quad \forall\kappa\in X_h,\; 0\leq i,j\leq p,\; t>0,
\end{equation}

\noindent with
\begin{align}
  {\bf R}_\kappa^{ij}({\bf u}_h) &= 2\omega_i\omega_j\Big(\sum_{k=0}^p D_{ik}{\bf h}_{ec}\big({\bf U}_\kappa^{ij},{\bf U}_\kappa^{kj},\{J_\kappa\nabla\xi\}_{(i,k)j}\big) + \sum_{k=0}^p D_{jk}{\bf h}_{ec}\big({\bf U}_\kappa^{ij},{\bf U}_\kappa^{ik},\{J_\kappa\nabla\eta\}_{i(j,k)}\big)\Big) \nonumber \\
  &+ \omega_i\Big(J_{e}^{ip}\delta_{jp}{\bf d}({\bf x}_\kappa^{ip},t)+J_{e}^{i0}\delta_{j0}{\bf d}({\bf x}_\kappa^{i0},t)\Big) + \omega_j\Big(J_{e}^{pj}\delta_{ip}{\bf d}({\bf x}_\kappa^{pj},t)+J_{e}^{0j}\delta_{i0}{\bf d}({\bf x}_\kappa^{0j},t)\Big), \label{eq:semi-discr_DGSEM-res}
\end{align}

\noindent where $\{J_\kappa\nabla\xi\}_{(i,k)j}=\tfrac{1}{2}\big(J_\kappa^{ij}\nabla\xi(\bxi_{ij})+J_\kappa^{kj}\nabla\xi(\bxi_{kj})\big)$, $\{J_\kappa\nabla\eta\}_{i(j,k)}=\tfrac{1}{2}\big(J_\kappa^{ij}\nabla\eta(\bxi_{ij})+J_\kappa^{ik}\nabla\eta(\bxi_{ik})\big)$, and ${\bf h}_{ec}(\cdot,\cdot,\cdot)$ denotes the EC numerical flux \cref{eq:EC_flux}. The DGSEM \cref{eq:semi-discr_DGSEM} from \cite{winters_etal_16,wintermeyer_etal_17} is one of more general conservative elementwise flux differencing schemes satisfying the semi-discrete entropy inequality for the cell-averaged entropy \cite{fisher_carpenter_13}:
\begin{equation}\label{eq:mean_DGSEM}
 |\kappa|\frac{d\langle\eta({\bf u}_h)\rangle_\kappa}{dt}+\sum_{e\in\partial\kappa}\sum_{k=0}^p\tfrac{|e|}{2}Q\big({\bf u}_h^-({\bf x}_e^{k},t),{\bf u}_h^+({\bf x}_e^{k},t),{\bf n}_e\big) \leq 0,
\end{equation}

\noindent where $Q(\cdot,\cdot,\cdot)$ is defined in \cref{eq:3pt-scheme-ineq}. Finally, high-order accuracy of \cref{eq:semi-discr_DGSEM} has been proved in \cite{chen_shu_17,ranocha_18}.

%
%
\section{Fully discrete scheme}\label{sec:scheme_properties}

We now focus on the fully discrete scheme and we first use a one-step first-order explicit time discretization and analyze its properties. High-order time integration will be done by using strong-stability preserving explicit Runge-Kutta methods \cite{spiteri_ruuth02} that keep the properties of the first-order in time scheme under some condition on the time step.

\subsection{Time discretization}

Let $\Delta t^{(n)}=t^{(n+1)}-t^{(n)}>0$, with $t^{(0)}=0$, be the time step, and use the notations ${\bf u}_h^{(n)}(\cdot)={\bf u}_h(\cdot,t^{(n)})$ and ${\bf U}_\kappa^{ij,n}={\bf U}_\kappa^{ij}(t^{(n)})$. The fully discrete DGSEM scheme for \cref{eq:HRM_PDEs} reads
\begin{equation}\label{eq:fully-discr_DGSEM}
 \omega_i\omega_j J_\kappa^{ij} \frac{{\bf U}_\kappa^{ij,n+1}-{\bf U}_\kappa^{ij,n}}{\Delta t^{(n)}} + {\bf R}_\kappa^{ij}({\bf u}_h^{(n)}) = 0 \quad \forall \kappa\in X_h, \; 0\leq i,j\leq p, \; n\geq 0,
\end{equation}

\noindent where the vector of residuals ${\bf R}_\kappa^{ij}(\cdot)$ is defined by \cref{eq:semi-discr_DGSEM-res}. The projection of the initial condition \cref{eq:HRM_PDEs-b} onto the function space reads ${\bf U}_\kappa^{ij,0} = {\bf u}_0({\bf x}_\kappa^{ij})$ for all $\kappa$ in $X_h$ and $0\leq i,j\leq p$.

%
\subsection{Properties of the discrete scheme}\label{sec:ana_1st_order_time_discr}

We have the following results for the fully discrete solution of the DGSEM that guaranty its robustness and the preservation of stationary material interfaces. Let recall that ${\bf h}(\cdot,\cdot,\cdot)$ in \cref{eq:semi-discr_DGSEM,eq:semi-discr_DGSEM-res} is the relaxation-based ES and robust numerical flux \cref{eq:relax_flux}.

\begin{theorem}\label{th:pos_DG_BN}
Let $n\geq0$ and assume that ${\bf U}_{\kappa}^{ij,n}$ is in $\Omega^a$ for all $0\leq i,j\leq p$ and $\kappa$ in $X_h$, then under the CFL condition
\begin{equation}\label{eq:CFL_cond}
 \Delta t^{(n)}\max_{\kappa\in X_h}\max_{e\in{\cal E}_h}\frac{|\partial\kappa_e|}{|e|}\max_{0\leq k\leq p}\frac{|\partial\kappa|}{\tilde{J}_\kappa^k}\big|\lambda\big({\bf u}_h^\pm({\bf x}_e^{k},t^{(n)})\big)\big| < \frac{1}{2p(p+1)},
\end{equation}

\noindent where $\tilde{J}_\kappa^k:=\min_{e\in\partial\kappa}J_\kappa({\bf x}_e^k)$, $|\lambda(\cdot)|$ is defined in \cref{eq:3-pt-CFL_cond}, and the $\kappa_e$ such that $\kappa=\cup_{e\in\partial\kappa}\kappa_e$ are defined in \cref{th:positive-2D-FVO2-scheme}, we have
\begin{equation*}
 \langle{\bf u}_h^{(n+1)}\rangle_\kappa\in\Omega^a \quad \forall\kappa\in X_h.
\end{equation*}

Moreover, the scheme exactly resolves stationary material interfaces at interpolation points.
\end{theorem}

\begin{proof}
 The positivity of the solution at time $t^{(n+1)}$ relies on techniques introduced in \cite{perthame_shu_96,zhang2010positivity} to rewrite a conservative high-order scheme for the cell-averaged solution as a convex combination of positive quantities. First, consider the three-point scheme \cref{eq:3pt-scheme-a} with the relaxation-based ES numerical flux \cref{eq:relax_flux}. We know that this scheme preserves the solution in the set of states \cref{eq:HRM-set-of-states} under the CFL condition \cref{eq:3-pt-CFL_cond}. The finite volume scheme \cref{eq:2D-FVO2-scheme} on quadrilaterals will thus be also positive under the condition \cref{eq:CFL-positive-2D-FVO2-scheme} where $|\lambda(\cdot)|$ is defined in \cref{eq:3-pt-CFL_cond}.

Now summing \cref{eq:fully-discr_DGSEM} over $0\leq i,j\leq p$ gives for the cell-averaged solution
\begin{eqnarray*}
 \langle{\bf u}_h^{(n+1)}\rangle_\kappa &\overset{\cref{eq:fully-discr_DGSEM}}{=}& \langle{\bf u}_h^{(n)}\rangle_\kappa -\tfrac{\Delta t^{(n)}}{|\kappa|}\sum_{0\leq i,j\leq p}{\bf R}_\kappa^{ij}({\bf u}_h^{(n)}) \\
 &\overset{\cref{eq:semi-discr_DGSEM-res}}{=}& \langle{\bf u}_h^{(n)}\rangle_\kappa - \tfrac{\Delta t^{(n)}}{|\kappa|}\sum_{e\in\partial\kappa}\sum_{k=0}^p\omega_k\tfrac{|e|}{2}{\bf h}\big({\bf u}_h^-({\bf x}_e^{k},t^{(n)}),{\bf u}_h^+({\bf x}_e^{k},t^{(n)}),{\bf n}_e\big), 
\end{eqnarray*}

\noindent by conservation of the DGSEM \cite{winters_etal_16,wintermeyer_etal_17}. Multiplying $\langle{\bf u}_h^{(n)}\rangle_\kappa$ in the RHS by $\sum_{e\in\partial\kappa}\tfrac{|e|}{|\partial\kappa|}=1$ and using \cref{eq:cell-average},  we rewrite the above relation as
\begin{equation*}
  \langle{\bf u}_h^{(n+1)}\rangle_\kappa = \sum_{e\in\partial\kappa}\tfrac{|e|}{|\partial\kappa|}\sum_{0\leq i,j\leq p}\omega_i\omega_j\tfrac{J_\kappa^{ij}}{|\kappa|}{\bf U}_\kappa^{ij,n} - \tfrac{\Delta t^{(n)}}{|\kappa|}\sum_{e\in\partial\kappa}\sum_{k=0}^p\omega_k\tfrac{|e|}{2}{\bf h}\big({\bf u}_h^-({\bf x}_e^{k},t^{(n)}),{\bf u}_h^+({\bf x}_e^{k},t^{(n)}),{\bf n}_e\big).
\end{equation*}

Then, using $\omega_0=\omega_p$ and removing and adding the same quantity, we get
\begin{align*}
 \langle{\bf u}_h^{(n+1)}\rangle_\kappa &= \sum_{e\in\partial\kappa}\sum_{0\leq i,j\leq p}\tfrac{|e|}{|\partial\kappa|}\omega_i\omega_j\tfrac{J_\kappa^{ij}}{|\kappa|}{\bf U}_\kappa^{ij,n} - \sum_{e\in\partial\kappa}\sum_{k=0}^p\tfrac{|e|}{|\partial\kappa|}\omega_k\omega_0\tfrac{\tilde{J}_\kappa^k}{|\kappa|}{\bf u}_h^{-}({\bf x}_e^{k},t^{(n)}) \\
  &+ \sum_{k=0}^p\omega_k\omega_0\tfrac{\tilde{J}_\kappa^k}{|\kappa|}\Big(\sum_{e\in\partial\kappa}\tfrac{|e|}{|\partial\kappa|}{\bf u}_h^-({\bf x}_e^{k},t^{(n)})-\tfrac{\Delta t^{(n)}}{2\omega_0\tilde{J}_\kappa^k}\sum_{e\in\partial\kappa}|e|{\bf h}\big({\bf u}_h^-({\bf x}_e^{k},t^{(n)}),{\bf u}_h^+({\bf x}_e^{k},t^{(n)}),{\bf n}_e\big)\Big),
\end{align*}

\noindent where $\tilde{J}_\kappa^k$ is defined in \cref{th:pos_DG_BN}. We now use the fact that the traces ${\bf u}_h^{-}({\bf x}_e^{k})$ correspond to some DOFs ${\bf U}_\kappa^{ij}$ that share the edge $e$ (see \cref{fig:stencil_2D}): if ${\bf x}_\kappa^{ij}\in e$ then there exists $0\leq k\leq p$ such that ${\bf x}_e^{k}={\bf x}_\kappa^{ij}$, so ${\bf u}_h^{-}({\bf x}_e^{k},t^{(n)})={\bf U}_\kappa^{ij,n}$. Rearranging the two first terms and multiplying the last term in the RHS with $\sum_{f\in\partial\kappa}\tfrac{|f|}{|\partial\kappa|}=1$, we finally obtain
\begin{align*}
  \langle{\bf u}_h^{(n+1)}\rangle_\kappa 
  &= \sum_{e\in\partial\kappa}\sum_{0\leq i,j\leq p,{\bf x}_\kappa^{ij}\notin e}\tfrac{|e|}{|\partial\kappa|}\tfrac{\omega_i\omega_jJ_\kappa^{ij}}{|\kappa|}{\bf U}_\kappa^{ij,n} + \sum_{e\in\partial\kappa}\sum_{k=0}^p\tfrac{|e|}{|\partial\kappa|}\omega_k\omega_0\tfrac{J_\kappa({\bf x}_e^k)-\tilde{J}_\kappa^k}{|\kappa|}{\bf u}_h^-({\bf x}_e^{k},t^{(n)})  \\
  &+ \sum_{f\in\partial\kappa}\sum_{k=0}^p\tfrac{|f|}{|\partial\kappa|}\tfrac{\omega_k\omega_0\tilde{J}_\kappa^k}{|\kappa|}\Big(\sum_{e\in\partial\kappa}\!\!\tfrac{|e|{\bf u}_h^-({\bf x}_e^{k},t^{(n)})}{|\partial\kappa|} -\tfrac{\Delta t^{(n)}}{2\omega_0\tilde{J}_\kappa^k}\sum_{e\in\partial\kappa}\!\!|e|{\bf h}\big({\bf u}_h^-({\bf x}_e^{k},t^{(n)}),{\bf u}_h^+({\bf x}_e^{k},t^{(n)}),{\bf n}_e\big)\Big).
\end{align*}

The terms between brackets correspond to the RHS in \cref{eq:2D-FVO2-scheme} and are therefore positive under the condition \cref{eq:CFL_cond}. We thus conclude that $\langle{\bf u}_h^{(n+1)}\rangle_\kappa$ is a convex combination of positive quantities with weights $\tfrac{|e|}{|\partial\kappa|}\omega_i\omega_jJ$ with $J=J_\kappa^{ij}$, $\tilde{J}_\kappa^k$, or $J_\kappa({\bf x}_e^k)-\tilde{J}_\kappa^k\geq0$ from the definition of $\tilde{J}_\kappa^k$ in \cref{th:pos_DG_BN}.

Finally, suppose that the initial condition consists in a stationary material interface with states ${\bf Y}_L,\rho_L,{\bf v}=0,$ and $\mathrm{p}$ in $\Omega_L$ and ${\bf Y}_R,\rho_R,{\bf v}=0$, and $\mathrm{p}$ in $\Omega_R$ with $\ol{\Omega_L}\cup\ol{\Omega_R}=\ol{\Omega}$, then so do the DOFs. The numerical fluxes \cref{eq:EC_flux} and \cref{eq:relax_flux} reduce to ${\bf h}_{ec}({\bf u}^-,{\bf u}^+,{\bf n}) = {\bf h}({\bf u}^-,{\bf u}^+,{\bf n}) = (0,0,\mathrm{p}{\bf n}^\top,0)^\top$ and we easily obtain from \cref{eq:semi-discr_DGSEM-res} that ${\bf R}_\kappa^{ij}({\bf u}_h^{(n)})=0$ so stationary contacts remain stationary for all times and the DOFs are the exact values. Note that when the discontinuity $\ol{\Omega_L}\cap\ol{\Omega_R}$ corresponds to mesh interfaces, the relaxation based approximate Riemann solver \cref{eq:relax-ARS-xt} provides the exact solution and the contact discontinuity is exactly resolved within cell elements.\qed
\end{proof}

\begin{remark}
 The factor $\tfrac{2}{p(p+1)}=\omega_0=\omega_p$ in \cref{eq:CFL_cond} compared to \cref{eq:CFL-positive-2D-FVO2-scheme} may be compared with the results in \cite{zhang2010positivity} obtained on Cartesian meshes. Though conditions \cref{eq:CFL_cond} and \cref{eq:CFL-positive-2D-FVO2-scheme} are not optimal, they are sufficient for our purpose with the assumption of a shape-regular mesh. We refer to \cite{calgaro_etal_13} and references therein for a review on sharp CFL conditions in the context of finite volume schemes.
\end{remark}

\subsection{Limiting strategy}\label{sec:limiters}

The properties in \cref{th:pos_DG_BN} hold only for the cell-averaged numerical solution at time $t^{(n+1)}$, which is not sufficient for robustness and stability of numerical computations. We use the a posteriori limiter introduced in \cite{zhang2010positivity} to extend positivity of the solution at nodal values within elements in order to guaranty robustness of the DGSEM which requires $\rho>0$, $\rho_i>0$, and $e>0$. From \cref{eq:partial_densities,eq:mixture_r_Cv_Cp} this imposes $r({\bf Y})=r_{n_c}+\sum_{i=1}^{n_c-1}Y_i(r_i-r_{n_c})>0$. Without loss of generality we select the $n_c$th component as one satisfying $r_{n_c}=\min_{1\leq r_i\leq n_c}r_i$ and we impose $Y_{1\leq i<n_c}>-\tfrac{r_{n_c}}{(n_c-1)(r_i-r_{n_c})}$. We then enforce positivity of nodal values through
\begin{subequations}\label{eq:pos_limiter}
\begin{align}
 \breve{\rho}_\kappa^{ij,n+1} &= \theta_\kappa^\rho{\rho}_\kappa^{ij,n+1}+(1-\theta_\kappa^\rho)\langle{\rho}_h^{(n+1)}\rangle_\kappa, \quad 
 \tilde{\rho Y_k}_{\kappa}^{ij,n+1} = \theta_\kappa^k{\rho Y_k}_{\kappa}^{ij,n+1}+(1-\theta_\kappa^k)\breve{\rho}_\kappa^{ij,n+1}\tfrac{\langle\rho Y_{i_h}^{(n+1)}\rangle_\kappa}{\langle\rho_{h}^{(n+1)}\rangle_\kappa}, \\
 \tilde{\bf V}_\kappa^{ij,n+1} &= \theta_\kappa^e\breve{\bf V}_\kappa^{ij,n+1} + (1-\theta_\kappa^e)\langle{\bf v}_h^{(n+1)}\rangle_\kappa \quad \forall 0\leq i,j\leq p, \quad \kappa \in  X_h,
\end{align}
\end{subequations}

\noindent with ${\bf V}=(\rho,\rho{\bf v}^\top,\rho E)^\top$, ${\bf v}_h$ the numerical solution for ${\bf V}$, and $0\leq \theta_\kappa^{\rho},\theta_\kappa^{1\leq i<n_c},\theta_\kappa^{e}\leq1$ defined by
\begin{align*}
 \theta_\kappa^{\rho} &= \min\Big(\frac{\langle\rho_h^{(n+1)}\rangle_\kappa-\varepsilon}{\langle\rho_h^{(n+1)}\rangle_\kappa-\rho_\kappa^{min}},1\Big), \quad \rho_\kappa^{min}=\min_{0\leq i,j\leq p} \rho_\kappa^{ij,n+1}, \\
 \theta_\kappa^{Y_i} &= \min\Big(\frac{\langle \rho Y_{i_h}^{(n+1)}\rangle_\kappa-\langle \rho_h^{(n+1)}\rangle_\kappa\big(\varepsilon-\tfrac{r_{n_c}}{(n_c-1)(r_i-r_{n_c})}\big)}{\langle \rho Y_{i_h}^{(n+1)}\rangle_\kappa-\langle \rho_h^{(n+1)}\rangle_\kappa Y_{i_\kappa}^{min}},1\Big), \quad Y_{i_\kappa}^{min}=\min_{0\leq k,l\leq p} \tfrac{\rho Y_{i_\kappa}^{kl,n+1}}{\rho_\kappa^{kl,n+1}}, \\
 \theta_\kappa^e &= \min_ {0\leq i,j\leq p}\Big(\theta_\kappa^{e,ij}: \quad e\big(\theta_\kappa^{e,ij}(\breve{\bf V}_\kappa^{ij,n+1}- \langle{\bf v}_h^{(n+1)}\rangle_\kappa) + \langle{\bf v}_h^{(n+1)}\rangle_\kappa\big) \geq \varepsilon\Big),
\end{align*}

\noindent and $\varepsilon=10^{-10}$ a parameter to guaranty positivity of nodal values, e.g., $\tilde{\rho}_\kappa^{ij,n+1}\geq\varepsilon>0$.


%
%
\section{Numerical experiments}\label{sec:num_xp}

In this section we present numerical experiments, obtained with the CFD code {\it Aghora} developed at ONERA \cite{renac_etal15}, on problems involving $n_c=2$ components in one and two space dimensions in order to illustrate the performance of the DGSEM derived in this work. Unless stated otherwise, we use a fourth-order accurate ($p=3$) scheme in space together with a four-stage third-order strong-stability preserving Runge-Kutta method \cite{spiteri_ruuth02}, while the limiter \cref{eq:pos_limiter} is applied at the end of each stage. We set $\gamma=\max(\gamma_1,\gamma_2)$ in \cref{eq:wave-estimate} and ensure the inequality in \cref{eq:gamma-max} by increasing the wave speed estimates $S_{X=L,R}$ by a factor $1.05$. The time step is evaluated through $\Delta t^{(n)}\max_{\kappa\in X_h}\tfrac{1}{|\kappa|}\sqrt{\sum_{e\in\kappa}|e|^2}|\lambda(\langle{\bf u}_h^{(n)}\rangle_\kappa)|\leq0.4$ where $\lambda(\cdot)$ is defined in \cref{eq:3-pt-CFL_cond}. This condition was seen to preserve the positivity of the solution during our experiments, while it constitutes a less complex and less restrictive condition than \cref{eq:CFL_cond}.

\subsection{Convection of void fraction and density waves}\label{sec:num_xp_HO_test_case}

We first consider the convection of void fraction and density waves in a flow with uniform velocity and pressure \cite{coquel_etal_DGSEM_BN_21}. Let $\Omega=(0,1)$, we set $\gamma_1=1.6$, $\gamma_2=1.4$, $C_{v_1}=2$, $C_{v_2}=1$, and solve the problem \cref{eq:HRM_PDEs} with periodic conditions and the initial condition
\begin{equation*}
 Y_{1_0}(x) = \tfrac{1}{2}+\tfrac{1}{4}\sin(4\pi x), \quad \rho_0(x) = 1+\tfrac{1}{2}\sin(2\pi x), \quad u_0(x) = 1, \quad \mathrm{p}_0(x) = 1 \quad \forall x\in\Omega.
\end{equation*}

\Cref{tab:density_wave_error} indicates the norms of the numerical error on density $e_h=\rho_h-\rho$ for different polynomial degrees and grid refinements with associated convergence rates in space. We use the five-stage fourth-order Runge-Kutta scheme from \cite{spiteri_ruuth02} for $p=3$. \red{Results obtained with the first-order three-point scheme \cref{eq:2D-FV-scheme} (referred to as $p=0$) are also provided for the sake of comparison.} The expected $p+1$ order of convergence is recovered with the present method.

\begin{table}
     \begin{center}
     \caption{Void fraction and density waves: norms of the error at time $t=5$ and associated orders of convergence.}
     \begin{tabular}{clcccccc}
        \noalign{\smallskip}\hline\noalign{\smallskip}
		$p$ & $1/N$ & $\|e_h\|_{L^1(\Omega)}$  & ${\cal O}_1$ & $\|e_h\|_{L^2(\Omega)}$  & ${\cal O}_2$ & $\|e_h\|_{L^\infty(\Omega)}$  & ${\cal O}_\infty$ \\
        \noalign{\smallskip}\hline\noalign{\smallskip}
  	    & $1/16$ & $0.31913e\!+\!00$ & $-$    & $0.35219e\!+\!00$ & $-$     & $0.48878e\!+\!00$ & $-$ \\
  	    & $1/32$ & $0.29791e\!+\!00$ & $0.10$ & $0.33037e\!+\!00$ & $0.09$  & $0.46538e\!+\!00$ & $0.07$ \\
  	$0$ & $1/64$ & $0.23717e\!+\!00$ & $0.33$ & $0.26332e\!+\!00$ & $0.33$  & $0.37208e\!+\!00$ & $0.32$ \\
  	    & $1/128$& $0.15798e\!+\!00$ & $0.59$ & $0.17549e\!+\!00$ & $0.58$  & $0.24836e\!+\!00$ & $0.58$ \\
  	    & $1/256$& $0.92549e\!-\!01$ & $0.77$ & $0.10288e\!+\!00$ & $0.77$  & $0.14642e\!+\!00$ & $0.76$ \\
        \noalign{\smallskip}\hline\noalign{\smallskip}
  	    & $1/4$	 & $0.24745e\!+\!00$ & $-$     & $0.34889e\!+\!00$ & $-$     & $0.49518e\!+\!00$ & $-$ \\
  	    & $1/8$	 & $0.39647e\!+\!00$ & $-0.68$ & $0.43745e\!+\!00$ & $-0.33$ & $0.62568e\!+\!00$ & $-0.34$ \\
  	$1$ & $1/16$ & $0.20853e\!+\!00$ &  $0.93$ & $0.23087e\!+\!00$ & $0.92$  & $0.34328e\!+\!00$ & $0.87$ \\
  	    & $1/32$ & $0.60157e\!-\!01$ &  $1.79$ & $0.67213e\!-\!01$ & $1.78$  & $0.10832e\!+\!00$ & $1.66$ \\
  	    & $1/64$ & $0.17318e\!-\!01$ &  $1.80$ & $0.18977e\!-\!01$ & $1.82$  & $0.34753e\!-\!01$ & $1.64$ \\
        \noalign{\smallskip}\hline\noalign{\smallskip}
  	    & $1/4$	 & $0.82260e\!-\!01$ & $-$    & $0.10338e\!+\!00$ & $-$    & $0.16582e\!+\!00$ & $-$ \\
  	    & $1/8$	 & $0.85848e\!-\!02$ & $3.26$ & $0.11908e\!-\!01$ & $3.12$ & $0.25905e\!-\!01$ & $2.68$ \\
  	$2$ & $1/16$ & $0.21406e\!-\!02$ & $2.00$ & $0.28735e\!-\!02$ & $2.05$ & $0.74976e\!-\!02$ & $1.79$ \\
  	    & $1/32$ & $0.25474e\!-\!03$ & $3.08$ & $0.35685e\!-\!03$ & $3.01$ & $0.12013e\!-\!02$ & $2.64$ \\
  	    & $1/64$ & $0.23934e\!-\!04$ & $3.41$ & $0.33731e\!-\!04$ & $3.40$ & $0.13061e\!-\!03$ & $3.20$ \\
        \noalign{\smallskip}\hline\noalign{\smallskip}
  	    & $1/4$	 & $0.56380e\!-\!02$ & $-$    & $0.74671e\!-\!02$ & $-$    & $0.14145e\!-\!01$ & $-$ \\
  	    & $1/8$	 & $0.16611e\!-\!02$ & $1.76$ & $0.21002e\!-\!02$ & $1.83$ & $0.45926e\!-\!02$ & $1.62$ \\
  	$3$ & $1/16$ & $0.92849e\!-\!04$ & $4.16$ & $0.13191e\!-\!03$ & $3.99$ & $0.43411e\!-\!03$ & $3.40$ \\
  	    & $1/32$ & $0.30461e\!-\!05$ & $4.96$ & $0.44199e\!-\!05$ & $4.90$ & $0.19246e\!-\!04$ & $4.50$ \\
  	    & $1/64$ & $0.30281e\!-\!06$ & $3.33$ & $0.42853e\!-\!06$ & $3.37$ & $0.14364e\!-\!05$ & $3.74$ \\
        \noalign{\smallskip}\hline\noalign{\smallskip}
    \end{tabular}
    \label{tab:density_wave_error}
    \end{center}
\end{table}

\subsection{One-dimensional shock-tube problems}

Let consider Riemann problems associated with the initial condition ${\bf u}_0(x)={\bf u}_L$ if $x<x_s$ and ${\bf u}_R$ if $x>x_s$ (see \cref{tab:RP_IC} for details).

\begin{table}
     \begin{bigcenter}
     \caption{Initial conditions and physical parameters of Riemann problems with ${\bf\cal U}=(Y_1,\rho,u,\mathrm{p})^\top$.}
     \begin{tabular}{lllcccccc}
        \noalign{\smallskip}\hline\noalign{\smallskip}
        test & left state ${\bf\cal U}_L$ & right state ${\bf\cal U}_R$ & $x_s$ & $t$ & $\gamma_1$ & $C_{v_1}$ & $\gamma_2$ & $C_{v_2}$ \\
        \noalign{\smallskip}\hline\noalign{\smallskip}
        RP0 & $(0.4, 2, 0, 1)^\top$ & $(0.6, 1.5, 0, 2)^\top$ & $0$ & $0.2$ & $1.5$ & $1$ & $1.3$ & $1$ \\
        RP1 & $(0.5, 1, 0, 1)^\top$ & $(0.5, 0.125, 0, 0.1)^\top$ & $0$ & $0.2$ & $1.5$ & $1$ & $1.3$ & $1$ \\
        RP2 & $(1, 1.602, 0, 10^6)^\top$ & $(0, 1.122, 0, 10^5)^\top$ & $-0.1$ & $3\times10^{-4}$ & $\frac{5}{3}$ & $3.12$  & $1.4$ & $0.743$ \\
        RP3 & $(0.2, 0.99988, -1.99931, 0.4)^\top$ & $(0.5, 0.99988, 1.99931, 0.4)^\top$ & $0$ & $0.15$ & $1.5$ & $1$  & $1.3$ & $1$ \\
        RP4 & $(1, 1, 0, 1)^\top$ & $(0, 0.1, 0, 1)^\top$ & $0$ & $0.08$ & $1.6$ & $1$  & $1.4$ & $1$ \\
        RP5 & $(1, 1, 1, 1)^\top$ & $(0, 0.1, 1, 1)^\top$ & $0$ & $0.08$ & $1.6$ & $1$  & $1.4$ & $1$ \\
        \noalign{\smallskip}\hline\noalign{\smallskip}
    \end{tabular}
     \label{tab:RP_IC}
    \end{bigcenter}
\end{table}

We first validate the entropy conservation of the numerical flux \cref{eq:EC_flux}. We thus replace the ES numerical flux ${\bf h}$ at interfaces in \cref{eq:semi-discr_DGSEM-res} by the EC flux \cref{eq:EC_flux}. We follow the experimental setup introduced in \cite{bohm_etal_18} and choose an initial condition corresponding to problem RP0 in \cref{tab:RP_IC} resulting in the development of weak shock and contact waves on a domain of unit length with periodic boundary conditions. As a result of entropy conservation of the space discretization, only the time integration scheme should modify the global entropy budget at the discrete level. We thus evaluate the difference
\begin{equation}\label{eq:entropy_conserv_error}
 e_h(t) := \sum_{\kappa\in X_h}|\kappa|\langle\eta({\bf u}_h)-\eta({\bf u}_0)\rangle_\kappa(t),
\end{equation}

\noindent which quantifies the difference between the discrete entropy at final time and the initial entropy over the domain $\Omega$. We observe in \cref{tab:test_EC} that the error \cref{eq:entropy_conserv_error} decreases to machine accuracy when refining the time step with third-order of convergence as asymptotic limit corresponding to the theoretical approximation order of the time integration scheme \cite{spiteri_ruuth02}. This validates the entropy conservation property of the numerical flux \cref{eq:EC_flux}. For comparison we also ran the same experiments when using the ES flux \cref{eq:relax_flux} at interfaces which confirmed a global entropy dissipation independent of the time step.

\begin{table}
     \begin{bigcenter}
     \caption{Entropy conservation error \cref{eq:entropy_conserv_error} and order of convergence ${\cal O}$ obtained for problem RP0 ($N=100$ cells, $p=3$) with either an EC or an ES numerical flux at interfaces.}
     \begin{tabular}{lccc}
       \noalign{\smallskip}\hline\noalign{\smallskip}
       flux & \multicolumn{2}{c}{EC} & ES \\
       \noalign{\smallskip}\hline\noalign{\smallskip}
       & $e_h(t)$ & ${\cal O}$ & $e_h(t)$ \\
        \noalign{\smallskip}\hline\noalign{\smallskip}
        $\Delta t$  & $-6.34260e\!-\!08$ & $1.16$ & $-7.69825e\!-\!07$ \\
        $\Delta t/2$  & $-1.33328e\!-\!08$ & $2.25$ & $-7.69777e\!-\!07$ \\
        $\Delta t/4$  & $-1.82337e\!-\!09$ & $2.87$ & $-7.69771e\!-\!07$ \\
        $\Delta t/8$ & $-2.30876e\!-\!10$ & $2.98$ & $-7.69770e\!-\!07$ \\
        $\Delta t/16$ & $-2.89133e\!-\!11$ & $3.00$ & $-7.69770e\!-\!07$ \\
        \noalign{\smallskip}\hline\noalign{\smallskip}
    \end{tabular}
     \label{tab:test_EC}
    \end{bigcenter}
\end{table}

Results for problems RP1 to RP5 are displayed in \cref{fig:solution_RPs,fig:solution_RPs_contact} where we compare the numerical solution in symbols with the exact solution in lines. Problem RP1 corresponds to the classical Sod problem for the compressible Euler equations since the mass fraction is uniform and corresponds to an equivalent $\gamma({\bf Y})=1.4$ for the mixture. Problem RP2 comes from \cite{houim_kuo_11} and corresponds to a He-N$_2$ shock tube problem, RP3 corresponds to a multicomponent near vacuum problem with two rarefaction waves, while RP4 and RP5 consist in material interfaces \cite{abgrall_karni_01}. We observe that the shock and contact waves are well captured and only some spurious oscillations of small amplitude are observed in RP2 which also exhibits a train of oscillations at the tail of the rarefaction wave. Positivity of the density is preserved in the near vacuum region which highlights the robustness of the scheme. Finally, the stationary contact wave in RP4 is exactly resolved as expected from \cref{th:pos_DG_BN}, while spurious oscillations inherent to \blue{discretely} conservative schemes \cite{abgrall_96,abgrall_karni_01} are observed around the moving interface in RP5 \blue{as pointed out in \cref{sec:intro_ECEF}}. 

\begin{figure}
\begin{bigcenter}
\subfloat{\begin{picture}(0,0) \put(-9,50){$Y_1$} \end{picture}}
\subfloat{\epsfig{figure=\LOCALFIGPATHRP 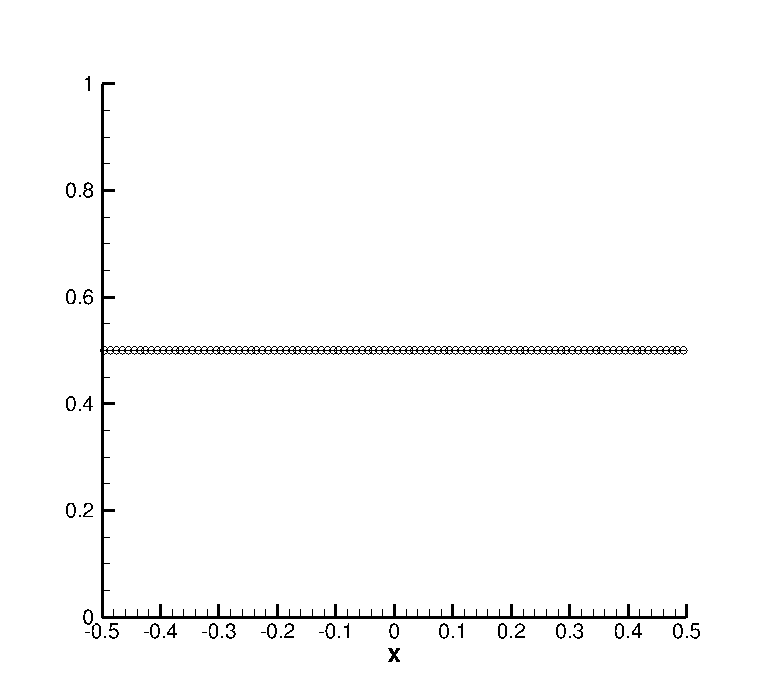 ,width=4.5cm}}
\subfloat{\epsfig{figure=\LOCALFIGPATHRP 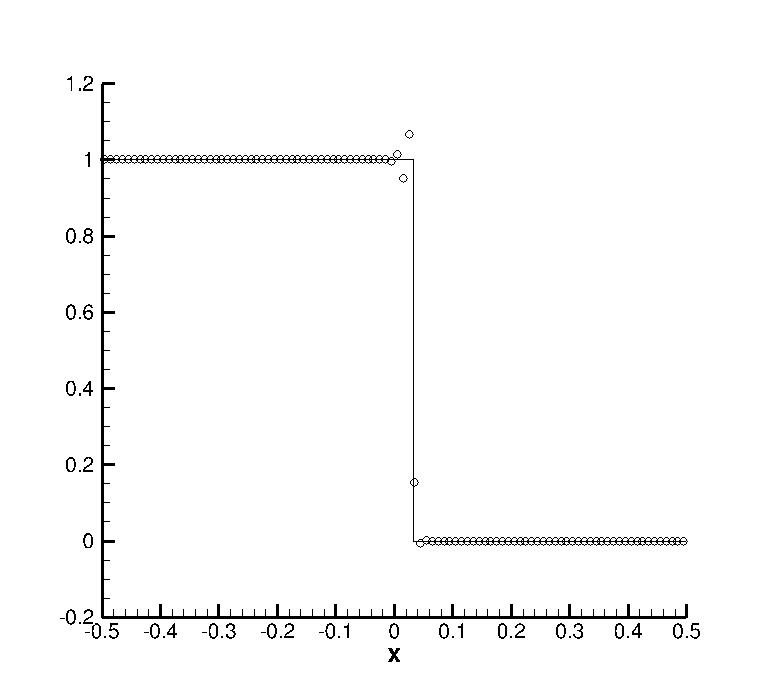 ,width=4.5cm}}
\subfloat{\epsfig{figure=\LOCALFIGPATHRP 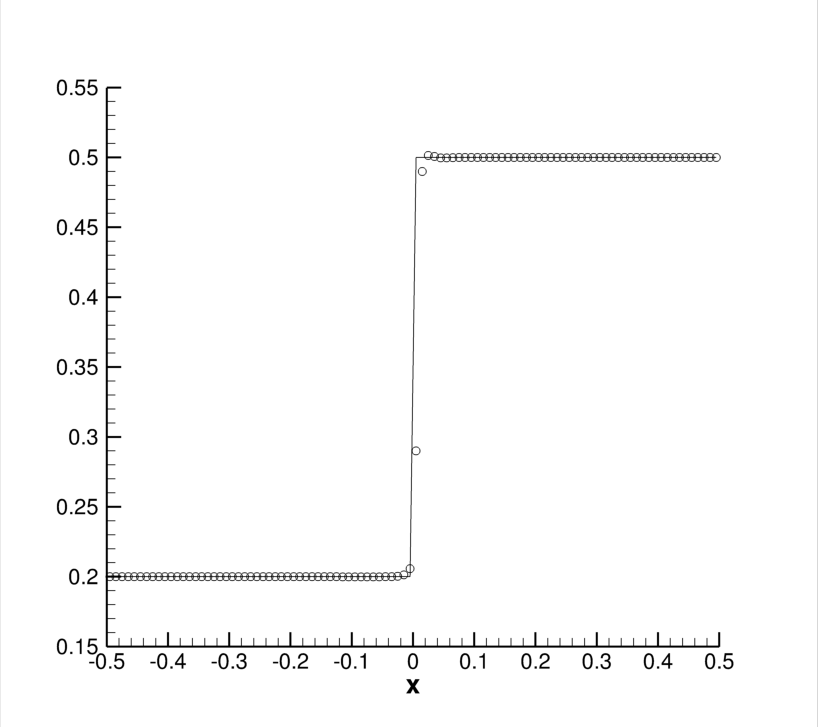 ,width=4.5cm}}\\
%
\subfloat{\begin{picture}(0,0) \put(-8,50){$\rho$} \end{picture}}
\subfloat{\epsfig{figure=\LOCALFIGPATHRP 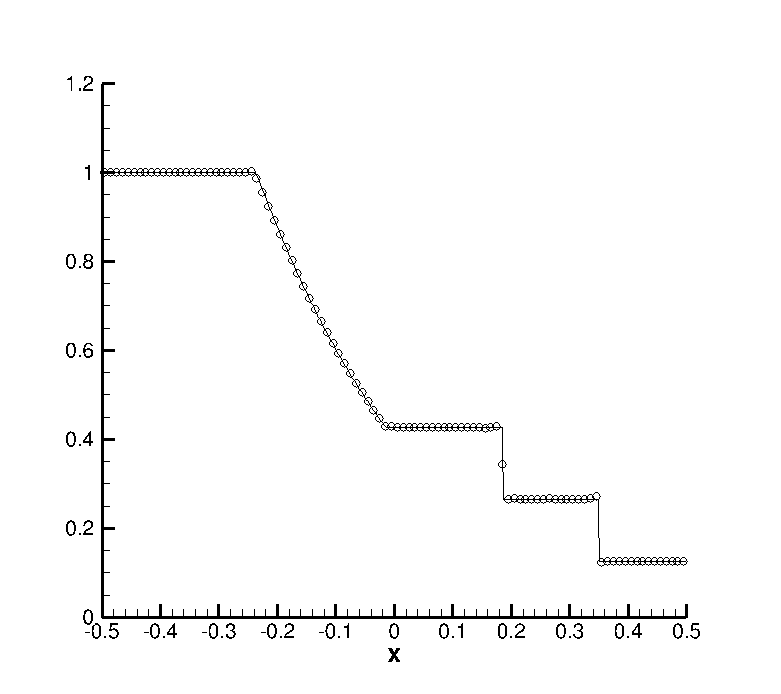 ,width=4.5cm}}
\subfloat{\epsfig{figure=\LOCALFIGPATHRP 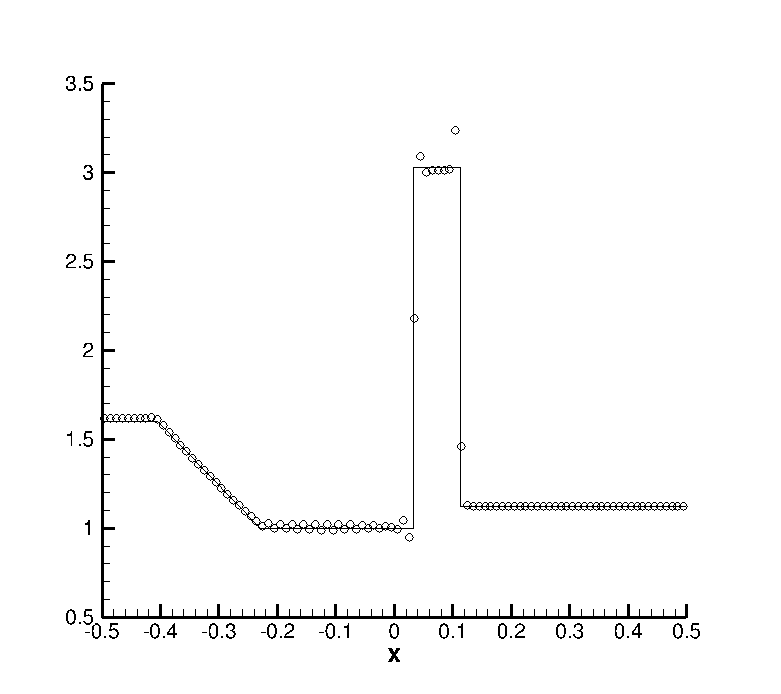 ,width=4.5cm}}
\subfloat{\epsfig{figure=\LOCALFIGPATHRP 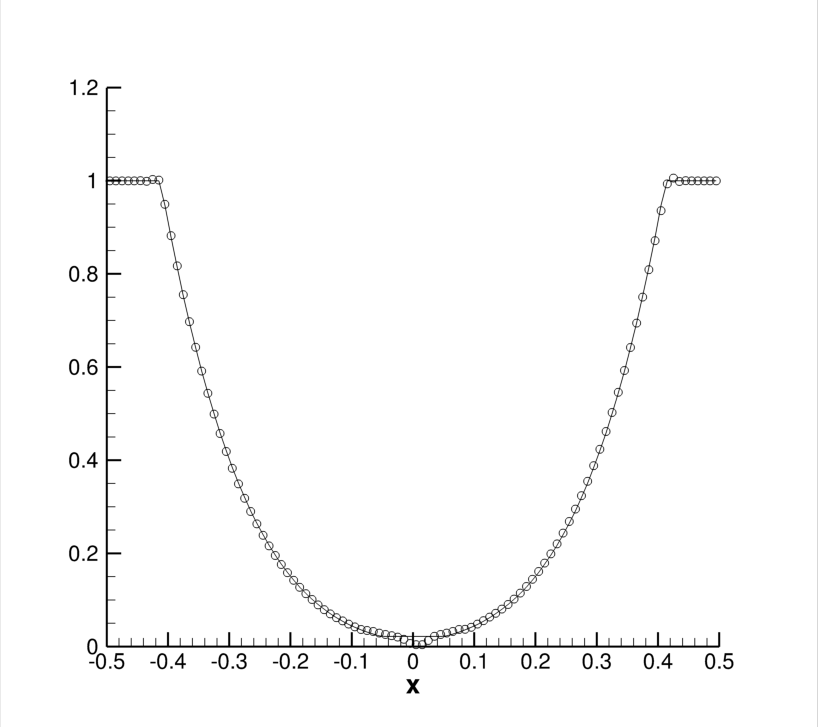 ,width=4.5cm}}\\
%
\subfloat{\begin{picture}(0,0) \put(-8,50){$u$} \end{picture}}
\subfloat{\epsfig{figure=\LOCALFIGPATHRP 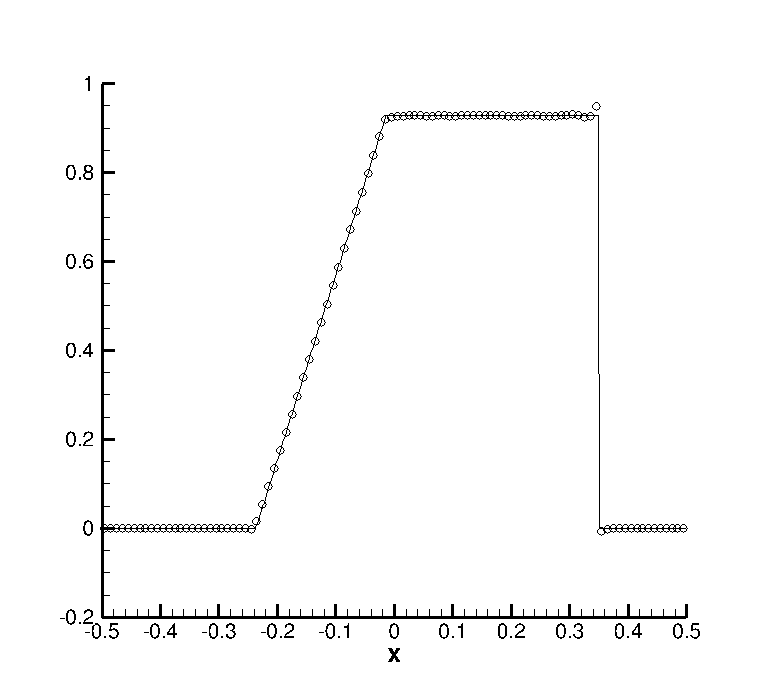 ,width=4.5cm}}
\subfloat{\epsfig{figure=\LOCALFIGPATHRP 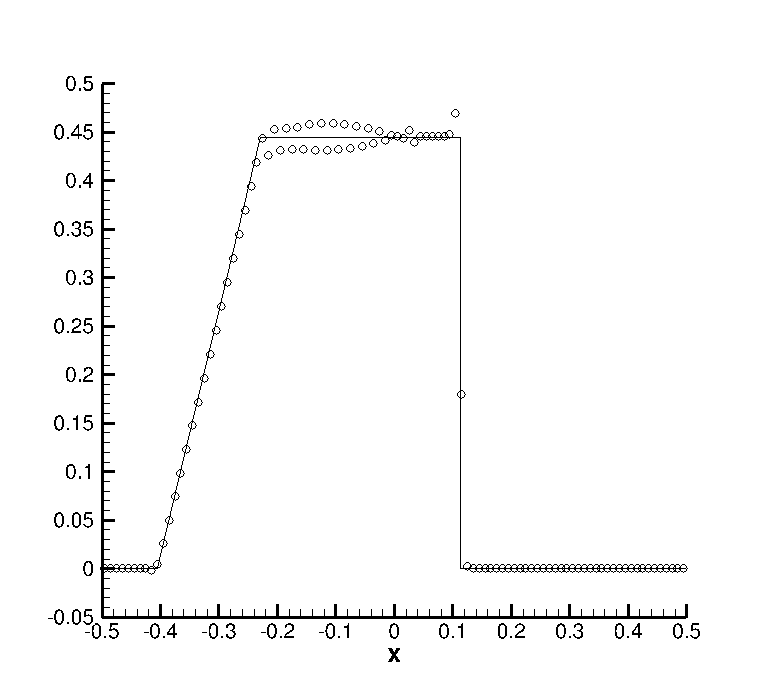 ,width=4.5cm}}
\subfloat{\epsfig{figure=\LOCALFIGPATHRP 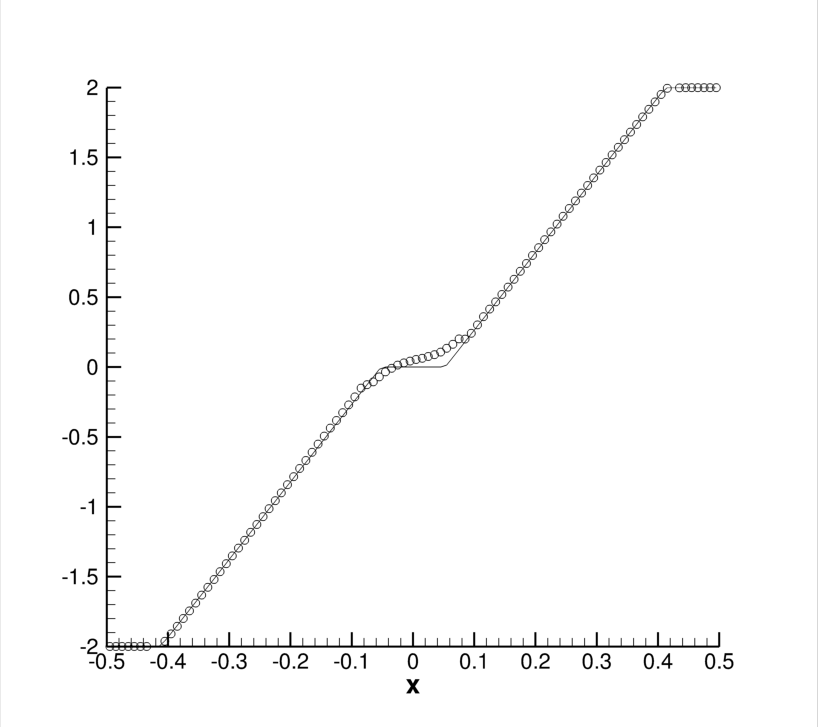 ,width=4.5cm}}\\
%
\subfloat{\begin{picture}(0,0) \put(-8,50){$\mathrm{p}$} \end{picture}}
\setcounter{subfigure}{0}
\subfloat[RP1]{\epsfig{figure=\LOCALFIGPATHRP 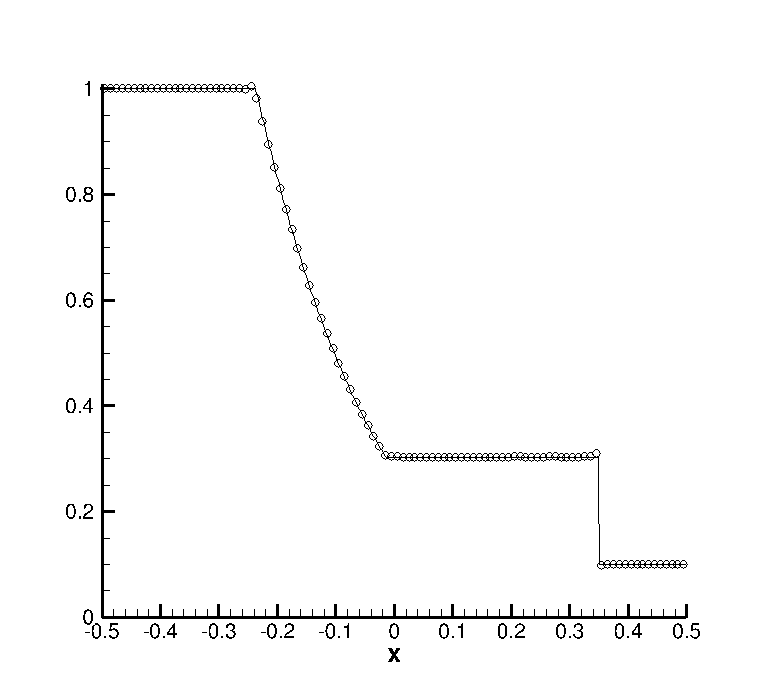 ,width=4.5cm}}
\subfloat[RP2]{\epsfig{figure=\LOCALFIGPATHRP 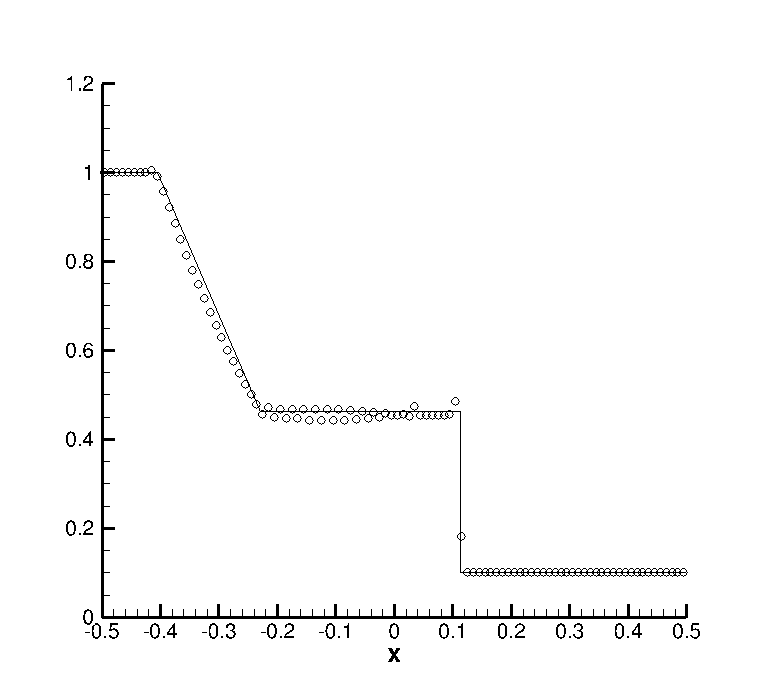 ,width=4.5cm}}
\subfloat[RP3]{\epsfig{figure=\LOCALFIGPATHRP 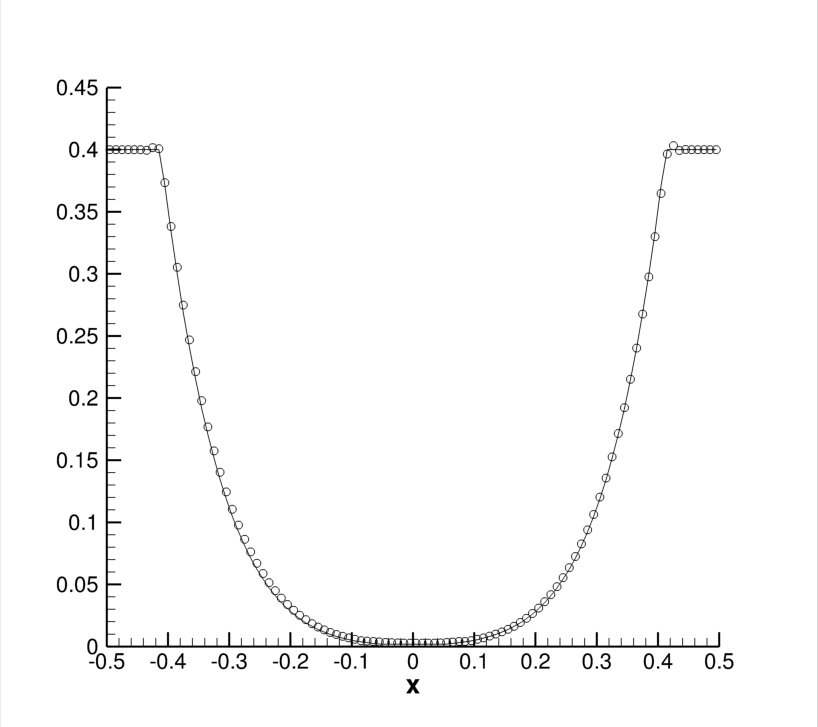 ,width=4.5cm}} 
%
\caption{Riemann problems RP1 to RP3 from \cref{tab:RP_IC} discretized with $p=3$ and $N=100$ cells (in RP2, velocity and pressure have been scaled by factors $10^3$ and $10^6$, respectively).}
\label{fig:solution_RPs}
\end{bigcenter}
\end{figure}
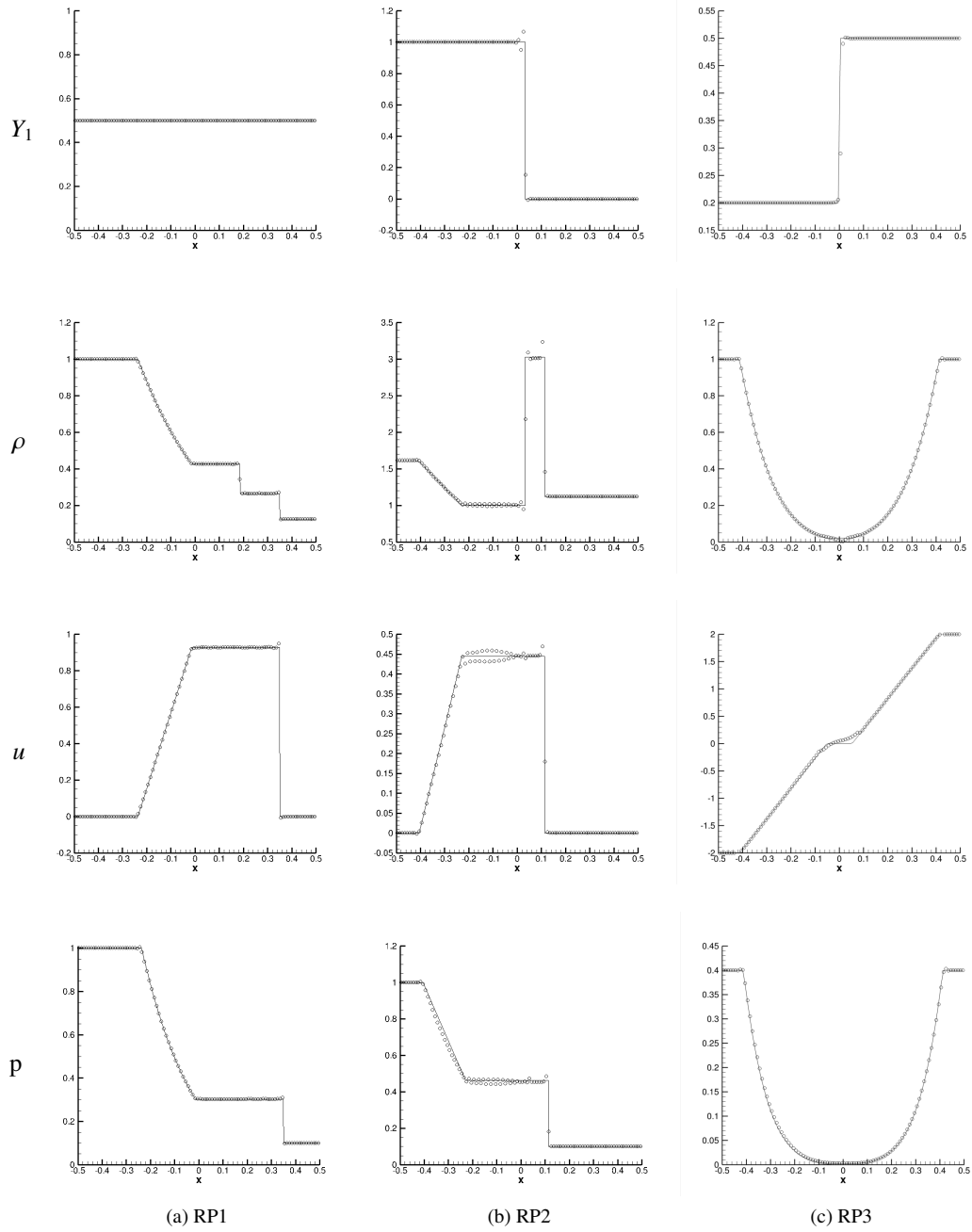

\begin{figure}
\begin{bigcenter}
\subfloat{\begin{picture}(0,0) \put(-19,50){RP4} \end{picture}}
\subfloat{\epsfig{figure=\LOCALFIGPATHRP 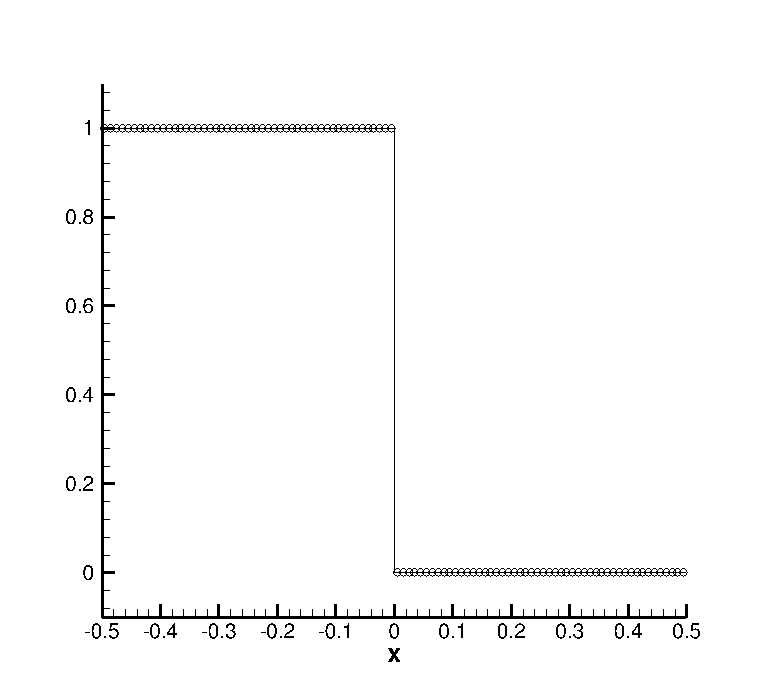 ,width=3.8cm}}
\subfloat{\epsfig{figure=\LOCALFIGPATHRP 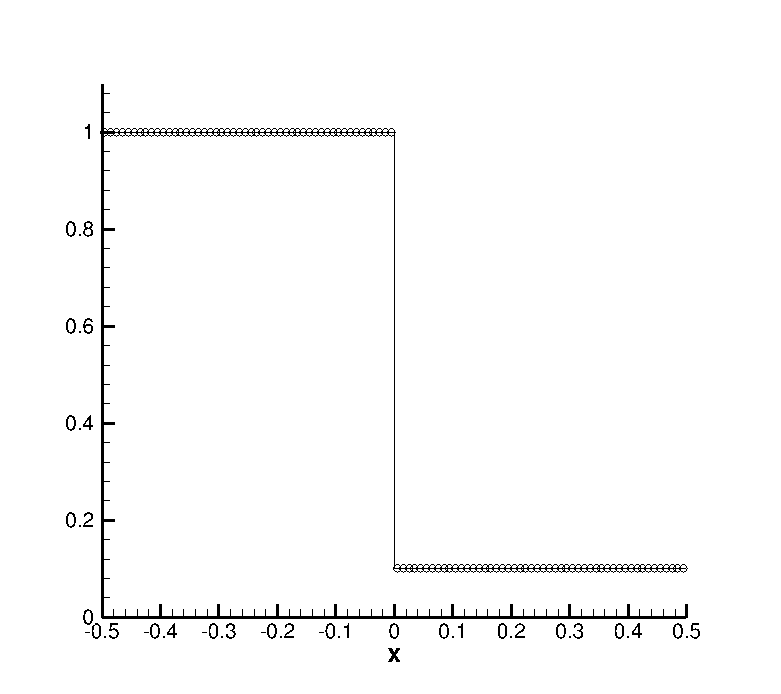 ,width=3.8cm}}
\subfloat{\epsfig{figure=\LOCALFIGPATHRP 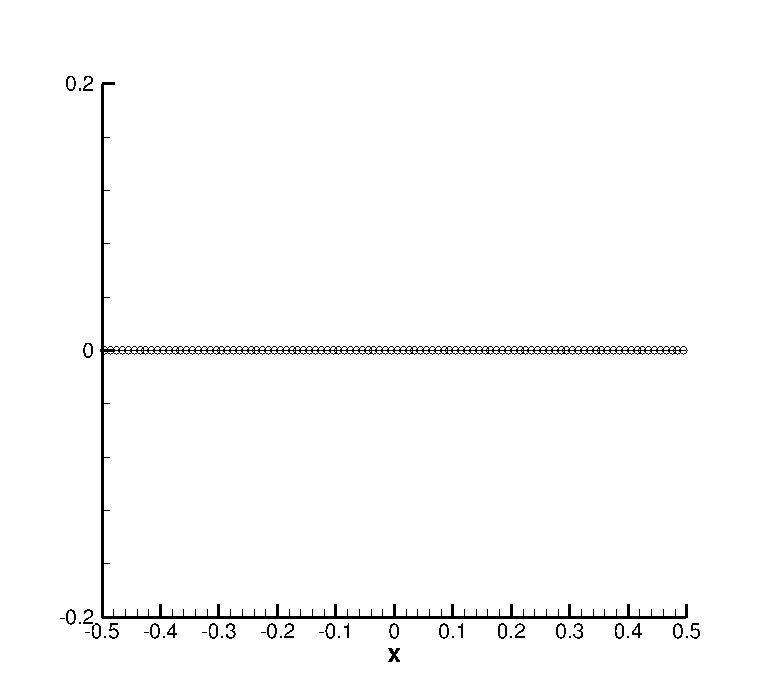   ,width=3.8cm}}
\subfloat{\epsfig{figure=\LOCALFIGPATHRP 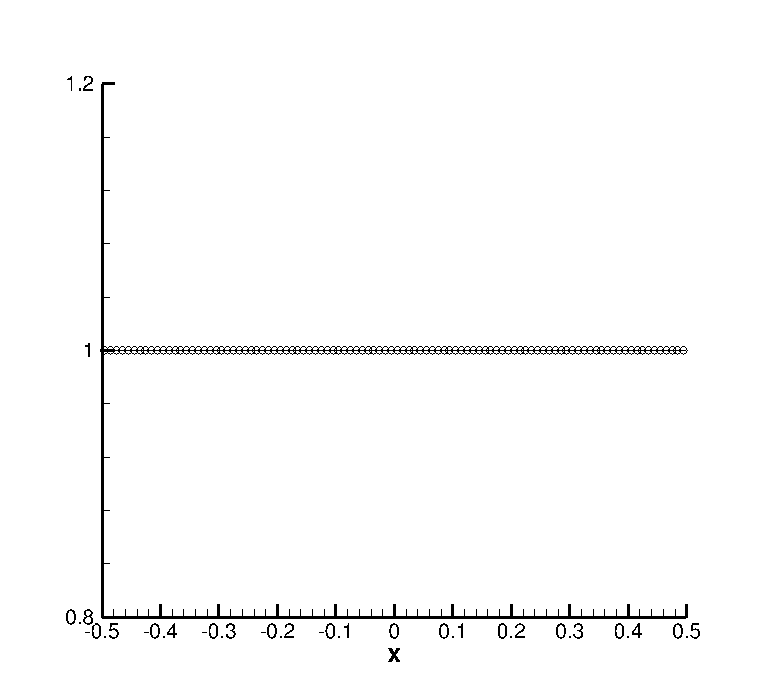   ,width=3.8cm}}\\
\subfloat{\begin{picture}(0,0) \put(-19,50){RP5} \end{picture}}
\setcounter{subfigure}{0}
\subfloat[$Y_1$]{\epsfig{figure=\LOCALFIGPATHRP 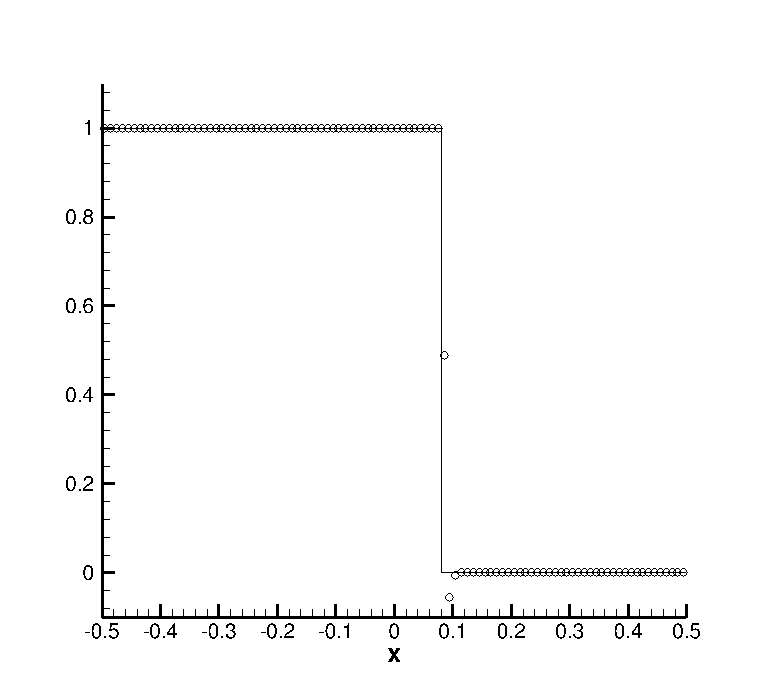  ,width=3.8cm}}
\subfloat[$\rho$]{\epsfig{figure=\LOCALFIGPATHRP 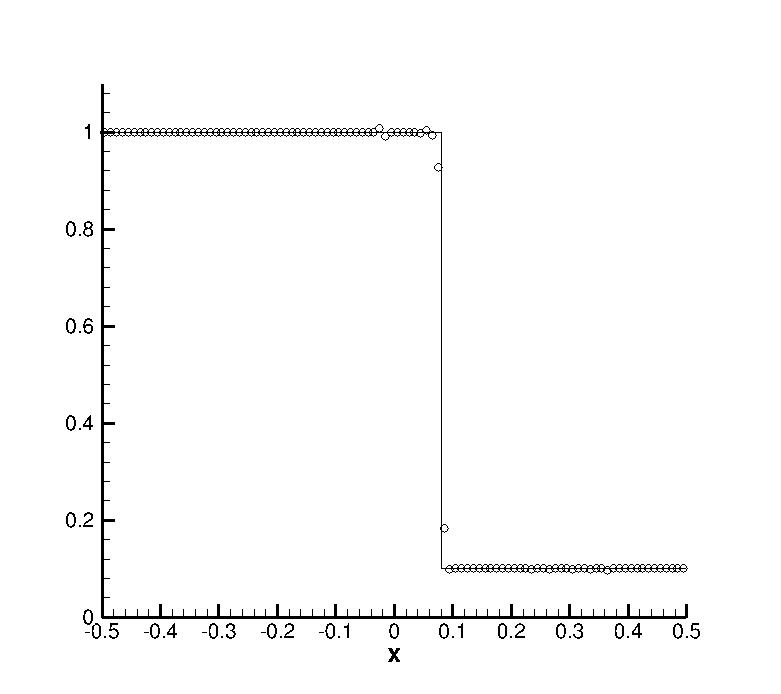 ,width=3.8cm}}
\subfloat[$u$]{\epsfig{figure=\LOCALFIGPATHRP 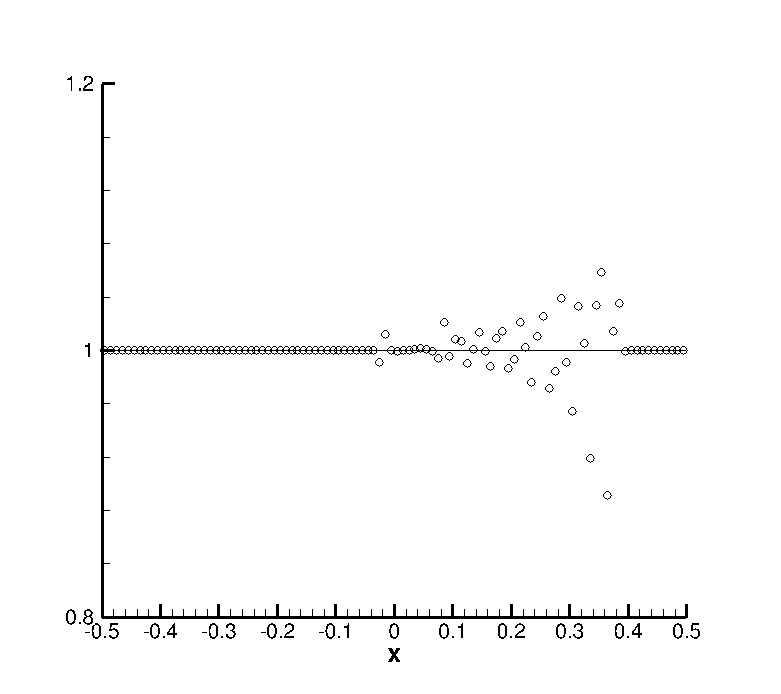      ,width=3.8cm}}
\subfloat[$\mathrm{p}$]{\epsfig{figure=\LOCALFIGPATHRP 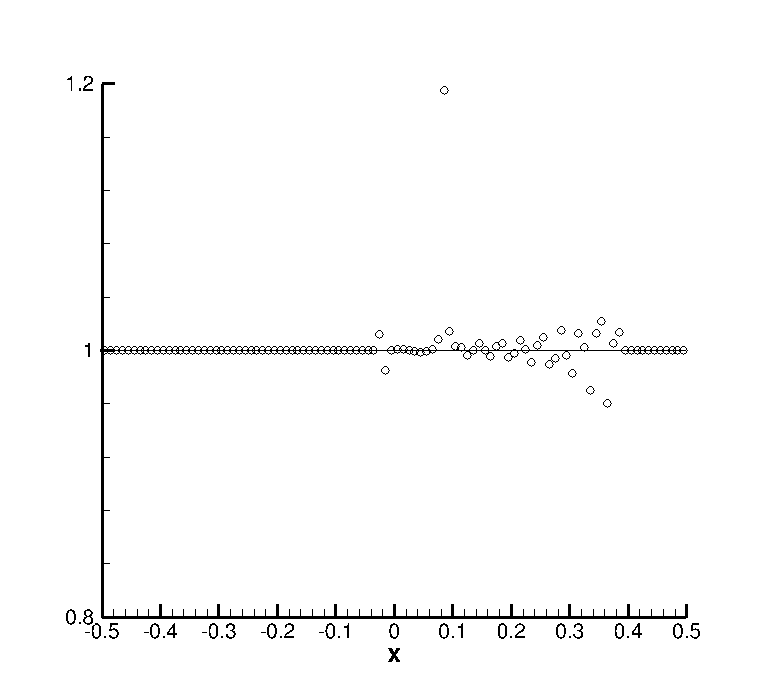 ,width=3.8cm}}
\caption{Riemann problems from \cref{tab:RP_IC} on isolated stationary (top) and moving (bottom) interfaces discretized with $p=3$ and $N=100$ cells.}
\label{fig:solution_RPs_contact}
\end{bigcenter}
\end{figure}

\subsection{Shock wave-helium bubble interaction}\label{sec:SBI_KT}

We now consider the interaction of a shock with a helium bubble \cite{haas-sturtenvant-87} which is commonly used to assess the resolution by numerical schemes of shock waves, material interfaces and their interaction in multiphase and multicomponent flows (see \cite{kawai_terashima_11,fedkiw_etal_99,quirk_karni_96,capuano_etal_18} and references therein).

The domain extends to $\Omega=[0,6.5]\times[0,0.89]$. A left moving $M=1.22$ normal shock wave in air is initially located at $x=4.5$ and interacts with a bubble of helium of unit diameter with center located at $x=3.5$ and $y=0$. Symmetry conditions are set to the top and bottom boundaries, while non reflecting conditions are applied to the left and right limits of the domain. The thermodynamical parameters of helium and air are $\gamma_1=1.648$, $C_{v_1}=6.89$ and $\gamma_2=1.4$, $C_{v_2}=1.7857$, respectively. Data are made nondimensional with the initial bubble diameter and pre-shock density, temperature and sound speed. We use an unstructured mesh with $N=238,673$ elements (see \Cref{fig:mesh_SBI}a). The complete setup of the initial condition can be found in \cite{kawai_terashima_11}. Note that this test case is usually computed including viscous effects. To avoid spurious oscillations at material interfaces in inviscid computations we regularize the initial condition of the bubble-air interface following \cite{billet_etal_08,kawai_terashima_11,houim_kuo_11}.

\begin{figure}
\begin{bigcenter}
  \subfloat[]{\epsfig{figure=\LOCALFIGPATHSBIKT 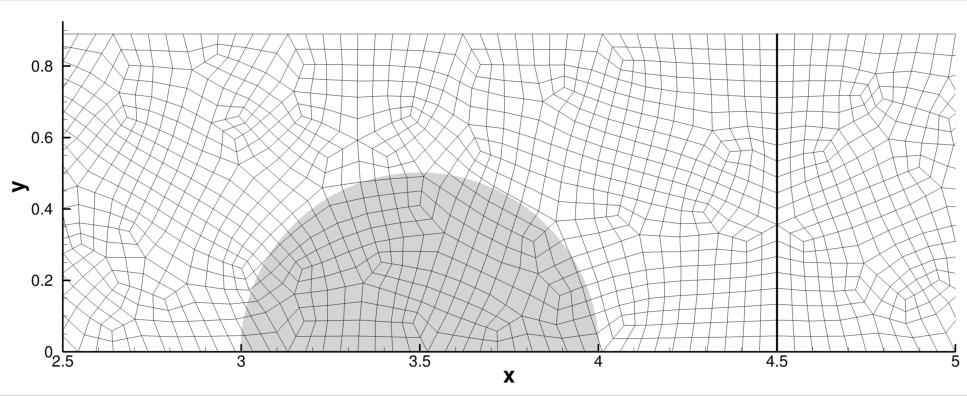,height=3.2cm,clip=true,trim=0.5cm 0.05cm 0cm 0.05cm}}
  \subfloat[]{\epsfig{figure=\LOCALFIGPATHSBIKT 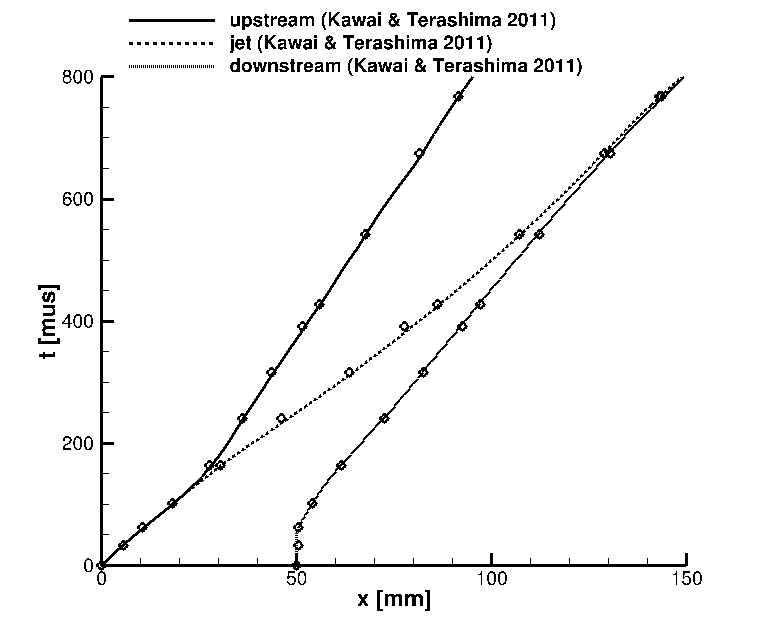,height=3.4cm}}
  \subfloat[]{\epsfig{figure=\LOCALFIGPATHSBIKT 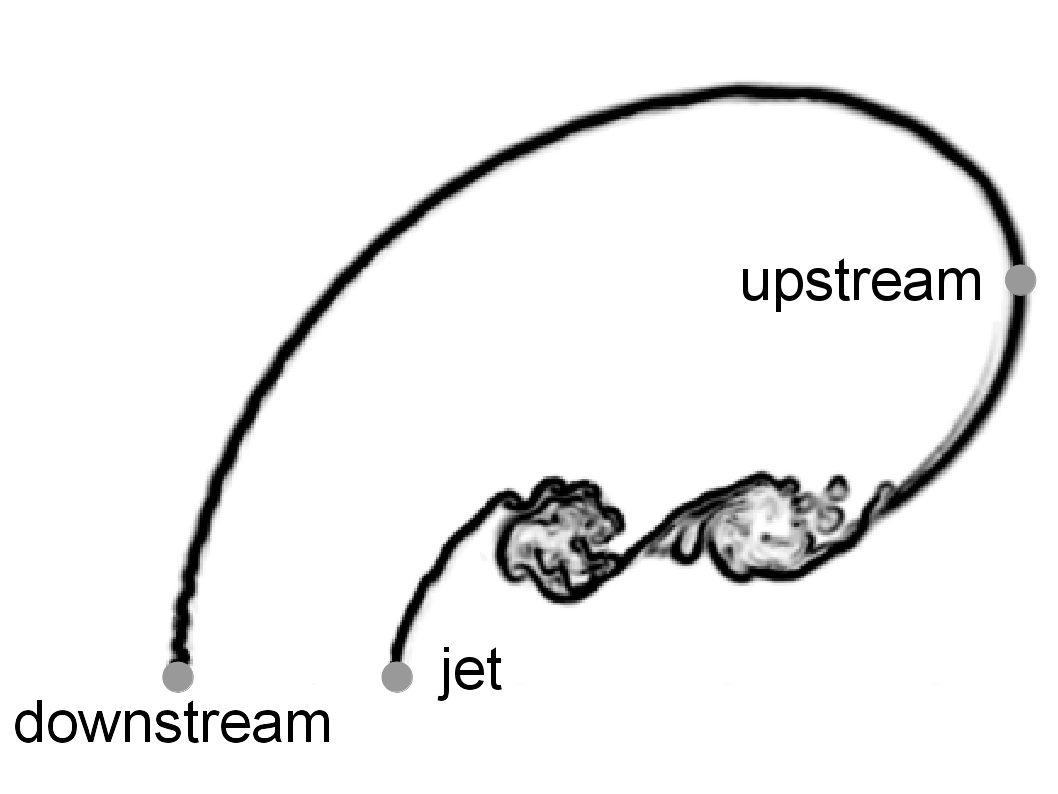,height=2.6cm}}
  \caption{Shock wave-helium bubble interaction: (a) example of unstructured mesh in the range $2.5\leq x\leq5$ with $N=790$ elements and initial positions of the shock and bubble; (b) space-time diagram of positions of three characteristic interface points (symbols) with comparison to \cite{kawai_terashima_11} (lines); (c) illustration of the characteristic interface points (gray filled circles).}
\label{fig:mesh_SBI}
\end{bigcenter}
\end{figure}

\Cref{fig:solution_SBI} displays contours of pressure, void fraction and numerical Schlieren obtained at different times initialized when the shock reaches the bubble. The shock and material interfaces are well resolved and the solution does not present significant spurious oscillations. The results are in good qualitative agreement with the experiment in \cite{haas-sturtenvant-87} and numerical simulations (see e.g., \cite{kawai_terashima_11}). In particular the shock dynamics and the bubble deformation are well reproduced, and vortices are generated along the bubble interface due to the Kelvin-Helmholtz instability. \Cref{fig:mesh_SBI} displays the $(x,t)$-positions of three characteristic interface points in close agreement with the results from \cite{kawai_terashima_11}. 


\begin{figure}
\begin{bigcenter}
\captionsetup[subfigure]{labelformat=empty}
\subfloat[$t=0\;\mu$s]{\epsfig{figure=\LOCALFIGPATHSBIKT 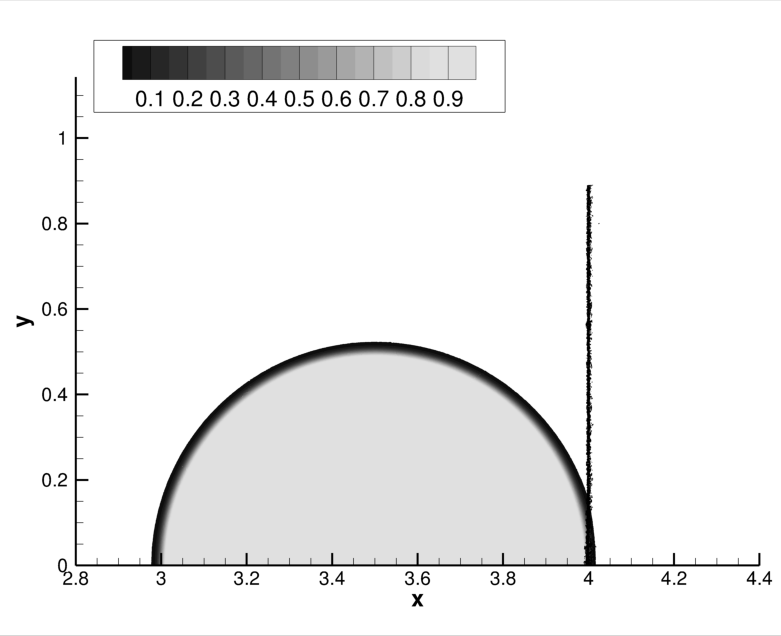,width=3.8cm}}
\subfloat{\epsfig{figure=\LOCALFIGPATHSBIKT 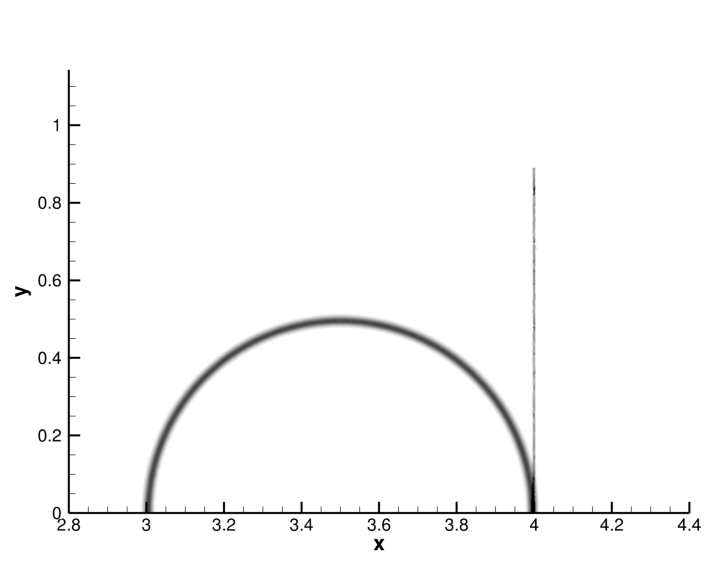,width=3.8cm}}
\subfloat[$t=240\;\mu$s]{\epsfig{figure=\LOCALFIGPATHSBIKT 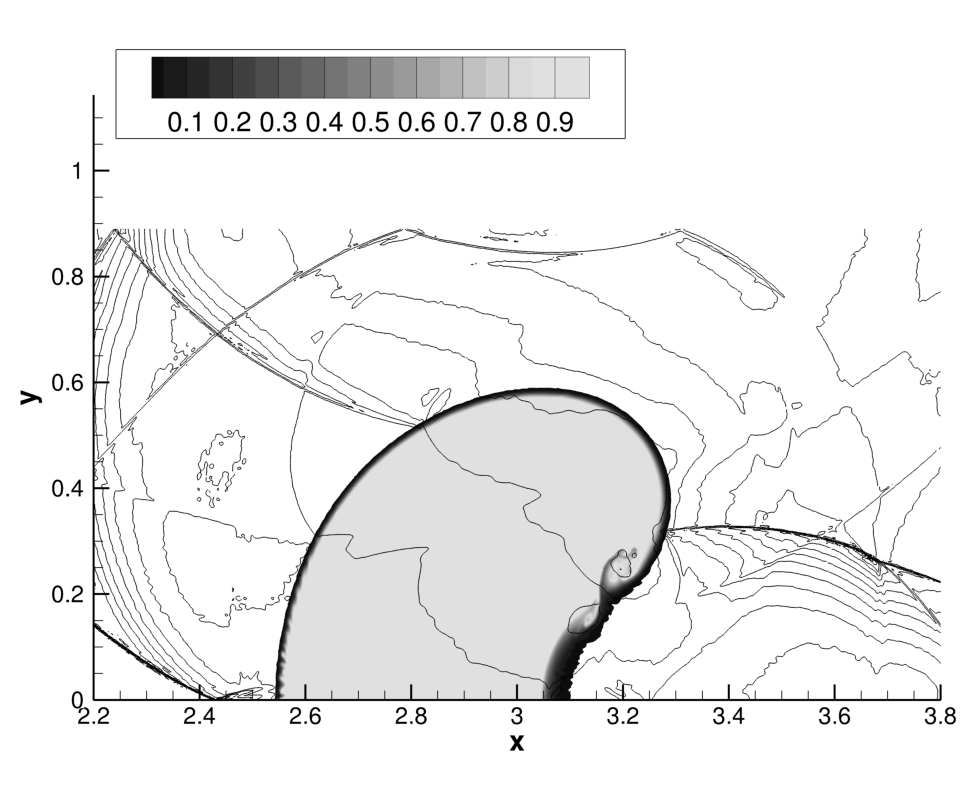,width=3.8cm}}
\subfloat{\epsfig{figure=\LOCALFIGPATHSBIKT 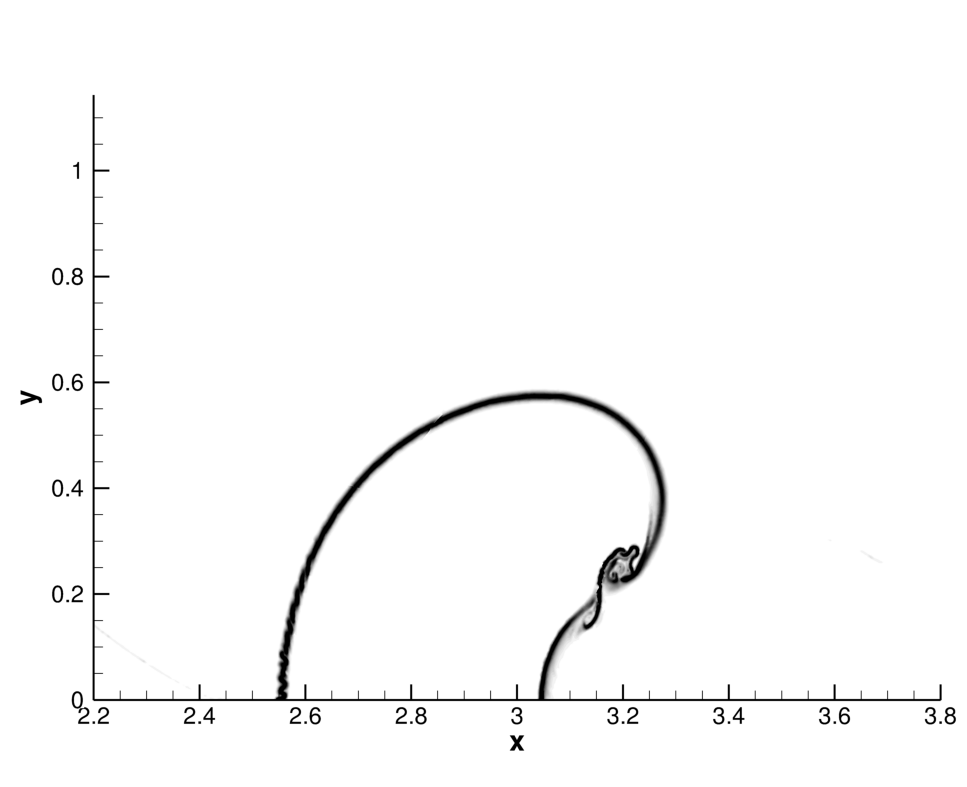,width=3.8cm}}\\
\subfloat[$t=32\;\mu$s]{\epsfig{figure=\LOCALFIGPATHSBIKT 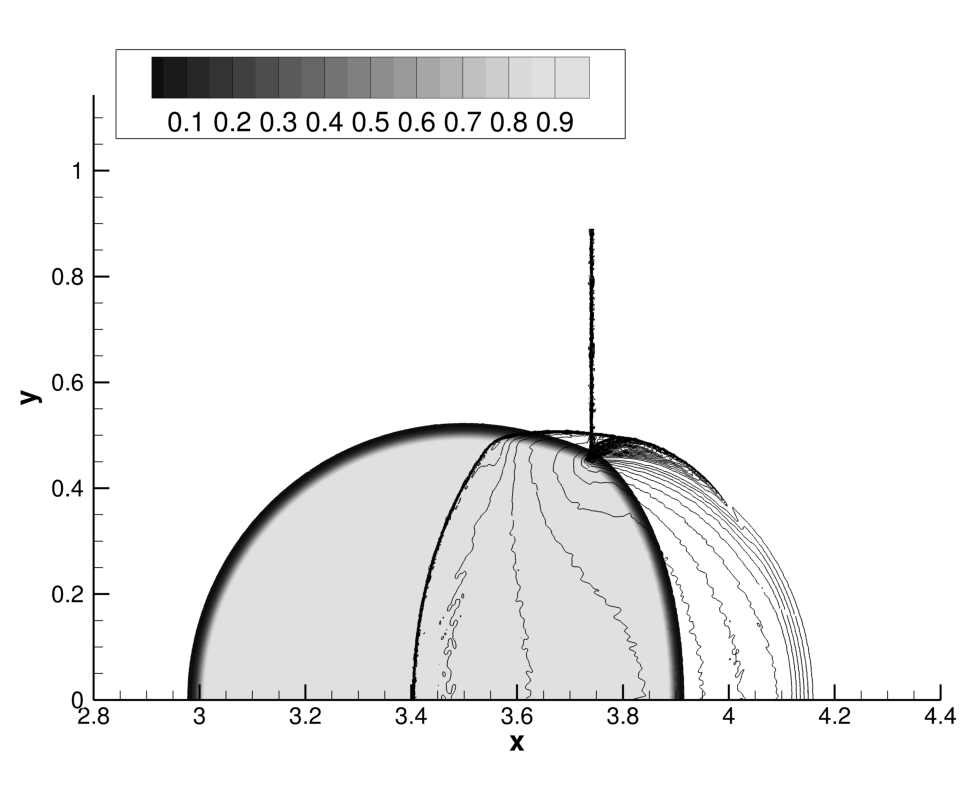,width=3.8cm}}
\subfloat{\epsfig{figure=\LOCALFIGPATHSBIKT 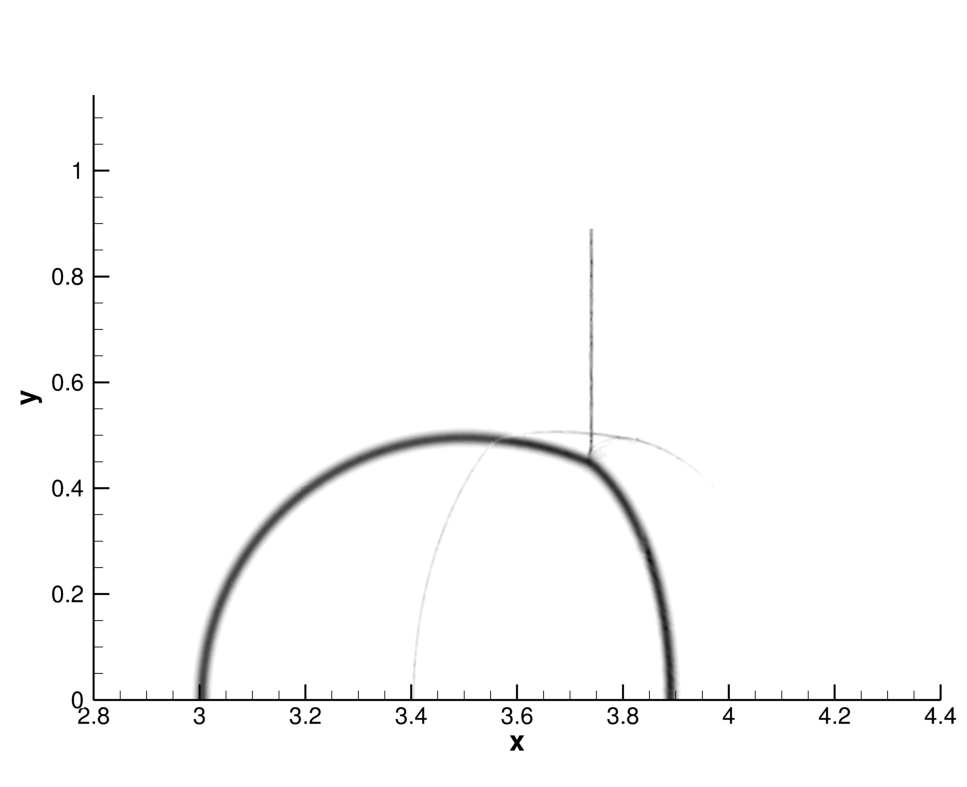,width=3.8cm}}
\subfloat[$t=427\;\mu$s]{\epsfig{figure=\LOCALFIGPATHSBIKT 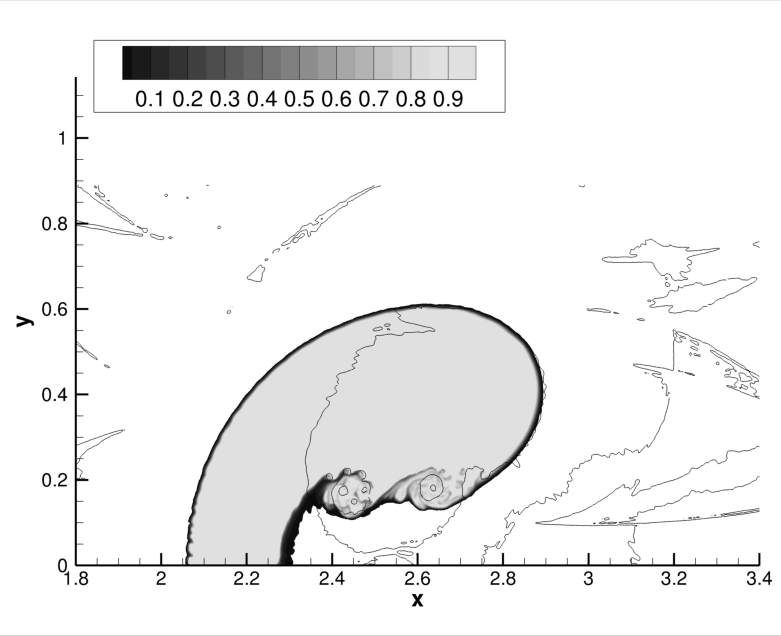,width=3.8cm}}
\subfloat{\epsfig{figure=\LOCALFIGPATHSBIKT 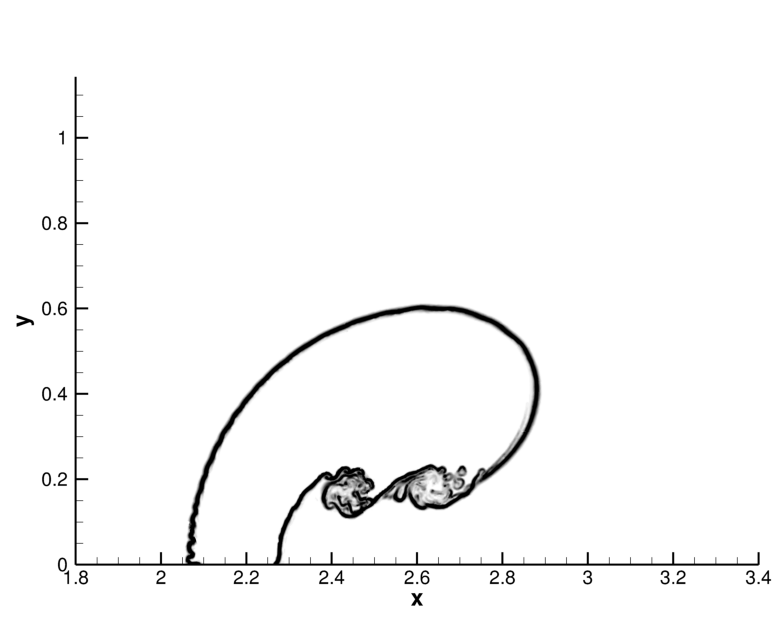,width=3.8cm}}\\
\subfloat[$t=62\;\mu$s]{\epsfig{figure=\LOCALFIGPATHSBIKT 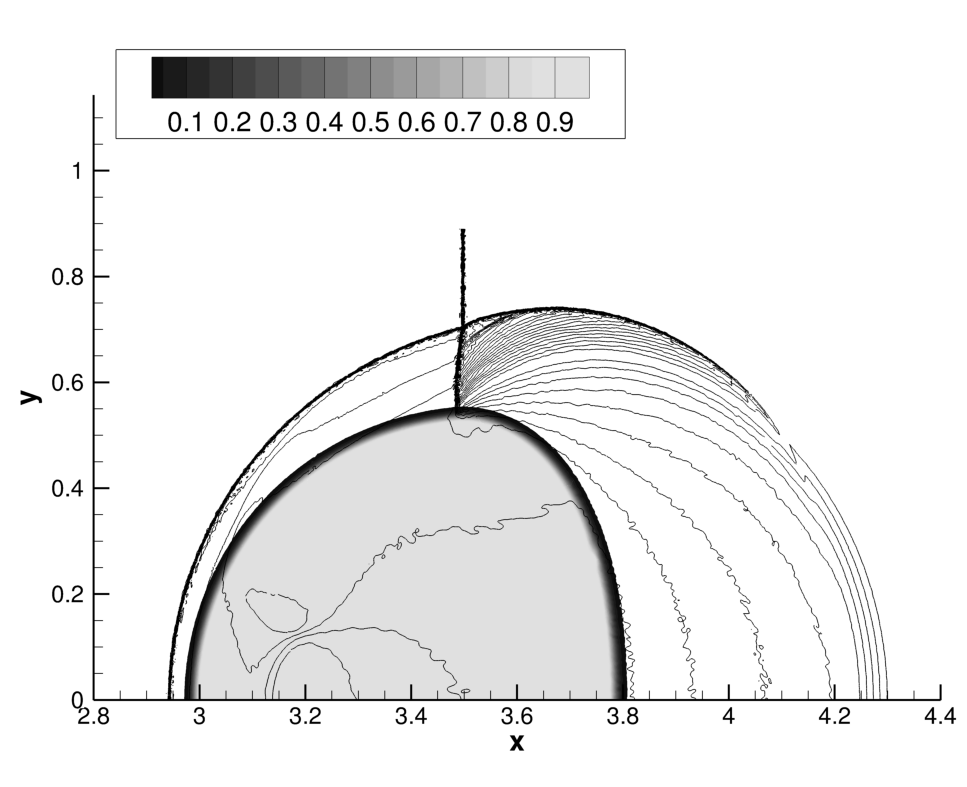,width=3.8cm}}
\subfloat{\epsfig{figure=\LOCALFIGPATHSBIKT 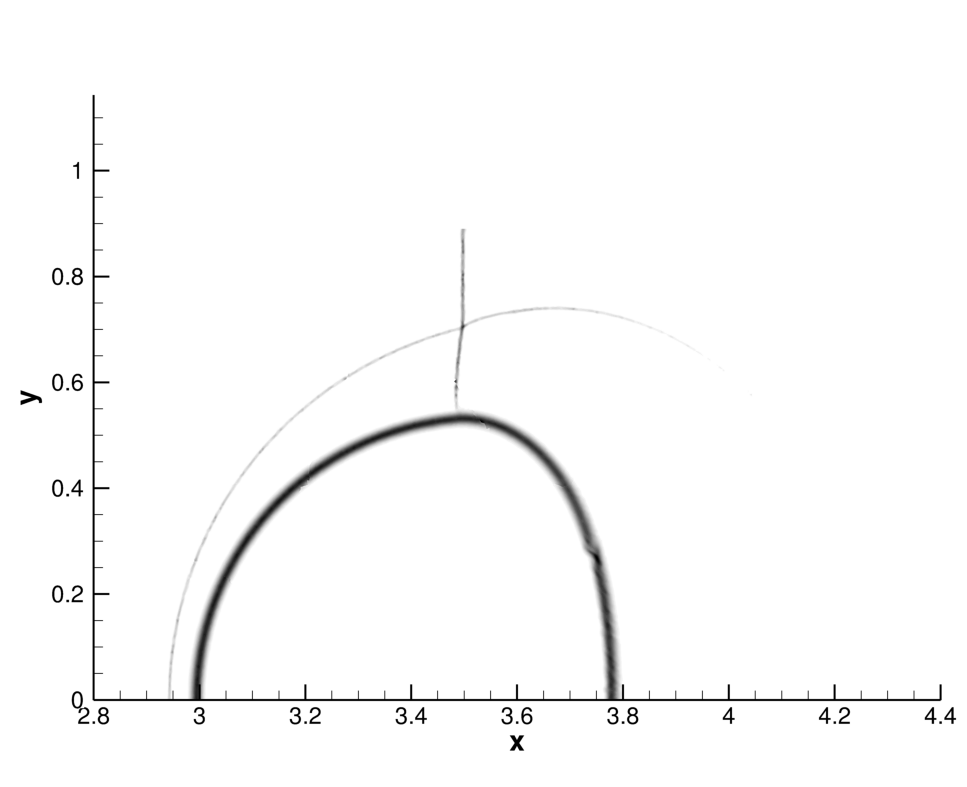,width=3.8cm}}
\subfloat[$t=674\;\mu$s]{\epsfig{figure=\LOCALFIGPATHSBIKT 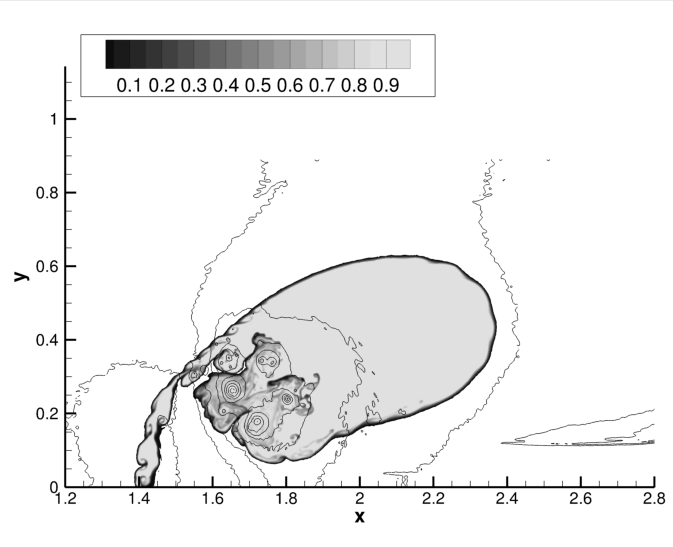,width=3.8cm}}
\subfloat{\epsfig{figure=\LOCALFIGPATHSBIKT 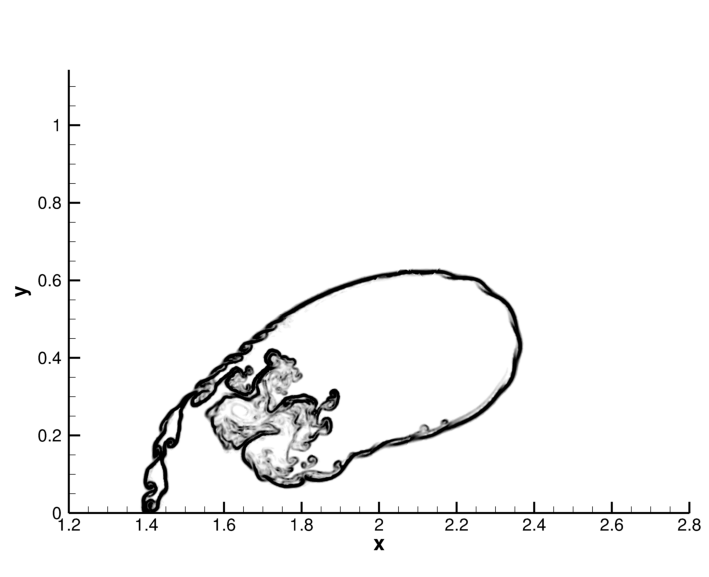,width=3.8cm}}\\
\subfloat[$t=102\;\mu$s]{\epsfig{figure=\LOCALFIGPATHSBIKT 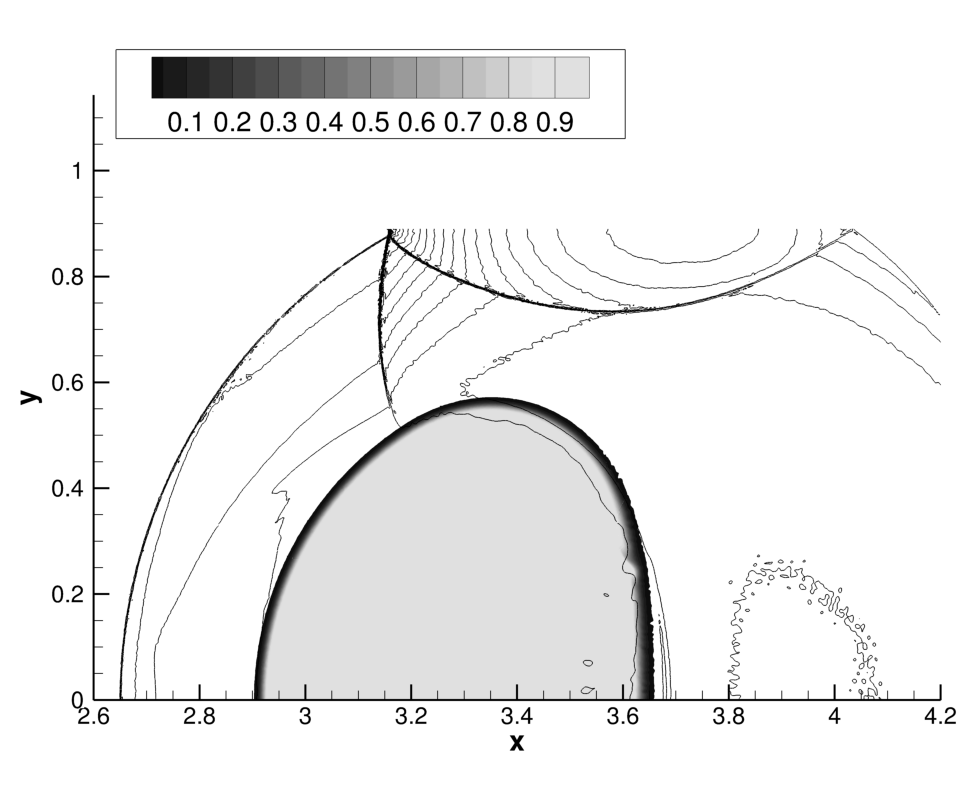,width=3.8cm}}
\subfloat{\epsfig{figure=\LOCALFIGPATHSBIKT 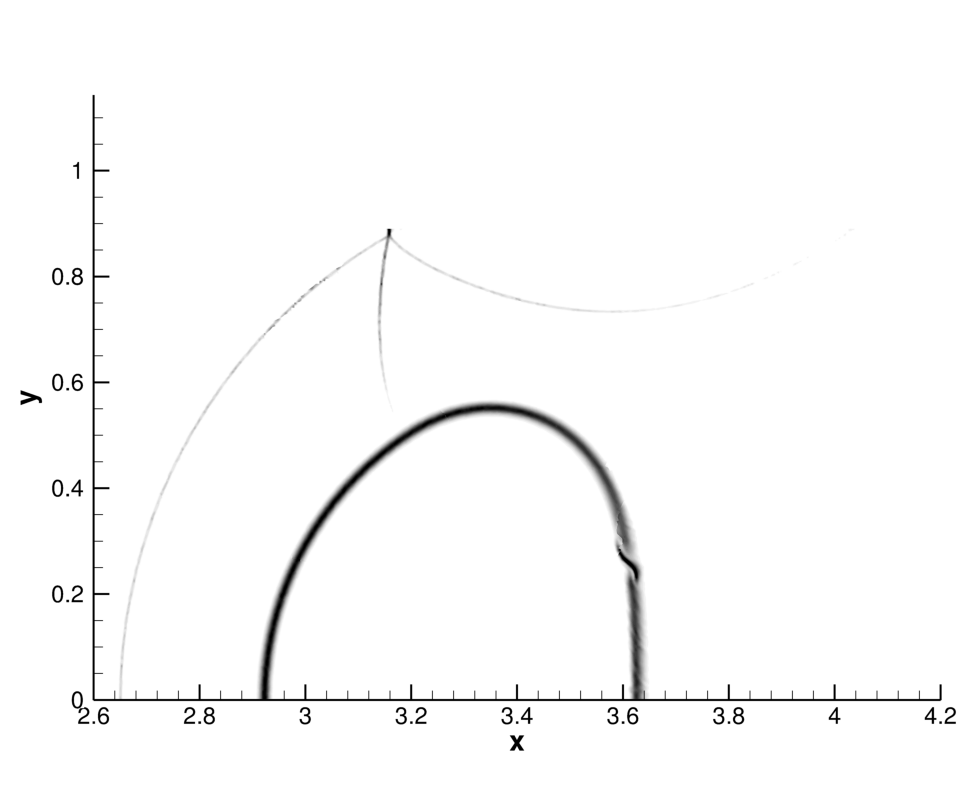,width=3.8cm}}
\subfloat[$t=983\;\mu$s]{\epsfig{figure=\LOCALFIGPATHSBIKT 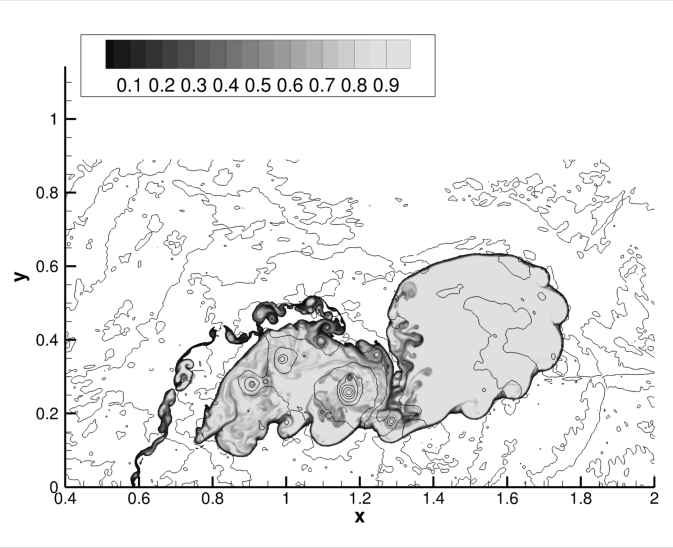,width=3.8cm}}
\subfloat{\epsfig{figure=\LOCALFIGPATHSBIKT 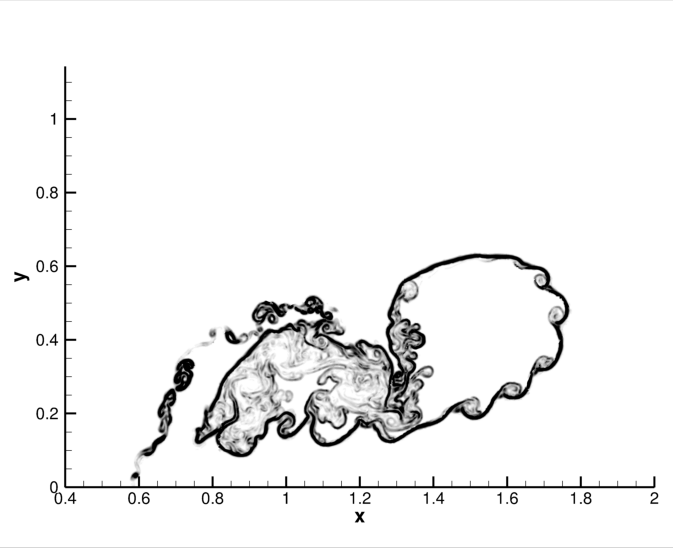,width=3.8cm}}
\caption{Shock wave-helium bubble interaction: $70$ pressure contours from $0.66$ to $1.29$ (lines) and $20$ void fraction contours (colors, see legend), and Schlieren $|\nabla\rho|/\rho$ obtained with $p=3$ and $N=238,673$ cells at different physical times.}
\label{fig:solution_SBI}
\end{bigcenter}
\end{figure}

\subsection{Richtmyer-Meshkov instability}\label{sec:RMI}

We also simulate the interaction of a Mach $1.21$ shock wave in a mixture of air and acetone vapor with a  perturbed interface separating the mixture from a dense SF$6$ gas \cite{brouillette_sturtevant_94}. The complete setup may be found in \cite{houim_kuo_11,mohaved_johnsen_13,latini_etal_07,capuano_etal_18}. The thermodynamical parameters of the mixture and SF6 are $\gamma_1=1.24815$, $C_{v_1}=3.2286$ and $\gamma_2=1.0984$, $C_{v_2}=$2.0019, respectively. The Atwood number of the initial state is $At=\tfrac{\rho_2-\rho_1}{\rho_1+\rho_2}=0.6053$ where $\rho_2=\gamma_2$ is taken as the pre-shock density of the mixture. Data are made nondimensional with a length scale of $1$cm, and the pre-shock pressure, temperature and sound speed of the mixture. The size of the domain is $\Omega=[0,80.1]\times[0,5.9]$, periodicity conditions are set to the top and bottom boundaries, while non reflecting and reflecting conditions are applied to the left and right boundaries, respectively. We use a Cartesian grid with $1601\times118$ elements which corresponds to $118$ elements per perturbation wavelength and constitutes a coarse mesh compared to other experiments \cite{houim_kuo_11,mohaved_johnsen_13,capuano_etal_18}.

We consider the single-mode perturbation of the material interface \cite{brouillette_sturtevant_94}. The shock travels to the right and interacts with the interface, then reflects at the right boundary and interacts a second time with the interface (re-shock regime). \Cref{fig:solution_RMI} shows results before and after re-shock where the density and vorticity contours are displayed. The first interaction produces vorticity at the interface and the formation and roll-up of spikes, while the second interaction with the reflected shock wave produces complex fine flow field structures and a low Mach number flow field. Again, we observe a good resolution of the shock and material interface and the associated vortical structures.

\begin{figure}
\begin{bigcenter}
\subfloat{\epsfig{figure=\LOCALFIGPATHRMI 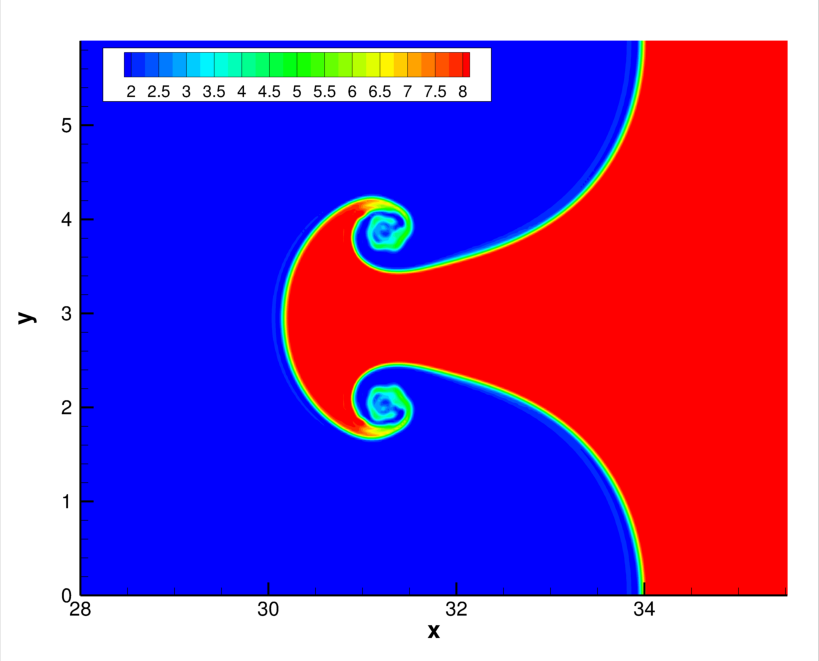,width=3.8cm}}
\subfloat{\epsfig{figure=\LOCALFIGPATHRMI 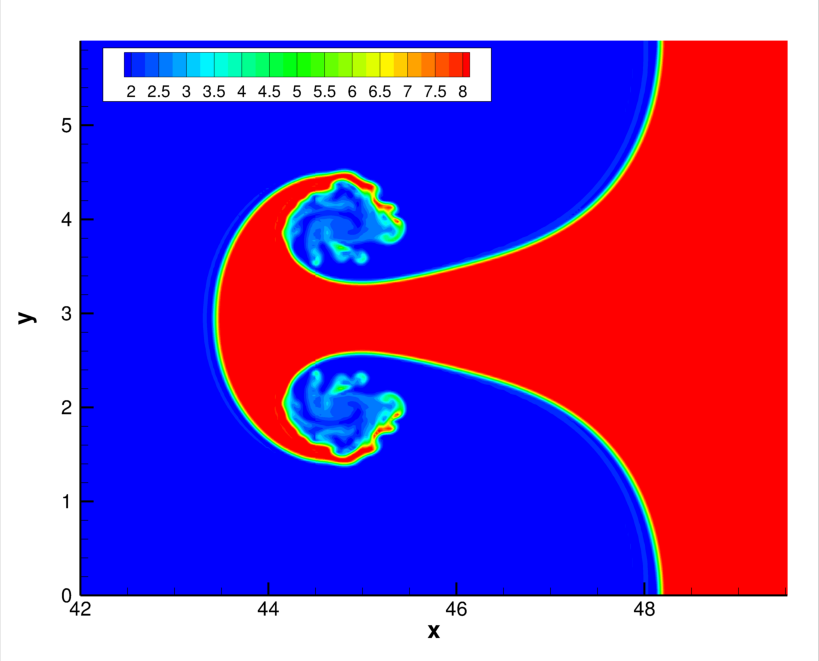,width=3.8cm}}
\subfloat{\epsfig{figure=\LOCALFIGPATHRMI 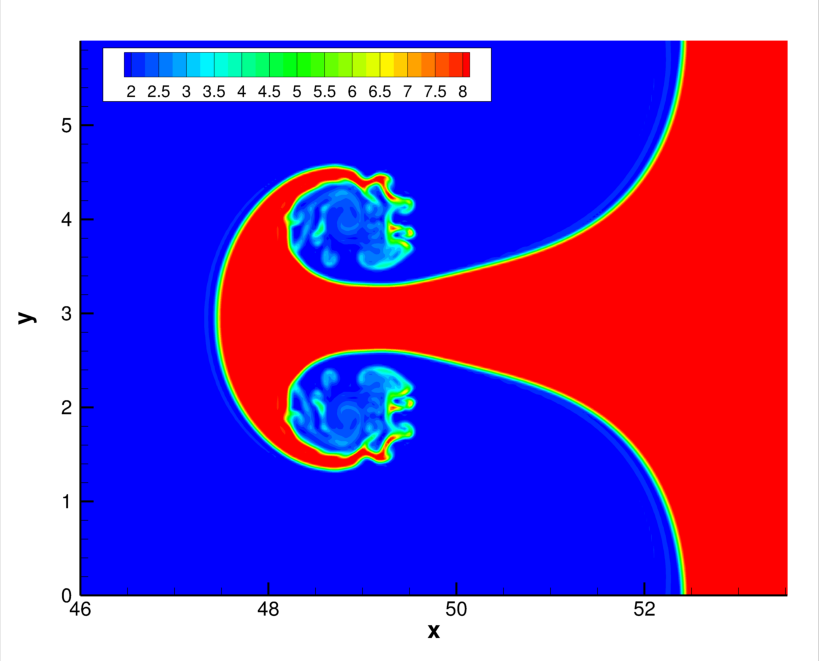,width=3.8cm}}
\subfloat{\epsfig{figure=\LOCALFIGPATHRMI 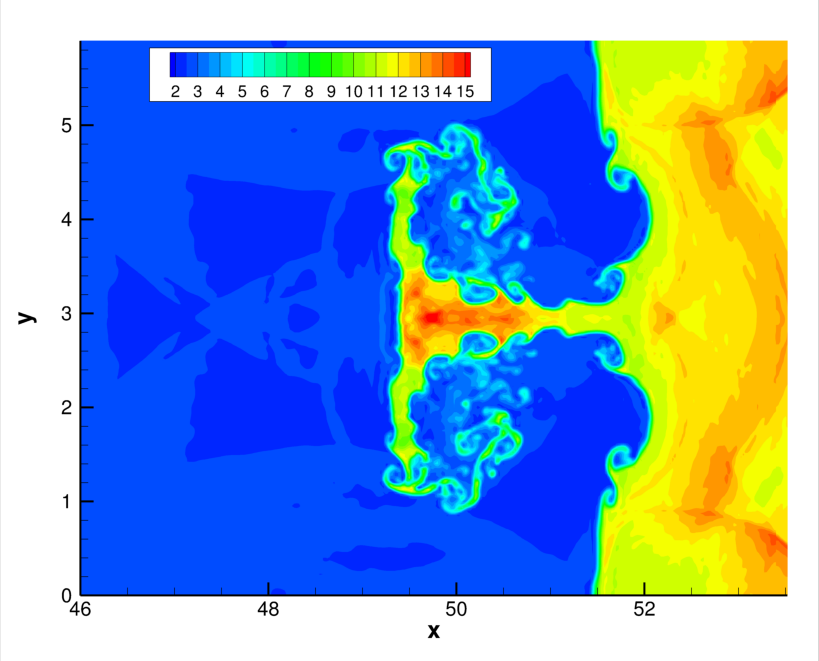,width=3.8cm}} \\
\setcounter{subfigure}{0}
\subfloat[$t=4\mu s$]{\epsfig{figure=\LOCALFIGPATHRMI 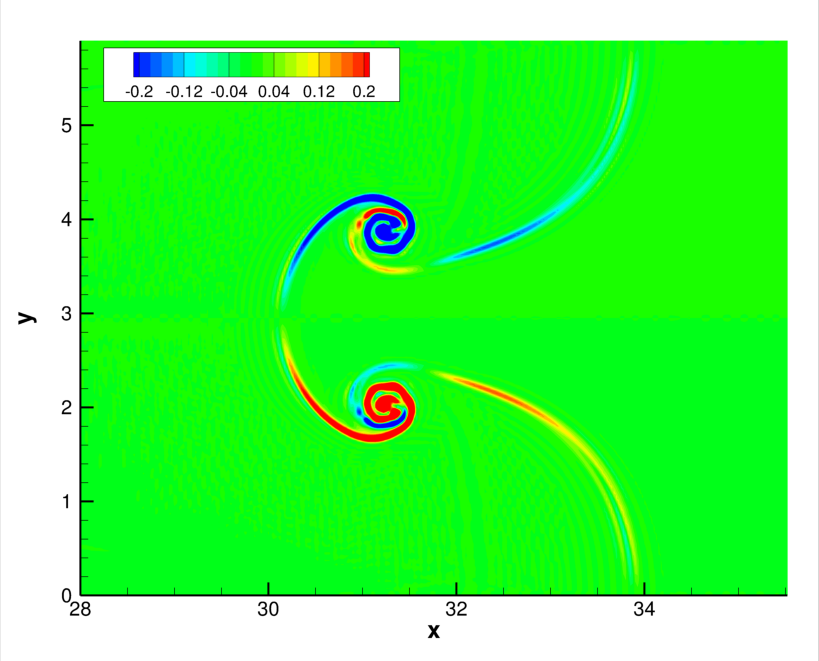,width=3.8cm}}
\subfloat[$t=6\mu s$]{\epsfig{figure=\LOCALFIGPATHRMI 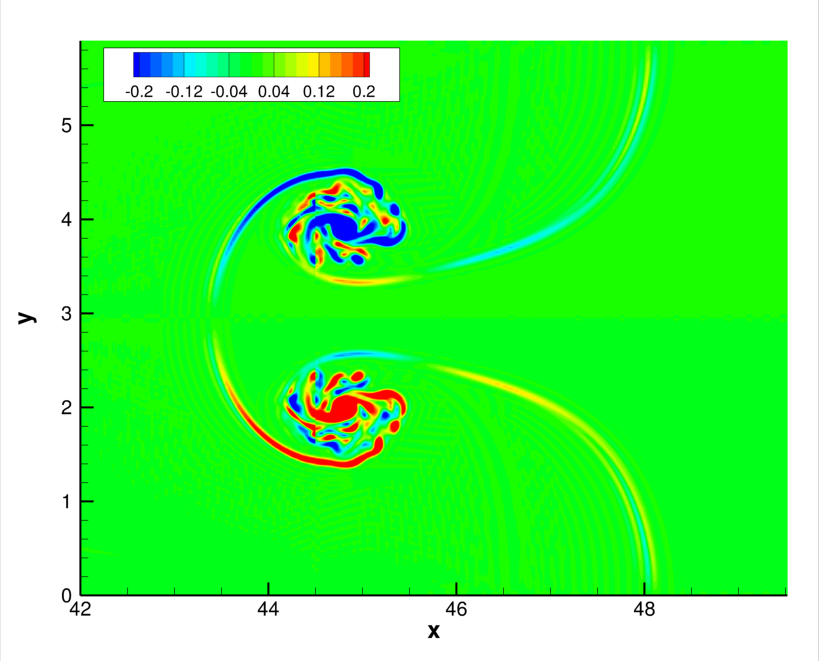,width=3.8cm}}
\subfloat[$t=6.6\mu s$]{\epsfig{figure=\LOCALFIGPATHRMI 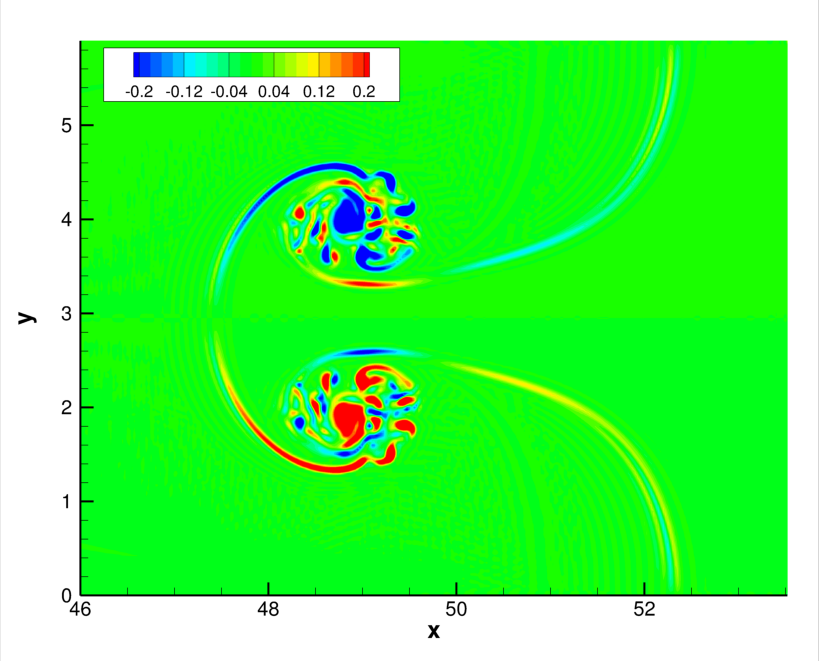,width=3.8cm}}
\subfloat[$t=7\mu s$]{\epsfig{figure=\LOCALFIGPATHRMI 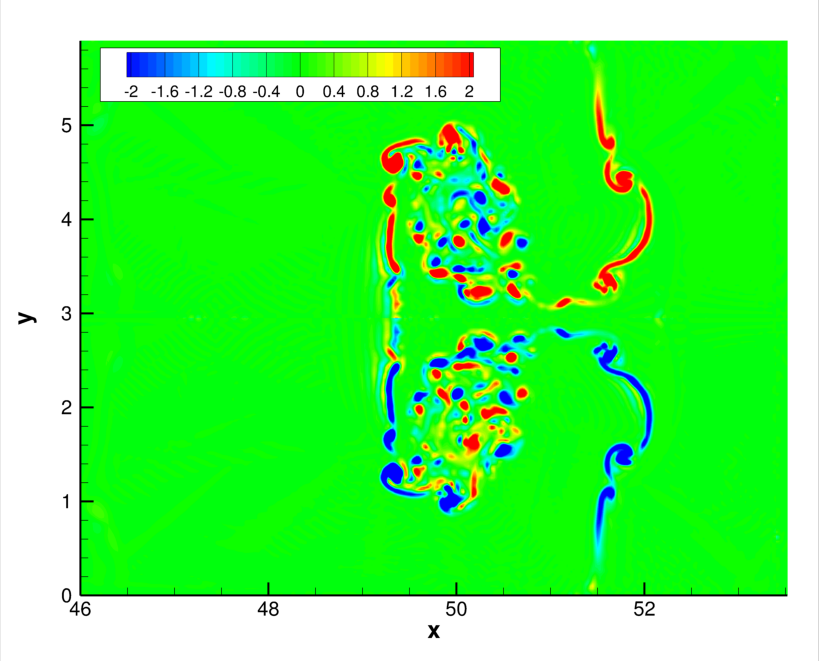,width=3.8cm}}
\caption{Richtmyer-Meshkov instability: $25$ density $\rho$ (top) and $21$ vorticity $\nabla\times{\bf v}$ (bottom) contours before (a-c) and after (d) re-shock obtained with a polynomial degree $p=3$ and a $h=1/20$ mesh at different physical times.}
\label{fig:solution_RMI}
\end{bigcenter}
\end{figure}

\subsection{Strong shock wave-hydrogen bubble interaction}\label{sec:SBI}

We finally consider an interaction problem of a planar shock in air with a circular hydrogen bubble that has been numerically investigated in \cite{SJOGREEN_yee_sbi_03,DON_sbi_95,billet_etal_08}. Compared to the test case in \cref{sec:SBI_KT}, the shock is stronger with a Mach number $M=2$ which results in faster shock and bubble dynamics. The bubble has a diameter of $5.6$ and is centered at $x=4$ and $y=0$ (see \cref{fig:solution_strong_SBI}), it is filled with hydrogen with $\gamma_1=1.41$ and $C_{v_1}=7.424$ and is surrounded by air with $\gamma_2=1.353$, $C_{v_2}=0.523$. The shock is initially set at $x=7$ with pre-shock conditions $\T=1000$K, $\p=1$atm, and $u=1240$m$/$s. Data are made nondimensional with pre-shock density, velocity and temperature, and $1$mm as length scale. The domain extends to $\Omega=[0,22.5]\times[0,7.5]$ and we use a coarse unstructured mesh with $N=74,504$ elements. Symmetry conditions are set to the top and bottom boundaries, while supersonic inflow is imposed at the left boundary and non reflecting conditions are applied at the right boundary. In \cref{fig:solution_SBI} we plot contours of pressure, void fraction and numerical Schlieren obtained at different times. Results compare well with other numerical experiments \cite{SJOGREEN_yee_sbi_03,DON_sbi_95,billet_etal_08} and highlight robustness and accuracy of the present method on unstructured meshes.

\begin{figure}
\begin{bigcenter}
\captionsetup[subfigure]{labelformat=empty}
\subfloat[$t=0\;\mu$s]{\epsfig{figure=\LOCALFIGPATHSBINS 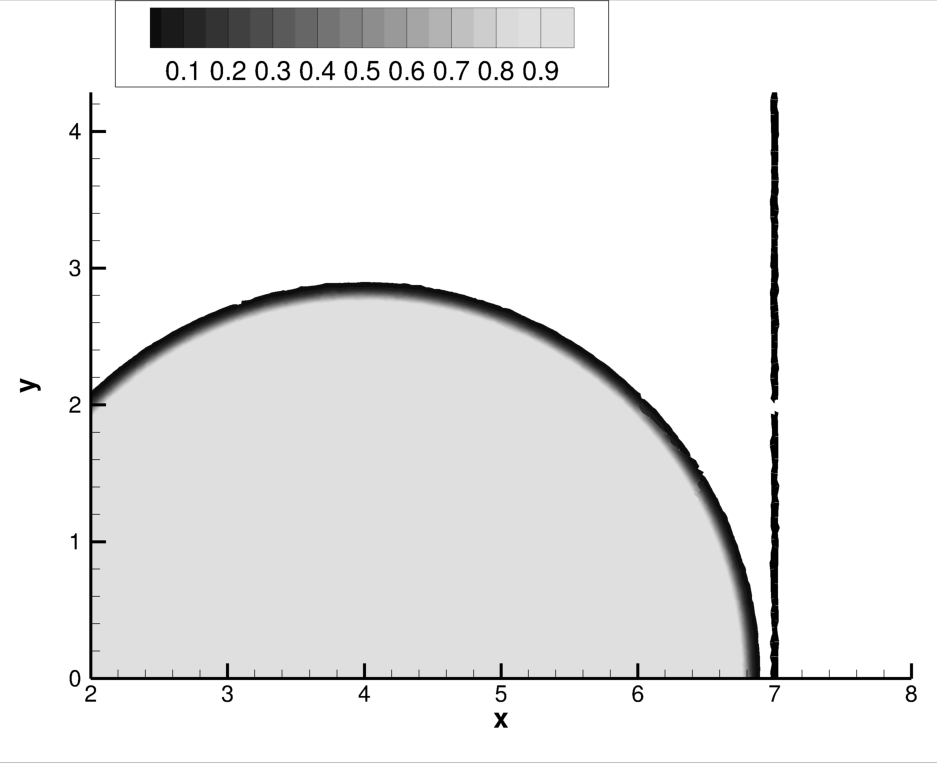,width=3.8cm}}
\subfloat             {\epsfig{figure=\LOCALFIGPATHSBINS 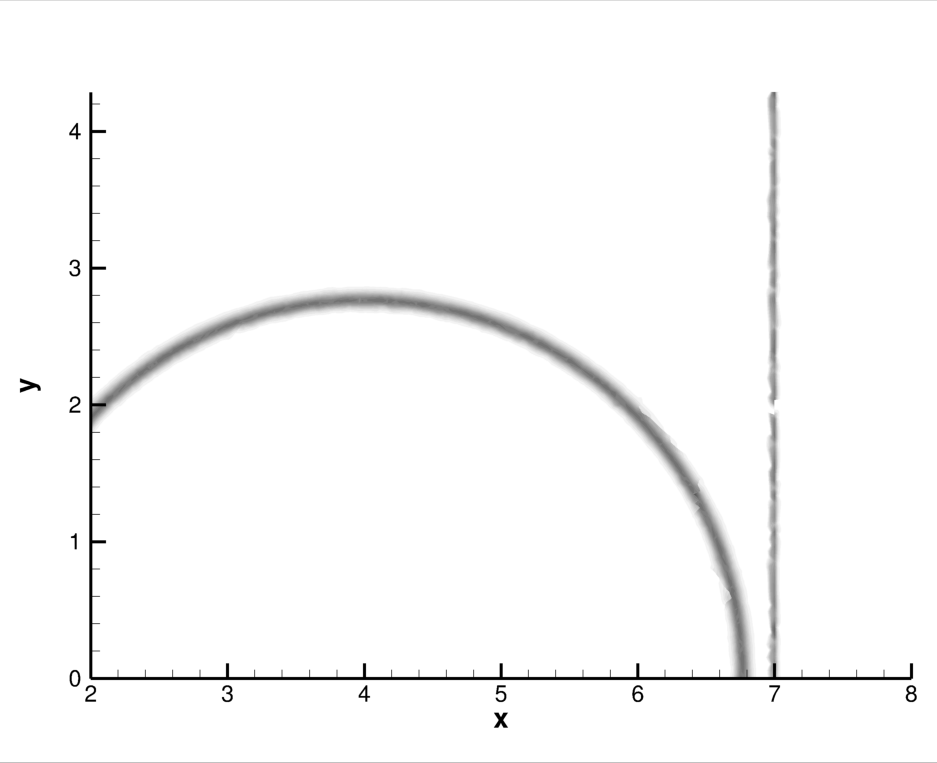,width=3.8cm}}
\subfloat[$t=1.5\;\mu$s]{\epsfig{figure=\LOCALFIGPATHSBINS 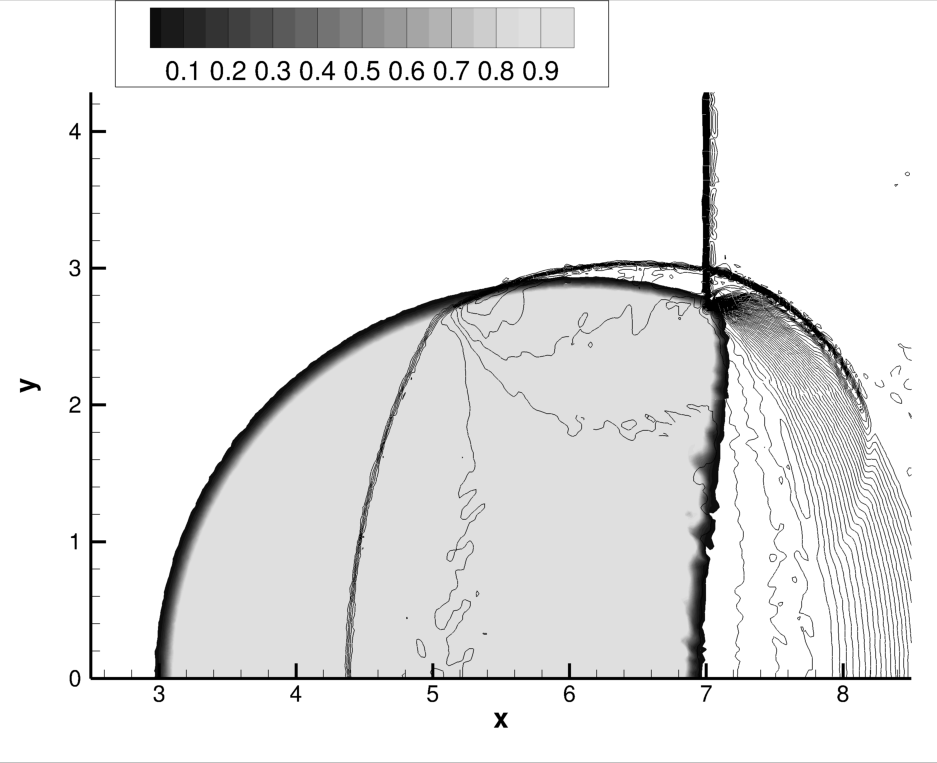,width=3.8cm}}
\subfloat               {\epsfig{figure=\LOCALFIGPATHSBINS 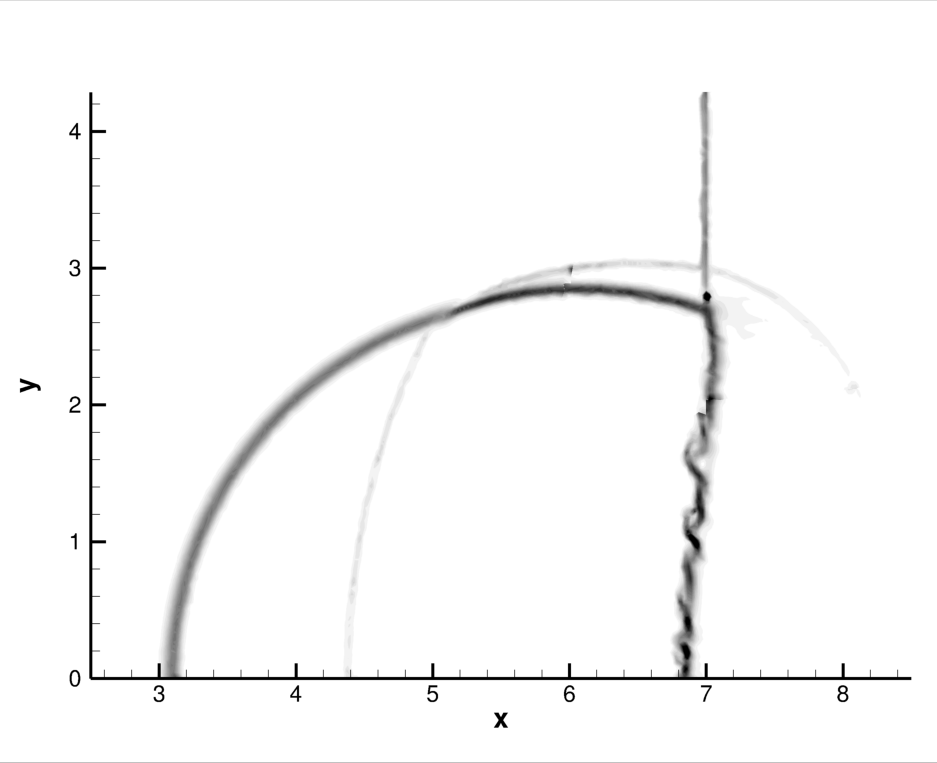,width=3.8cm}}\\
\subfloat[$t=2\;\mu$s]{\epsfig{figure=\LOCALFIGPATHSBINS 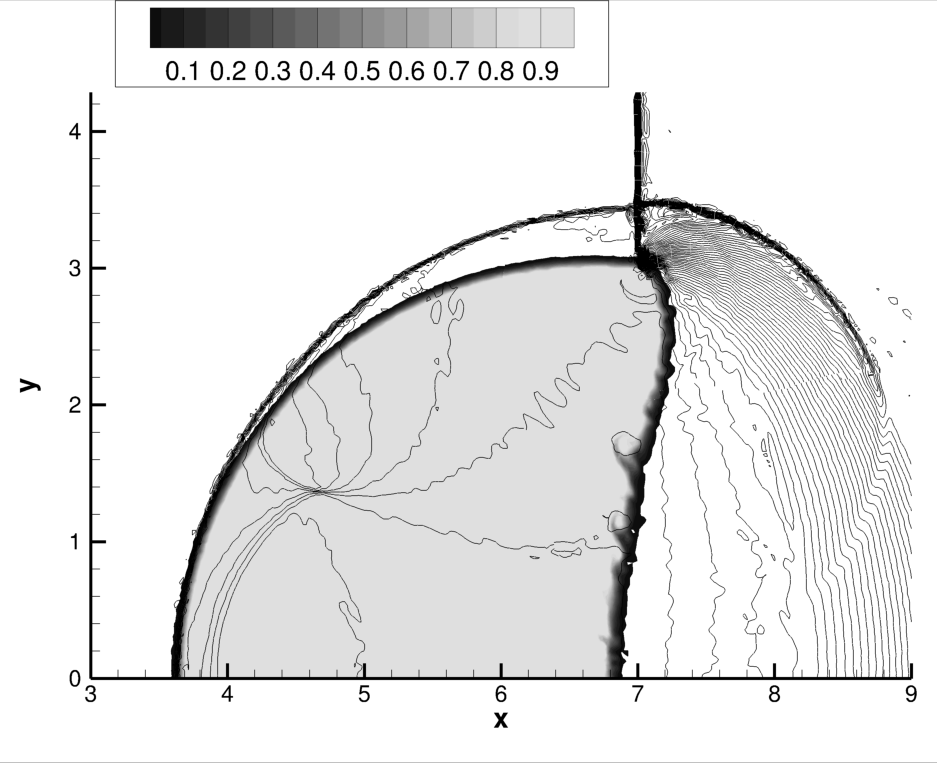,width=3.8cm}}
\subfloat             {\epsfig{figure=\LOCALFIGPATHSBINS 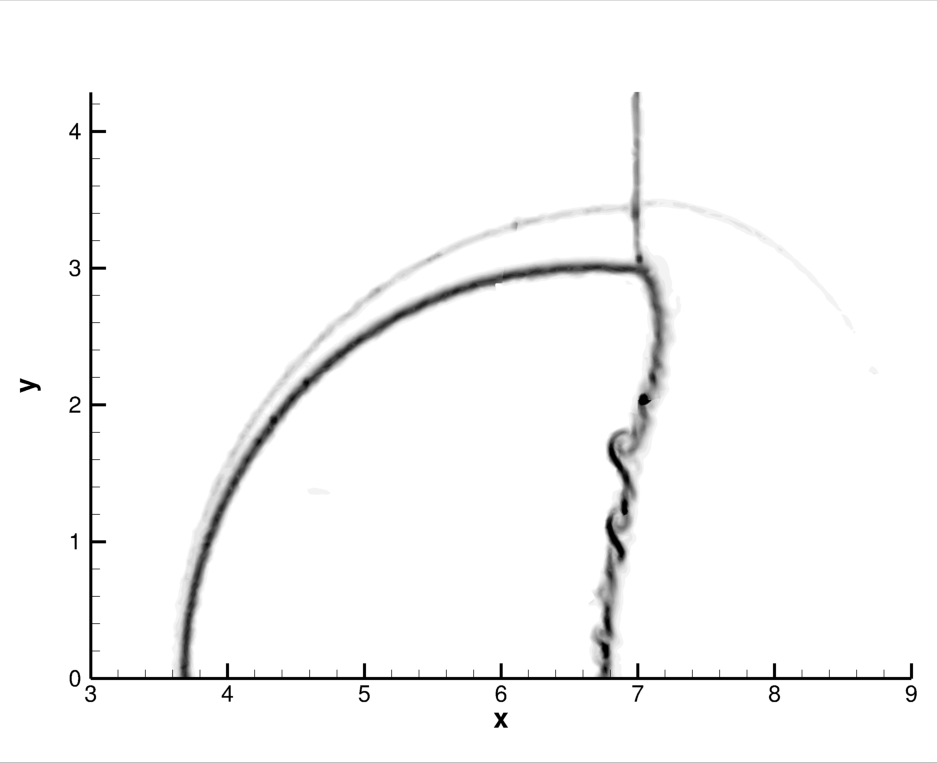,width=3.8cm}}
\subfloat[$t=2.5\;\mu$s]{\epsfig{figure=\LOCALFIGPATHSBINS 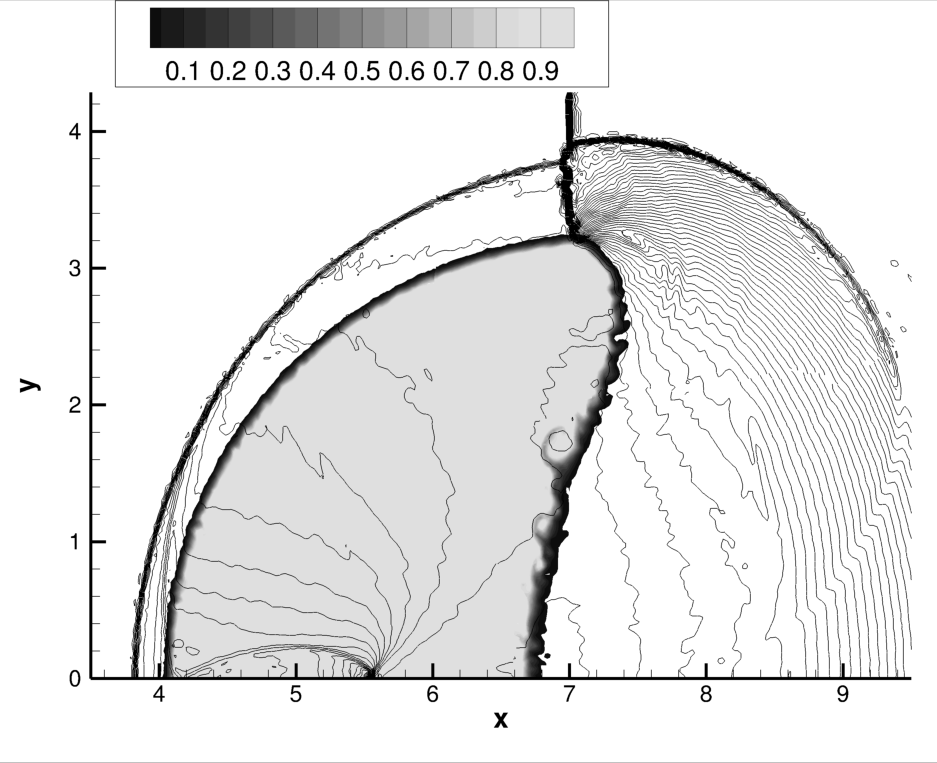,width=3.8cm}}
\subfloat               {\epsfig{figure=\LOCALFIGPATHSBINS 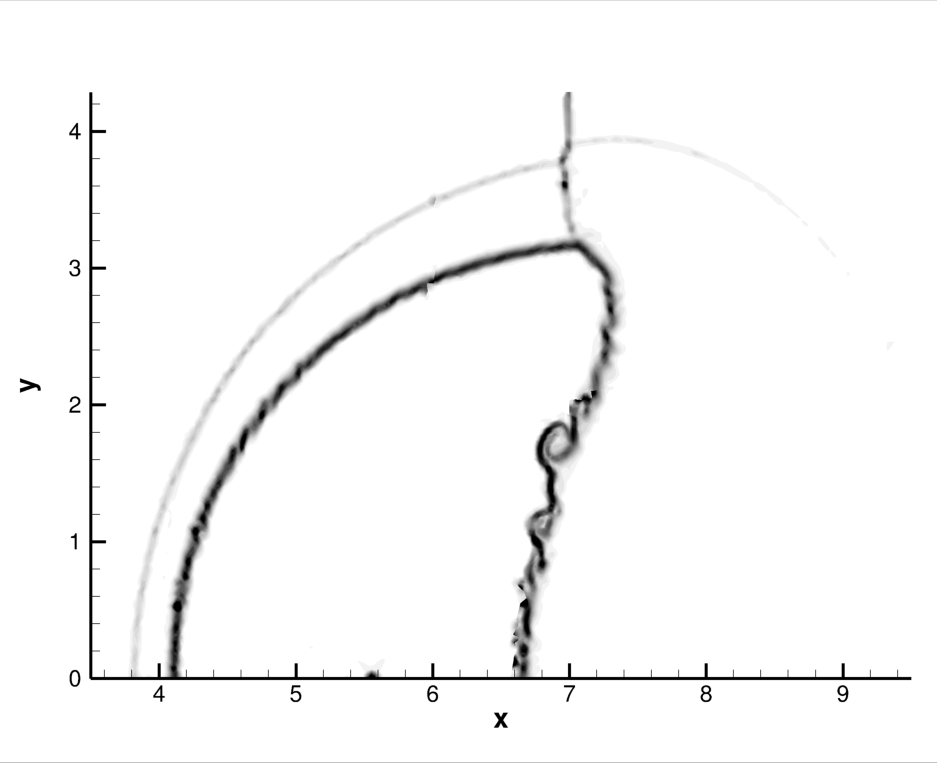,width=3.8cm}}\\
\subfloat[$t=3\;\mu$s]{\epsfig{figure=\LOCALFIGPATHSBINS 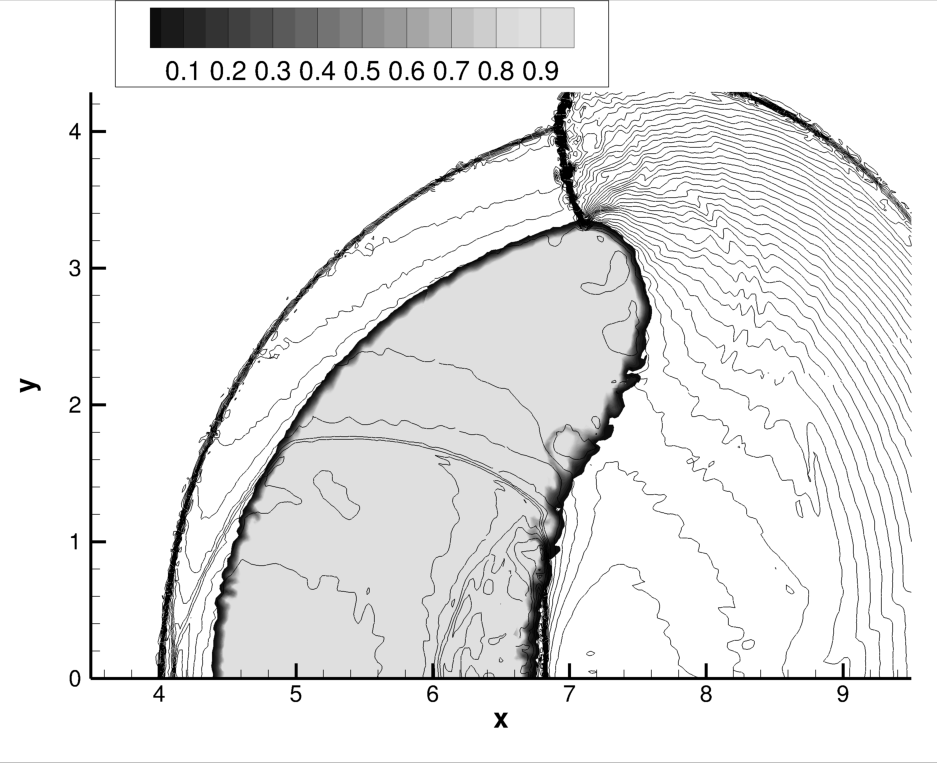,width=3.8cm}}
\subfloat             {\epsfig{figure=\LOCALFIGPATHSBINS 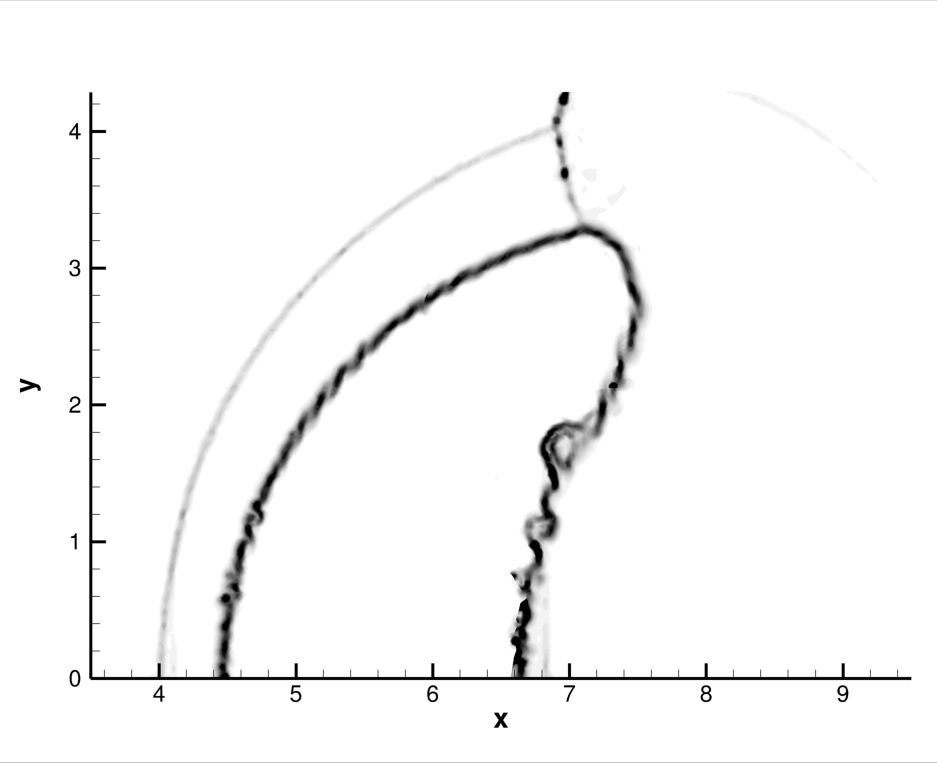,width=3.8cm}}
%
\subfloat[$t=4\;\mu$s]{\epsfig{figure=\LOCALFIGPATHSBINS 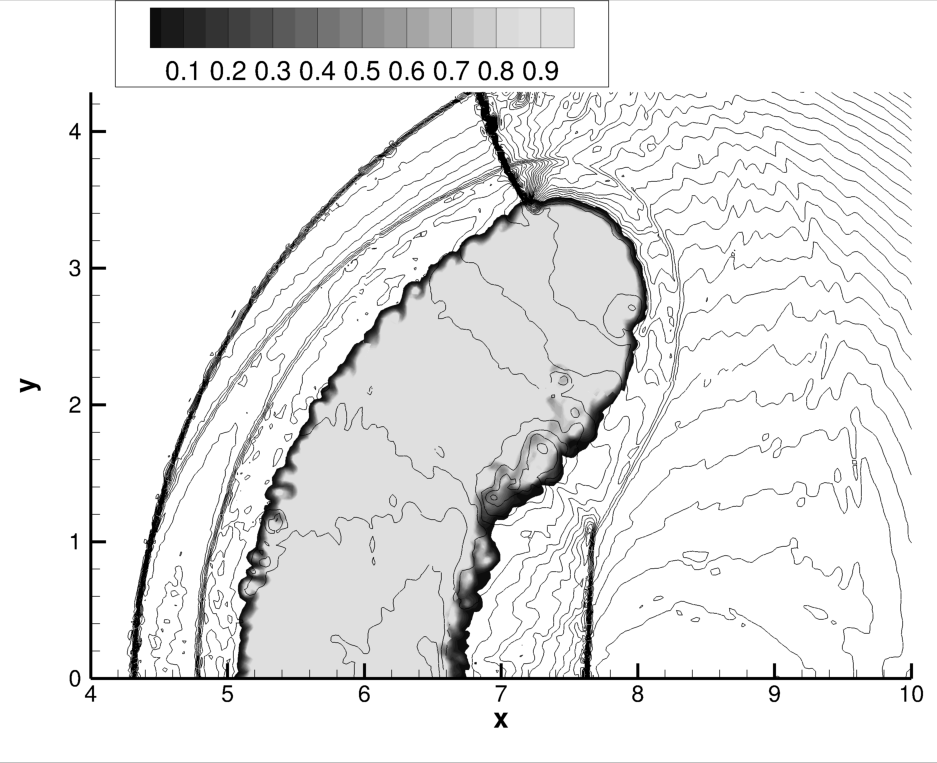,width=3.8cm}}
\subfloat             {\epsfig{figure=\LOCALFIGPATHSBINS 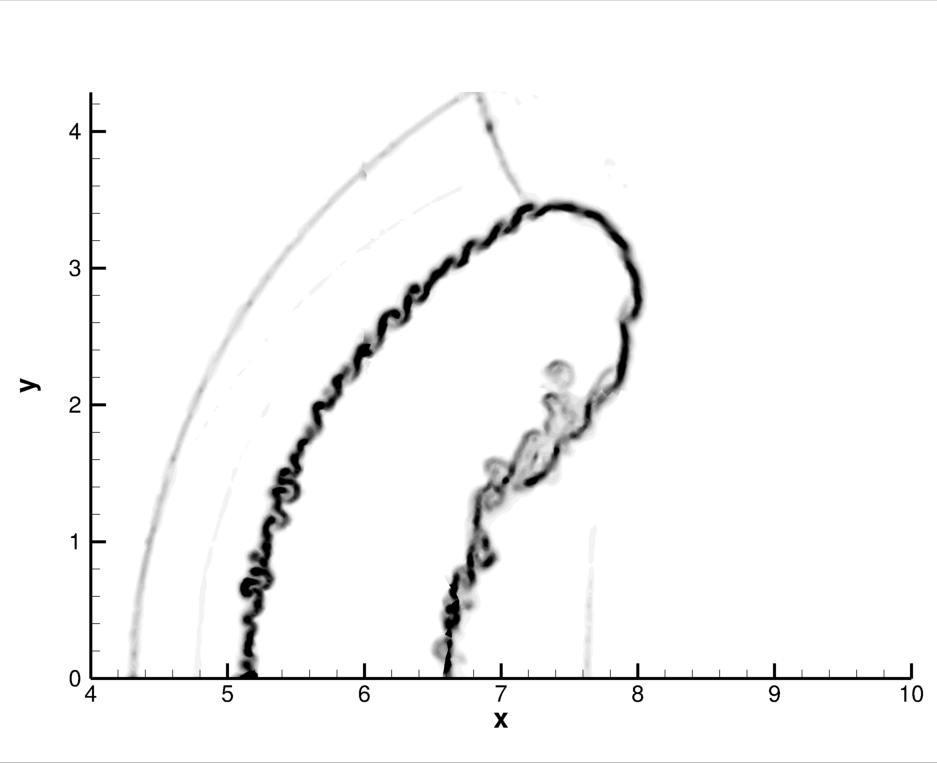,width=3.8cm}}\\
%
\subfloat[$t=8.8\;\mu$s]{\epsfig{figure=\LOCALFIGPATHSBINS 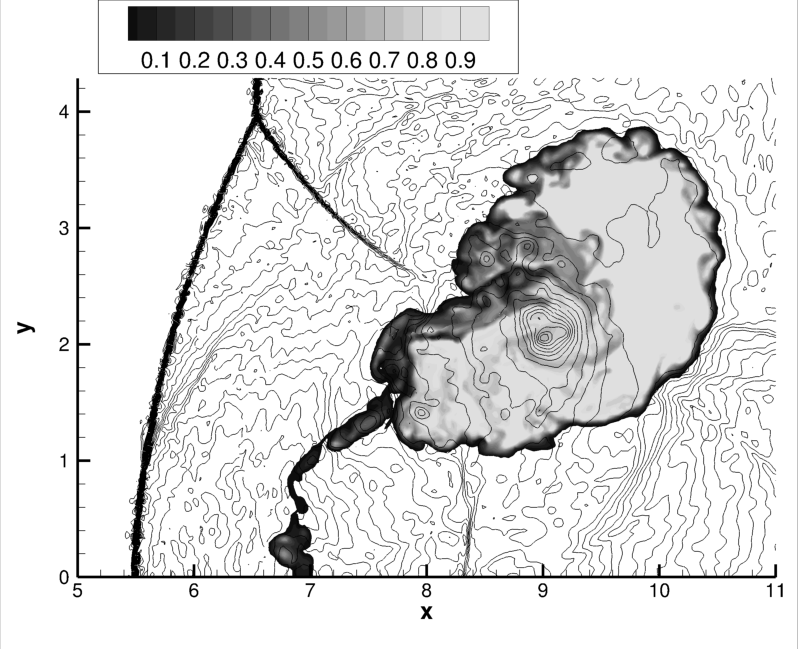,width=3.8cm}}
\subfloat               {\epsfig{figure=\LOCALFIGPATHSBINS 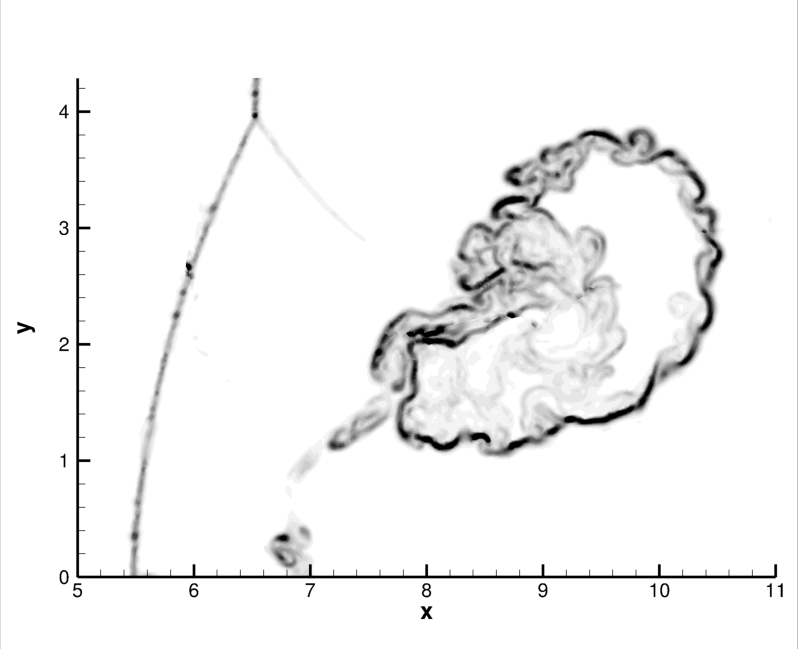,width=3.8cm}}
\subfloat[$t=13.6\;\mu$s]{\epsfig{figure=\LOCALFIGPATHSBINS 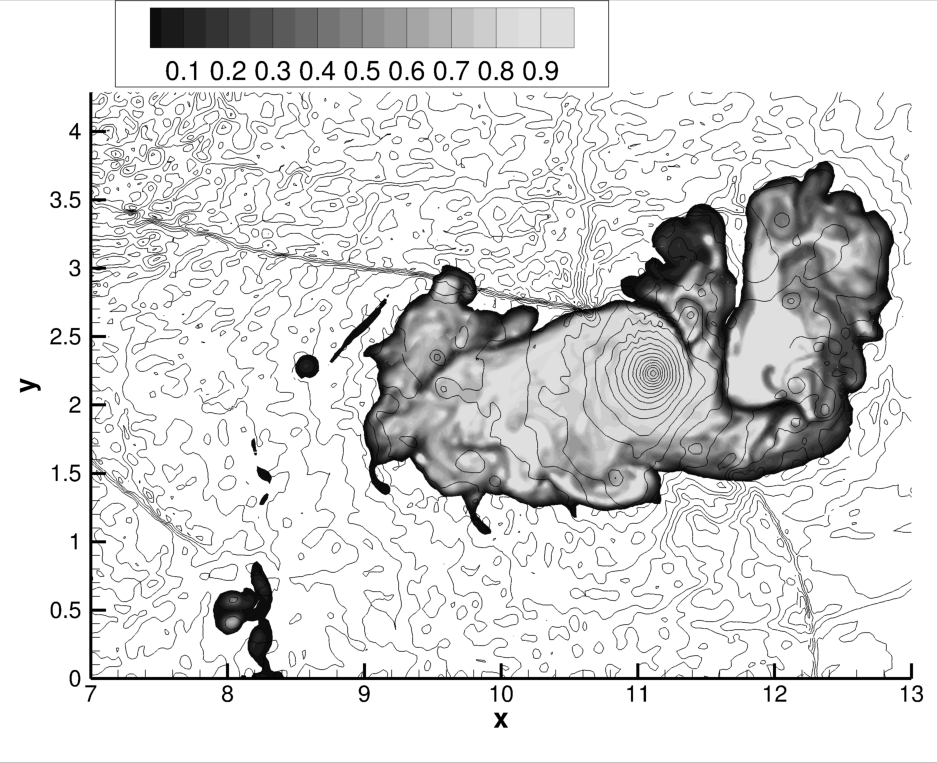,width=3.8cm}}
\subfloat                {\epsfig{figure=\LOCALFIGPATHSBINS 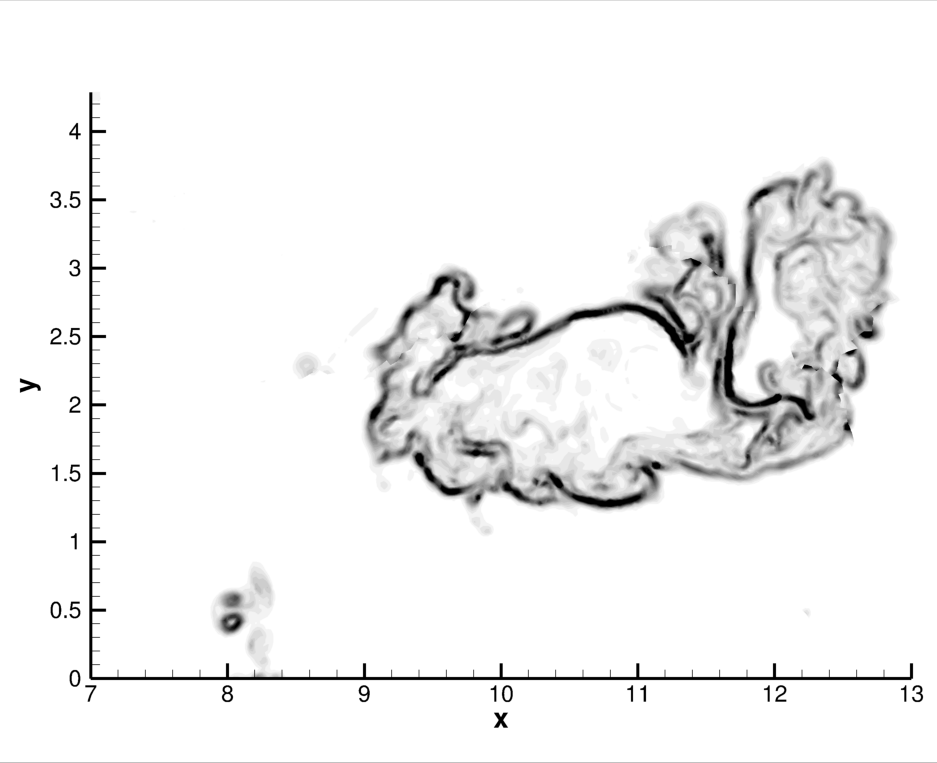,width=3.8cm}}
\caption{Strong shock wave-hydrogen bubble interaction: $68$ pressure contours from $0.06$ to $1.4$ (lines) and $20$ void fraction contours (colors, see legend), and Schlieren $|\nabla\rho|/\rho$ obtained with $p=3$ and $N=74,504$ cells at different physical times.}
\label{fig:solution_strong_SBI}
\end{bigcenter}
\end{figure}

%
%
\section{Concluding remarks}\label{sec:conclusions}

A high-order, ES and robust scheme is introduced in this work for the discretization of a multicomponent compressible Euler system on unstructured meshes in multiple space dimensions. The space discretization relies on the ES DGSEM framework \cite{fisher_carpenter_13,gassner_13,chen_shu_17} based on the modification of the integral over discretization elements where we replace the physical fluxes by EC numerical fluxes \cite{tadmor87} and on the use of ES numerical fluxes at element interfaces.

We first design two-point EC and ES fluxes for the multicomponent flow model. The latter is derived from the pressure relaxation scheme for the compressible Euler equations \cite{bouchut_04} and the energy relaxation approximation from \cite{coquel_perthame_98,renac_etal_ES_noneq_21} to allow the use of a simple polytropic equation of states for the mixture in the numerical approximation. This numerical flux provides an ES and robust three-point scheme which is then used as a building block for the design of the DGSEM. 

We then derive conditions on the numerical parameters and the time step to guaranty positivity of the density, internal energy, and void fractions of the cell-averaged solution when using a forward Euler discretization in time. We use a posteriori limiters from \cite{zhang2010positivity} to restore positivity of all DOFs within cells. The scheme is also proved to exactly resolve stationary contact waves. An explicit Runge-Kutta scheme \cite{spiteri_ruuth02} is used for the high-order time integration.

We perform high-order accurate numerical simulations of flows in one and two space dimensions with discontinuous solutions and complex wave interactions. The results highlight the accurate resolution of material interfaces, shock and contact waves, their interactions and associated small scale features. Likewise, robustness and nonlinear stability of the scheme are confirmed. Future investigations will focus on the \red{suppression} of spurious oscillations at moving material interfaces \red{by considering the discretization of alternative models \cite{abgrall_96,karni_mutlicomp_96,ALLAIRE2002577} while keeping the same properties as the present scheme} and on the extension of this approach to stiffened gas equations of states to account for mixture with liquid components.


%
%
\appendix

%
%
\section{Proof of \cref{th:positive-2D-FVO2-scheme}}\label{sec:proof_positive-2D-FVO2-scheme}

We here prove \cref{th:positive-2D-FVO2-scheme} that has been first proved in \cite[Th.~4]{perthame_shu_96} for triangles. Note that the proof can be generalized to star-shaped polygonal elements. We first need the following result that also extends \cite[Th.~3]{perthame_shu_96} to quandrangles.

\begin{lemma}\label{th:positive-2D-scheme}
Let a three-point numerical scheme of the form \cref{eq:3pt-scheme-a} with a consistent \cref{eq:consistent_flux}, conservative \cref{eq:conservative_flux}, and Lipschitz continuous numerical flux ${\bf h}(\cdot,\cdot,\cdot)$ for the discretization of \cref{eq:HRM_PDEs-a} that satisfies positivity of the solution, ${\bf U}_{j\in\mathbb{Z}}^{n\geq0}$ in $\Omega^a$, under the CFL condition in \cref{eq:3pt-scheme-a}. Then, the numerical scheme \cref{eq:2D-FV-scheme} on quadrangular meshes with the same numerical flux is positive under the condition
\begin{equation}\label{eq:CFL-positive-2D-scheme}
  \Delta t \max_{\kappa\in\Omega_h}\frac{|\partial\kappa|}{|\kappa|}\max_{e\in\partial\kappa}|\lambda({\bf U}_{\kappa^\pm}^{n})| \leq \frac{1}{2}, \quad |\partial\kappa|:=\sum_{e\in\partial\kappa}|e|.
\end{equation}
\end{lemma}

\begin{proof}
Using $\sum_{e\in\partial\kappa}|e|{\bf n}_e=0$, we rewrite \cref{eq:2D-FV-scheme} under the form
\begin{eqnarray*}
  {\bf U}_\kappa^{n+1} &=& {\bf U}_\kappa^{n} - \tfrac{\Delta t}{|\kappa|}\sum_{e\in\partial\kappa} |e|\big({\bf h}({\bf U}_\kappa^{n},{\bf U}_{\kappa_e^+}^{n},{\bf n}_{e})-{\bf f}({\bf U}_\kappa^{n})\cdot{\bf n}_{e}\big) \\
  &\overset{\cref{eq:consistent_flux}}{=}& {\bf U}_\kappa^{n} - \tfrac{\Delta t}{|\kappa|}\sum_{e\in\partial\kappa} |e|\big({\bf h}({\bf U}_\kappa^{n},{\bf U}_{\kappa_e^+}^{n},{\bf n}_{e})-{\bf h}({\bf U}_\kappa^{n},{\bf U}_{\kappa}^{n},{\bf n}_{e})\big) \\
  &\overset{\cref{eq:CFL-positive-2D-scheme}}{=}& \sum_{e\in\partial\kappa} \tfrac{|e|}{|\partial\kappa|}\Big({\bf U}_\kappa^{n} - \tfrac{\Delta t|\partial\kappa|}{|\kappa|}\big({\bf h}({\bf U}_\kappa^{n},{\bf U}_{\kappa_e^+}^{n},{\bf n}_{e})-{\bf h}({\bf U}_\kappa^{n},{\bf U}_{\kappa}^{n},{\bf n}_{e})\big) \Big),
\end{eqnarray*}

\noindent which is a convex combination of positive three-point schemes \cref{eq:3pt-scheme-a} under the condition \cref{eq:CFL-positive-2D-scheme}. \qed
\end{proof}

\begin{proof}[Proof of \cref{th:positive-2D-FVO2-scheme}]
We use the notations in Figure~\ref{fig:quad-with-subtriangles} and set ${\bf n}_{ef}=-{\bf n}_{fe}$ and $l_{ef}=l_{fe}$. By conservation \cref{eq:conservative_flux}, we have
\begin{equation*}
 \sum_{e\in\partial\kappa}\sum_{f\in\partial\kappa_e\backslash\{e\}}l_{ef}{\bf h}({\bf U}_{\kappa_e^-}^{n},{\bf U}_{\kappa_f^-}^{n},{\bf n}_{ef}) = 0.
\end{equation*}

Adding this quantity to \cref{eq:2D-FVO2-scheme}, we get
\begin{eqnarray*}
  {\bf U}_\kappa^{n+1} &\overset{\cref{eq:2D-FVO2-scheme}}{=}& \sum_{e\in\partial\kappa} \tfrac{|e|}{|\partial\kappa|}{\bf U}_{\kappa_e^-}^{n} - \tfrac{\Delta t}{|\kappa|}\sum_{e\in\partial\kappa}\Big( |e|{\bf h}({\bf U}_{\kappa_e^-}^{n},{\bf U}_{\kappa_e^+}^{n},{\bf n}_e) + \sum_{f\in\partial\kappa_e\backslash\{e\}}l_{ef}{\bf h}({\bf U}_{\kappa_e^-}^{n},{\bf U}_{\kappa_f^-}^{n},{\bf n}_{ef})\Big) \\
  &=& \sum_{e\in\partial\kappa} \tfrac{|e|}{|\partial\kappa|}\Big({\bf U}_{\kappa_e^-}^{n} - \tfrac{|\partial\kappa|}{|e|}\tfrac{\Delta t}{|\kappa|}\sum_{f\in\partial\kappa_e}l_{ef}{\bf h}({\bf U}_{\kappa_e^-}^{n},\hat{\bf U}_{\kappa}^{ef,n},{\bf n}_{ef})\Big),
\end{eqnarray*}
\noindent with the conventions $l_{ee}=|e|$, ${\bf n}_{ee}={\bf n}_{e}$, $\hat{\bf U}_{\kappa}^{ee,n}={\bf U}_{\kappa_e^+}^{n}$, and $\hat{\bf U}_{\kappa}^{ef,n}={\bf U}_{\kappa_f^-}^{n}$ for $f\neq e$. The scheme \cref{eq:2D-FVO2-scheme} is therefore a convex combination of positive schemes of the form \cref{eq:2D-FV-scheme} under \cref{eq:CFL-positive-2D-FVO2-scheme}. \qed
\end{proof}

%
%
\section{Another EC flux for the physical entropy}\label{sec:appendix_EC_flux}

We here derive another EC flux which is similar to the EC introduced in \cite{gouasmi_etal_20} and reads
\begin{equation}\label{eq:EC_flux_gouasmi}
 {\bf h}_{ec}({\bf u}^-,{\bf u}^+,{\bf n}) = \begin{pmatrix} h_{\rho Y_{1} }({\bf u}^-,{\bf u}^+,{\bf n}) \\ \vdots \\ h_{\rho Y_{n_c-1} }({\bf u}^-,{\bf u}^+,{\bf n}) \\ h_{\rho}({\bf u}^-,{\bf u}^+,{\bf n}) \\ h_{\rho}({\bf u}^-,{\bf u}^+,{\bf n})\ol{{\bf v}} + \tfrac{\ol{\mathrm{p}\theta}}{\ol{\theta}}{\bf n} \\  \displaystyle\sum_{i=1}^{n_c}\Big(\tfrac{C_{v_i}}{\widehat{\theta}}+\tfrac{{\bf v}^{-}\cdot{\bf v}^{+}}{2}\Big)h_{\rho Y_i}({\bf u}^-,{\bf u}^+,{\bf n}) + \tfrac{\ol{\mathrm{p}\theta}}{\ol{\theta}}\ol{{\bf v}}\cdot{\bf n} \end{pmatrix}, \; \begin{aligned}  &h_{\rho Y_i}({\bf u}^-,{\bf u}^+,{\bf n}) = \ol{\alpha_i}\widehat{\rho_i}\ol{{\bf v}}\cdot{\bf n}, \\ &h_{\rho}(\cdot,\cdot,\cdot) = \sum_{i=1}^{n_c}h_{\rho Y_i}(\cdot,\cdot,\cdot),\end{aligned}
\end{equation}

\noindent with $\theta=\tfrac{1}{\T}$. Symmetry \cref{eq:symmetric_flux} and consistency \cref{eq:consistent_flux} follow from symmetry and consistency of the logarithmic mean and average operator. Let now prove that \cref{eq:EC_flux_gouasmi} is EC \cref{eq:entropy_conserv_flux} by following the lines of \cite{gouasmi_etal_20}. We will use $\du\mathrm{s}_i\df \overset{\cref{eq:partial_entropy}}{=} -C_{v_i}\du\ln\theta\df - r_i\du\ln\rho_i\df$. Let expand
\begin{equation*}
\du\psib({\bf u})\cdot{\bf n}\df \overset{\cref{eq:def-potential}}{=} \du r({\bf Y})\rho{\bf v}\cdot{\bf n}\df \overset{\cref{eq:mixture_eos}}{=} \big(\ol{\mathrm{p}\theta}\du {\bf v}\df+\du\rho r({\bf Y})\df\ol{\bf v}\big)\cdot{\bf n} \blue{ \overset{\cref{eq:def-rho-rhoE-p}}{=}  \ol{\mathrm{p}\theta}\du {\bf v}\df\cdot{\bf n} + \sum_{i=1}^{n_c}\ol{\alpha_i}\du\rho r({\bf Y})\df\ol{\bf v}\cdot{\bf n} \overset{\cref{eq:partial_densities}}{=} \ol{\mathrm{p}\theta}\du {\bf v}\df\cdot{\bf n} + \sum_{i=1}^{n_c}\ol{\alpha_i}r_i\du\rho_i\df\ol{\bf v}\cdot{\bf n}. }
\end{equation*}

Using short notations for the flux components and the definition of $h_\rho$ in \cref{eq:EC_flux_gouasmi} we get
\begin{align*}
 \du \etab'({\bf u}) \df \cdot {\bf h}_{ec}({\bf u}^-,{\bf u}^+,{\bf n}) &= \sum_{i=1}^{n_c-1}\du (C_{v_i}-C_{v_{n_c}})\ln\theta+r_i\ln\rho_i-r_{n_c}\ln\rho_{n_c}\df h_{\rho Y_i} \\
 &+ \du C_{v_{n_c}}\ln\theta+r_{n_c}\ln\rho_{n_c}-\tfrac{{\bf v}\cdot{\bf v}}{2}\theta\df h_\rho + \du\theta{\bf v}\df\cdot{\bf h}_{\rho {\bf v}}-\du\theta\df h_{\rho E} \\
 &\overset{\cref{eq:EC_flux_gouasmi}}{=} \sum_{i=1}^{n_c} (C_{v_i}\du\ln\theta\df+r_i\du\ln\rho_i\df) h_{\rho Y_i} - \Big(\ol{\tfrac{{\bf v}\cdot{\bf v}}{2}}\du\theta\df+\ol{\theta}\ol{\bf v}\cdot\du{\bf v}\df\Big)h_{\rho} \\
 &+ (\du\theta\df\ol{\bf v}+\ol{\theta}\du{\bf v}\df)\cdot{\bf h}_{\rho {\bf v}}-\du\theta\df h_{\rho E}
\end{align*}

\noindent and $\du \etab'({\bf u}) \df \cdot {\bf h}_{ec}({\bf u}^-,{\bf u}^+,{\bf n}) - \du \psib({\bf u})\cdot{\bf n}\df$ becomes
\begin{equation*}
  \sum_{i=1}^{n_c}r_i\du\ln\rho_i\df\big(h_{\rho Y_i}-\ol{\alpha_i}\widehat{\rho_i}\ol{\bf v}\cdot{\bf n}\big) + \ol{\theta}\du{\bf v}\df\cdot\big({\bf h}_{\rho {\bf v}}-h_{\rho}\ol{\bf v}-\tfrac{\ol{\mathrm{p}\theta}}{\ol{\theta}}{\bf n}\big) + \du\ln\theta\df\big(\sum_{i=1}^{n_c}C_{v_i} h_{\rho Y_i} - \widehat{\theta}(h_{\rho E} - \ol{\bf v}\cdot{\bf h}_{\rho {\bf v}} + \ol{\tfrac{{\bf v}\cdot{\bf v}}{2}}h_{\rho})\big) 
\end{equation*}
\noindent which indeed vanishes from \cref{eq:EC_flux_gouasmi}.
%

%
%
\section{Pressure-based relaxation system and numerical flux for \cref{eq:relax-nrj-HRM-sys}}\label{sec:appendix_relax_sys_for_nrj_relax}

We here give details on the pressure-based relaxation system used to derive the ES and robust numerical flux in \cref{eq:3pt-scheme-relax} for the relaxation model in homogeneous form \cref{eq:relax-nrj-HRM-sys-homogeneous}. This numerical scheme is adapted from \cite[Prop.~2.21]{bouchut_04} and is based on a relaxation approximation using evolution equations for a relaxation pressure in place of $\mathrm{p}_r(\rho,e_r)$ and for $a$ in \cref{eq:sol-PR-relax} in place of the Lagrangian sound speed $\rho c_\gamma(\rho,\mathrm{p}_r)=\sqrt{\gamma\rho\mathrm{p}_r}$. We first recall the model in \cref{sec:appendix_press_relax_sys} and derive the numerical scheme \cref{eq:3pt-scheme-relax} and entropy inequality \cref{eq:3pt-scheme-equal-relax} in \cref{sec:appendix_press_relax_flux}. We refer to \cite[Sec.~2.4]{bouchut_04} or \cite{coquel_etal_relax_fl_sys_12} for complete introductions and in-depth analyses.

\subsection{Pressure-based relaxation system}\label{sec:appendix_press_relax_sys}

This system is here adapted from \cite[Sec.~ 2.4]{bouchut_04} to multiple components and reads

\begin{equation}\label{eq:press-relaxation-sys}
 \partial_t{\bf w}_r + \nabla\cdot{\bf g}_r({\bf w}_r) = 0, \quad
 {\bf w}_r = \begin{pmatrix}\rho {\bf Y} \\ \rho \\ \rho {\bf v} \\ \rho E_r \\ \rho e_s \\ \tfrac{\rho\pi}{a^2} \\ \rho a \end{pmatrix}, \quad
{\bf g}_r({\bf w}_r) = \begin{pmatrix} \rho {\bf Y}{\bf v}^\top \\ \rho {\bf v}^\top \\ \rho {\bf v}{\bf v}^\top+\pi{\bf I} \\ \big(\rho E_r+\pi\big){\bf v}^\top \\ \rho e_s {\bf v}^\top \\ \tfrac{\rho\pi}{a^2}{\bf v}^\top + {\bf v}^\top \\ \rho a {\bf v}^\top \end{pmatrix}.
\end{equation}

Following \cite[Sec.~2.4]{bouchut_04}, the relaxation mechanisms are not included in \cref{eq:press-relaxation-sys}, but are replaced by time discrete projection onto the equilibrium manifold $\{{\bf w}_r:\;\pi=\mathrm{p}_r(\rho,e_{r})\}$. The system is hyperbolic in the direction ${\bf n}$ with eigenvalues ${\bf v}\cdot{\bf n}-a/\rho$, ${\bf v}\cdot{\bf n}$, and ${\bf v}\cdot{\bf n}+a/\rho$ associated to LD fields. The exact solution to the Riemann problem for \cref{eq:press-relaxation-sys} with initial data ${\bf w}_{r,0}(x) = {\bf w}_{r,L}$ if $x:={\bf x}\cdot{\bf n}<0$ and ${\bf w}_{r,0}(x) = {\bf w}_{r,R}$ if ${\bf x}\cdot{\bf n}>0$ reads

\begin{equation}\label{eq:press-relax-RP-sol}
 \bcW_r^{\pi}(\tfrac{x}{t};{\bf w}_{r,L},{\bf w}_{r,R},{\bf n}) = \left\{ \begin{array}{ll}  {\bf w}_{r,L}, & \tfrac{x}{t} < u_L-a_L/\rho_L, \\
 {\bf w}_{r,L}^\star, & u_L-a_L/\rho_L<\tfrac{x}{t}<u^\star,\\
 {\bf w}_{r,R}^\star, & u^\star<\tfrac{x}{t}<u_R+a_R/\rho_R,\\
 {\bf w}_{r,R}, & u_R+a_R/\rho_R < \tfrac{x}{t}, \end{array} \right.
\end{equation}

\noindent where $u_X={\bf v}_X\cdot{\bf n}$ and ${\bf w}_{r,X}^\star=\big(\rho_X^\star {\bf Y}_X^\top, \rho_X^\star, \rho_X^\star {\bf v}_X^{\star\top},\rho_X^\star E_{r,X}^\star,\rho_X^\star e_{s,X}, \tfrac{\rho_X^\star\pi^\star}{a_X^2},\rho_X^\star a_X\big)^\top$ for $X=L,R$. The quantities $\rho_X^\star$ and ${\bf v}_X^{\star}$ are defined in \cref{eq:sol-PR-relax}, while

\begin{subequations}\label{eq:press-relax-RP-sol-details}
\begin{align}
u^\star  &= \frac{a_Lu_L+a_Ru_R+\pi_L-\pi_R}{a_L+a_R}, \quad \pi^\star = \frac{a_R\pi_L+a_L\pi_R+a_La_R(u_L-u_R)}{a_L+a_R}, \\ E_{r,L}^\star &= E_{r,L} - \frac{\pi^\star u^\star - \mathrm{p}_{r,L}u_L}{a_L},\quad E_{r,R}^\star = E_{r,R} - \frac{\mathrm{p}_{r,R}u_R-\pi^\star u^\star}{a_R}.
\end{align}
\end{subequations}

From \cref{eq:def_zeta_tau} we write $\zeta({\bf Y},\tau,e_r,e_s)=-\tfrac{r({\bf Y})}{\gamma-1}\ln(\tau^{\gamma-1}e_r)+\varsigma_r({\bf Y},e_s)$ and introduce the quantities $X=e_r-\tfrac{\pi}{2a^2}$ and $I=\pi+a^2\tau$. For smooth solutions of \cref{eq:press-relaxation-sys} we have
\begin{equation*}
 \partial_ta+{\bf v}\cdot\nabla a = 0, \quad \partial_tX+{\bf v}\cdot\nabla X = 0, \quad \partial_tI+{\bf v}\cdot\nabla I = 0, \quad \partial_te_s+{\bf v}\cdot\nabla e_s = 0,
\end{equation*}

\noindent so the function $\zeta^r({\bf Y},\tau,e_r,e_s,\pi,a)=-\tfrac{r({\bf Y})}{\gamma-1}\ln\big(\tilde{\tau}(a,X,I)^{\gamma-1}\tilde{e_r}(a,X,I)\big)+\varsigma_r({\bf Y},e_s)$ satisfies
\begin{equation}\label{eq:press-relax-entropy-pde}
 \partial_t\rho\zeta^r+\nabla\cdot(\rho\zeta^r{\bf v}) = 0,
\end{equation}

\noindent where the functions satisfy $\tilde{\tau}(a,X,I)=\tau$ and $\tilde{e_r}(a,X,I)=e_r$ at equilibrium $\pi=\mathrm{p}_r(\rho,e_r)$ and are defined in \cite{berthon_05}. Moreover, under the subcharacteristic condition 
\begin{equation}\label{eq:subcharact_cond}
 a \geq  \sqrt{\gamma\rho\mathrm{p}_r(\rho,e_r)} \quad \forall \rho > 0, e_r > 0, 
\end{equation}

\noindent we have the following minimization principle
\begin{equation}\label{eq:press-relax-var-principle}
  \zeta({\bf Y},\tau,e_r,e_s)=\min_{\pi\in\mathbb{R}}\zeta^r({\bf Y},\tau,e_r,e_s,\pi,a),
\end{equation}  

\noindent and the minimum is reached at equilibrium $\pi=\mathrm{p}_r(\rho,e_r)$ \cite{berthon_05}. 

\subsection{Numerical scheme for \cref{eq:relax-nrj-HRM-sys-homogeneous}}\label{sec:appendix_press_relax_flux}

The numerical scheme \cref{eq:3pt-scheme-relax} for \cref{eq:relax-nrj-HRM-sys-homogeneous} is based on the pressure relaxation system \cref{eq:press-relaxation-sys} and uses two steps between times $t^{(n)}$ and $t^{(n+1)}=t^{(n)}+\Delta t$: an evolution step between $t^{(n)}$ and $t^{(n+1)-}=t^{(n)}+\Delta t$ and an instantaneous projection step from $t^{(n+1)-}$ to $t^{(n+1)}$. 

In the evolution step, we solve the Cauchy problem \cref{eq:press-relaxation-sys} with initial data ${\bf w}_{r,0}(x)={\bf W}_{r,j}^n=\big(({\bf W}_j^n)^\top,\tfrac{\rho_j^n\pi_j^n}{(a_j^n)^2},\rho_j^na_j^n\big)^\top$ for $x$ in the $j$th cell $\kappa_j$, where ${\bf W}_j^n$ are the DOFs in \cref{eq:3pt-scheme-relax}, and $\pi_j^n$ and $a_j^n$ will be defined below. The solution ${\bf w}_{r,h}(x,t)$ to this problem consists in the juxtaposition of Riemann problem solutions \cref{eq:press-relax-RP-sol} at each mesh interface. Under the condition \cref{eq:3-pt-CFL_cond} on the time step, these solutions do not interact. Let consider $\rho\zeta^r$ as a function of ${\bf w}_r$, integrating \cref{eq:press-relax-entropy-pde} on the rectangle $\kappa_j\times[t^{(n)},t^{(n+1)-}]$ we obtain
\begin{equation}\label{eq:press-relax-entropy-equality}
 \big\langle \rho\zeta^r\big({\bf w}_{r,h}(x,t^{(n+1)-})\big) \big\rangle_j - \rho\zeta^r({\bf W}_{r,j}^n) + \tfrac{\Delta t}{\Delta x}\Big( H_{\rho\zeta^r}\big({\bf W}_{r,j}^n,{\bf W}_{r,j+1}^n,{\bf n}\big) - H_{\rho\zeta^r}\big({\bf W}_{r,j-1}^n,{\bf W}_{r,j}^n,{\bf n}\big) \Big) = 0.
\end{equation}

\noindent where $H_{\rho\zeta^r}({\bf W}_{r,j}^n,{\bf W}_{r,j+1}^n,{\bf n})=\{\rho\zeta^r{\bf v}\cdot{\bf n}\}\big(\bcW_r^{\pi}(0,{\bf W}_{r,j}^n,{\bf W}_{r,j+1}^n,{\bf n})\big)$ denotes the entropy flux evaluated at the interface from \cref{eq:press-relax-RP-sol} and $\langle\cdot\rangle_j$ is the cell average. 

In the projection step, the solution at time $t^{(n+1)-}$ is projected onto the equilibrium manifold $\{{\bf w}_r:\;\pi=\mathrm{p}_r(\rho,e_{r})\}$ which amounts to impose $\pi_j^{n+1}=\mathrm{p}_r(\rho_j^{n+1},e_{r,j}^{n+1})$, so $\rho\zeta^r({\bf W}_{r,j}^{n+1})=\rho\zeta({\bf W}_j^{n+1})$. Using the minimization principle \cref{eq:press-relax-var-principle} and then the Jensen's inequality applied to the convex function $\rho\zeta({\bf w})$ we obtain
\begin{equation}\label{eq:press-relax-entropy-projection}
 \big\langle \rho\zeta^r\big({\bf w}_{r,h}(x,t^{(n+1)-})\big) \big\rangle_j \geq \big\langle \rho\zeta\big({\bf w}_{h}(x,t^{(n+1)})\big) \big\rangle_j \geq \rho\zeta({\bf W}_j^{n+1}).
\end{equation}

By combining \cref{eq:press-relax-entropy-equality,eq:press-relax-entropy-projection}, we obtain \cref{eq:3pt-scheme-equal-relax}
%
with
\begin{equation}\label{eq:3pt-scheme-equal-relax-details}
  Z({\bf W}_{j}^{n},{\bf W}_{j+1}^{n},{\bf n})=H_{\rho\zeta^r}({\bf W}_{r,j}^n,{\bf W}_{r,j+1}^n,{\bf n}), \quad {\bf W}_{r,j}^n=\big(({\bf W}_j^n)^\top,\tfrac{\rho_j^n\mathrm{p}_r(\rho_j^n,e_{r,j}^n)}{(a_j^n)^2},\rho_j^na_j^n\big)^\top,
\end{equation}

\noindent where the coefficients $a_j^n$ will be defined below to guaranty to satisfy the subcharacteristic condition \cref{eq:subcharact_cond} and the positivity of the solution, ${\bf W}_j^{n+1}\in\Omega^r$, \cite{berthon_05}. Indeed, observe that the variables $(\rho,\rho{\bf v}^\top,\rho E_r,\tfrac{\rho\pi}{a^2},\rho a)^\top$ in \cref{eq:press-relaxation-sys} are uncoupled from $\rho{\bf Y}$ and $\rho e_s$, so the associated equations correspond to the relaxation approximation from \cite[Sec.~ 2.4]{bouchut_04} for the Euler equations with a perfect gas equation of state \cref{eq:EOS-relax}. Applying \cite[Prop.~2.21]{bouchut_04}, the three-point scheme \cref{eq:3pt-scheme-relax} guaranties positivity of $\rho$ and $e_r$ under the CFL condition
\begin{equation}\label{eq:3-pt-CFL_cond-relax}
 \frac{\Delta t}{\Delta x}\max_{j\in\mathbb{Z}}|\lambda({\bf W}_j^{n})| < \frac{1}{2}, \quad |\lambda({\bf w})| := |{\bf v}\cdot{\bf n}|+\tfrac{a}{\rho}.
\end{equation}

Positivity of $\rho{\bf Y}$ and $\rho e_s$ then follows from averaging the Riemann solution \cref{eq:press-relax-RP-sol} under \cref{eq:3-pt-CFL_cond-relax}.

The pressure relaxation-based numerical flux for \cref{eq:relax-nrj-HRM-sys-homogeneous} thus reads 
\begin{equation}\label{eq:relax_flux_for_nrj_sys}
 {\bf H}({\bf w},{\bf w},{\bf n}) = {\bf g}\big(\bcW^{\pi}(0;{\bf w}^-,{\bf w}^+,{\bf n})\big)\cdot{\bf n},
\end{equation}

\noindent where $\bcW^\pi(\cdot;{\bf w}_L,{\bf w}_R,{\bf n})$ is used to approximate the solution to the Riemann problem for \cref{eq:relax-nrj-HRM-sys-homogeneous} with initial data ${\bf w}_0(x) = {\bf w}_L$ if $x:={\bf x}\cdot{\bf n}<0$ and ${\bf w}_0(x) = {\bf w}_R$ if ${\bf x}\cdot{\bf n}>0$:

\begin{equation}\label{eq:bouchut-relax-ARS-xt}
 \bcW^{\pi}(\tfrac{x}{t};{\bf w}_L,{\bf w}_R,{\bf n}) = \left\{ \begin{array}{ll}  {\bf w}_L, & \tfrac{x}{t} < S_L, \\
 {\bf w}_L^\star, & S_L<\tfrac{x}{t}<u^\star,\\
 {\bf w}_R^\star, & u^\star<\tfrac{x}{t}<S_R,\\
 {\bf w}_R, & S_R < \tfrac{x}{t}, \end{array} \right.
\end{equation}

\noindent where ${\bf w}_X^\star=(\rho_X^\star {\bf Y}_X^\top, \rho_X^\star, \rho_X^\star {\bf v}_X^{\star\top},\rho_X^\star E_{r,X}^\star, \rho_X^\star e_{s,X})^\top$, for $X=L,R$, with $\rho_X^\star$ and ${\bf v}_X^{\star}$ defined in \cref{eq:sol-PR-relax}, $E_{r,X}^{\star}$ in \cref{eq:press-relax-RP-sol-details}, and $S_L=u_L-a_L/\rho_L$ and $S_R=u_R+a_R/\rho_R$ from \cref{eq:wave-estimate} by using $\p_r$ instead of $\p$.

%
%
\section{EC numerical flux plus dissipation at interfaces}\label{sec:appendix_ES_flux}

One common way to design an ES two-point numerical flux relies on adding upwind-type dissipation to EC numerical fluxes and has been first proposed by Roe \cite{Roe_2006,ismail_roe_09} (see also the introduction in \cref{sec:intro} for more references). We here derive such a numerical flux and show some numerical results for the sake of comparison. We recall that this approach does not provide Lipschitz continuous numerical fluxes and prevents to prove positivity of the solution. Moreover, such numerical fluxes use expensive operations such as the logarithmic mean \cite{ismail_roe_09}. 

\subsection{derivation of the ES numerical flux}\label{sec:deriv_ES_flux}

We define the ES flux as the sum of the EC flux \cref{eq:EC_flux} and dissipation:
\begin{equation}\label{eq:ES_flux}
 {\bf h}({\bf u}^-,{\bf u}^+,{\bf n}) = {\bf h}_{ec}({\bf u}^-,{\bf u}^+,{\bf n}) - \frac{\nu_{AD}}{2}\big|{\bf f}'({\bf u})\big|\mathbfcal{D}_\eta\du\etab'({\bf u})\df,
\end{equation}

\noindent where $\nu_{AD}>0$ is a parameter and $\mathbfcal{D}_\eta$ is a symmetric positive definite matrix. We use the scalar dissipation operator $\big|{\bf f}'({\bf u})\big|=\max\big(\lambda^\pm:\;\lambda=|{\bf v}\cdot{\bf n}|+c(\rho,e)\big){\bf I}_{n_c+d+1}$ for its robustness compared to other matrix dissipation types \cite{winters_etal_16}. We here follow the work in \cite{winters_etal_16} and look for a dissipation of the form $\mathbfcal{D}_\eta\du\etab'({\bf u})\df=\mathbfcal{D}_{\bf u}\du{\bf u}\df$. The dissipation reads
\begin{equation*}
 \mathbfcal{D}_{\bf u}\du{\bf u}\df = \begin{pmatrix} 0 \\ \vdots \\ 0 \\ \tfrac{r_{n_c}}{r(\ol{\bf Y})}\du\rho_{n_c}\df \\ \tfrac{r_{n_c}}{r(\ol{\bf Y})}\du\rho_{n_c}\df\ol{\bf v} + \widehat{\rho}\du{\bf v}\df \\ \tfrac{r_{n_c}}{r(\ol{\bf Y})}\Big(\tfrac{C_{v_{n_c}}}{\widehat{\theta}}+\tfrac{{\bf v}^+\cdot{\bf v}^-}{2}\Big)\du\rho_{n_c}\df + \widehat{\rho}\ol{\bf v}\cdot\du{\bf v}\df+\widehat{\rho}C_v(\ol{\bf Y})\du\T\df \end{pmatrix},
\end{equation*}

\noindent with $\theta=\tfrac{1}{\T}$, which results in the following entropy dissipation rate 
\begin{align*}
 \du\etab'({\bf u})\df\cdot\mathbfcal{D}_{\bf u}\du{\bf u}\df &\overset{\cref{eq:entropy_var}}{=} \du C_{p_{n_c}}-\mathrm{s}_{n_c}-\tfrac{{\bf v}\cdot{\bf v}}{2}\theta \df\tfrac{r_{n_c}}{r(\ol{\bf Y})}\du\rho_{n_c}\df + (\du\theta\df\ol{\bf v} + \ol{\theta}\du{\bf v}\df)\cdot\Big(\tfrac{r_{n_c}}{r(\ol{\bf Y})}\du\rho_{n_c}\df\ol{\bf v} + \widehat{\rho}\du{\bf v}\df\Big) \\
 &- \du\theta\df\bigg(\Big(\tfrac{C_{v_{n_c}}}{\widehat{\theta}}+\tfrac{{\bf v}^+\cdot{\bf v}^-}{2}\Big)\tfrac{r_{n_c}}{r(\ol{\bf Y})}\du\rho_{n_c}\df + \widehat{\rho}\big(\ol{\bf v}\cdot\du{\bf v}\df+C_v(\ol{\bf Y})\du\T\df\big)\bigg) \\
 &\overset{\cref{eq:partial_entropy}}{=} \tfrac{r_{n_c}^2}{r(\ol{\bf Y})}\du\rho_{n_c}\df\du\ln\rho_{n_c}\df + \widehat{\rho}\ol{\theta}\du{\bf v}\df\cdot\du{\bf v}\df - \Big(\ol{\tfrac{{\bf v}\cdot{\bf v}}{2}}-\ol{\bf v}\cdot{\bf v} + \tfrac{{\bf v}^+\cdot{\bf v}^-}{2}\Big)\tfrac{r_{n_c}}{r(\ol{\bf Y})}\du\rho_{n_c}\df\du\theta\df - \widehat{\rho}C_v(\ol{\bf Y})\du\T\df\du\theta\df\\
 &= \tfrac{r_{n_c}^2}{r(\ol{\bf Y})}\du\rho_{n_c}\df\du\ln\rho_{n_c}\df + \widehat{\rho}\ol{\theta}\du{\bf v}\df\cdot\du{\bf v}\df + \widehat{\rho}C_v(\ol{\bf Y})\tfrac{\ol{\theta}}{\ol{\T}}\du\T\df^2,
\end{align*}

\noindent \blue{which is a sum of positive terms for all ${\bf u}$ in $\Omega^a$}, so \cref{eq:ES_flux} is indeed ES in the sense of \cref{eq:entropy_stable_flux}:
\blue{
\begin{align*}
 \du \etab'({\bf u}) \df \cdot {\bf h}({\bf u}^-,{\bf u}^+,{\bf n}) - \du \psib({\bf u})\cdot{\bf n} \df &\overset{\cref{eq:ES_flux}}{=} \du \etab'({\bf u}) \df \cdot {\bf h}_{ec}({\bf u}^-,{\bf u}^+,{\bf n}) - \du \psib({\bf u})\cdot{\bf n} \df - \frac{\nu_{AD}}{2}\big|{\bf f}'({\bf u})\big|\du\etab'({\bf u})\df\cdot\mathbfcal{D}_\eta\du\etab'({\bf u})\df \\
 &\overset{\cref{eq:entropy_conserv_flux}}{=} - \frac{\nu_{AD}}{2}\big|{\bf f}'({\bf u})\big|\du\etab'({\bf u})\df\cdot\mathbfcal{D}_{\bf u}\du{\bf u}\df \leq 0.
\end{align*}
}

Note that we do not add numerical dissipation to the mass fraction equations as they are associated to a LD field and remain uniform across shocks. Likewise, the choice of the coefficient before $\du\rho_{n_c}\df$ is motivated by \cref{eq:partial_densities} so we are approximating the jump in $\rho$ with $r_{n_c}\du\rho_{n_c}\df/r(\ol{\bf Y})$.

\subsection{numerical experiments}\label{sec:num_xp_ES_flux}

\Cref{fig:solution_RP2_esf,fig:solution_SBI_esf} display some results obtained with the ES numerical flux \cref{eq:ES_flux} at interfaces and the EC flux \cref{eq:EC_flux} in the evaluation of the volume terms. Our numerical experiments highlight robustness issues, which led us to reduce the time step $\Delta t\max_{\kappa\in X_h}\tfrac{1}{|\kappa|}\sqrt{\sum_{e\in\kappa}|e|^2}|\lambda(\langle{\bf u}_h^{(n)}\rangle_\kappa)|\leq0.2$, where $|\lambda({\bf u})|=|{\bf v}\cdot{\bf n}|+c({\bf Y},e)$, compared to a $0.4$ bound with the relaxation-based numerical flux \cref{eq:relax_flux}. The viscosity coefficient $\nu_{AD}$ in \cref{eq:ES_flux} was chosen based on a parametric study to keep robustness of the computations as long as possible. We recall that such robustness issues were already reported in \cite[Sec.~5.4]{gouasmi_etal_20}. The time discretization and limiter are the same as used in \cref{sec:num_xp}. We display partial results of the 2D computations which did not go to their end and were stopped due to negative solution at some integration point. Compared to the results obtained with \cref{eq:relax_flux} at interfaces, we observe a similar behavior in the 1D Riemann problem in \cref{fig:solution_RP2_esf} \red{and spurious oscillations around shocks in the 2D experiments in \cref{fig:solution_SBI_esf}.} It is certainly possible to reduce such oscillations and improve robustness of the method by carefully designing the artificial dissipation in \cref{eq:ES_flux}, but this is beyond the scope of the present study.

\begin{figure}
\begin{bigcenter}
\subfloat[$Y_1$]{\epsfig{figure=\LOCALFIGPATHRP 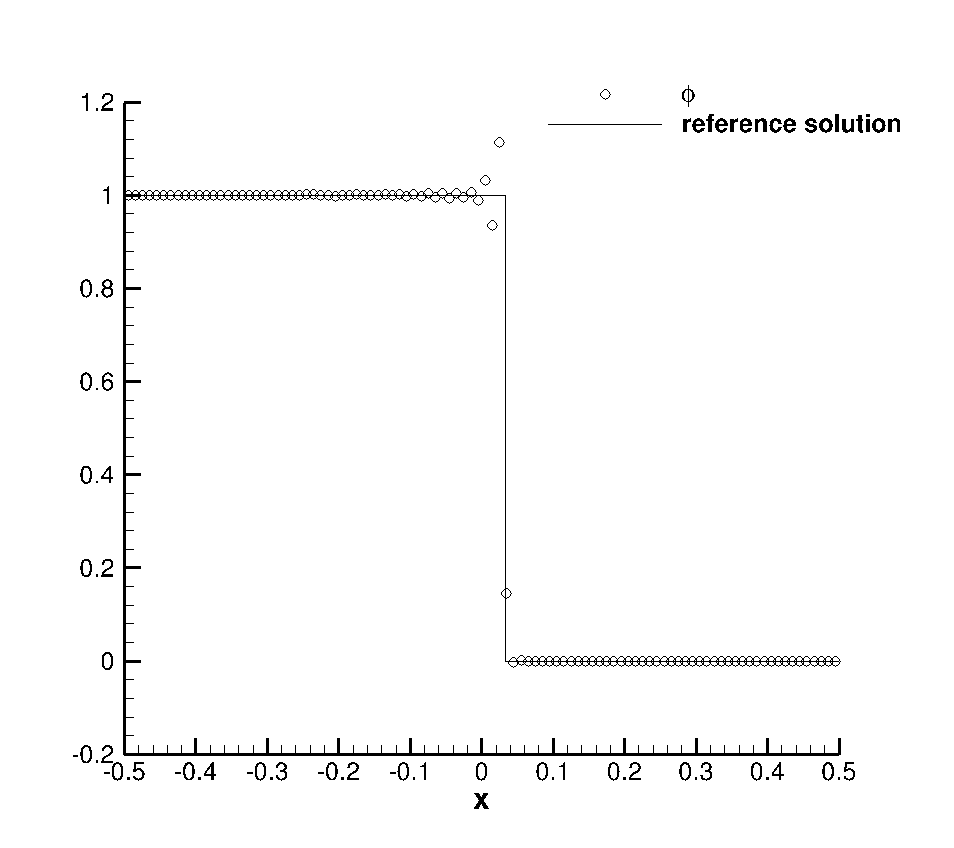 ,width=4.5cm}}
\subfloat[$\rho$]{\epsfig{figure=\LOCALFIGPATHRP 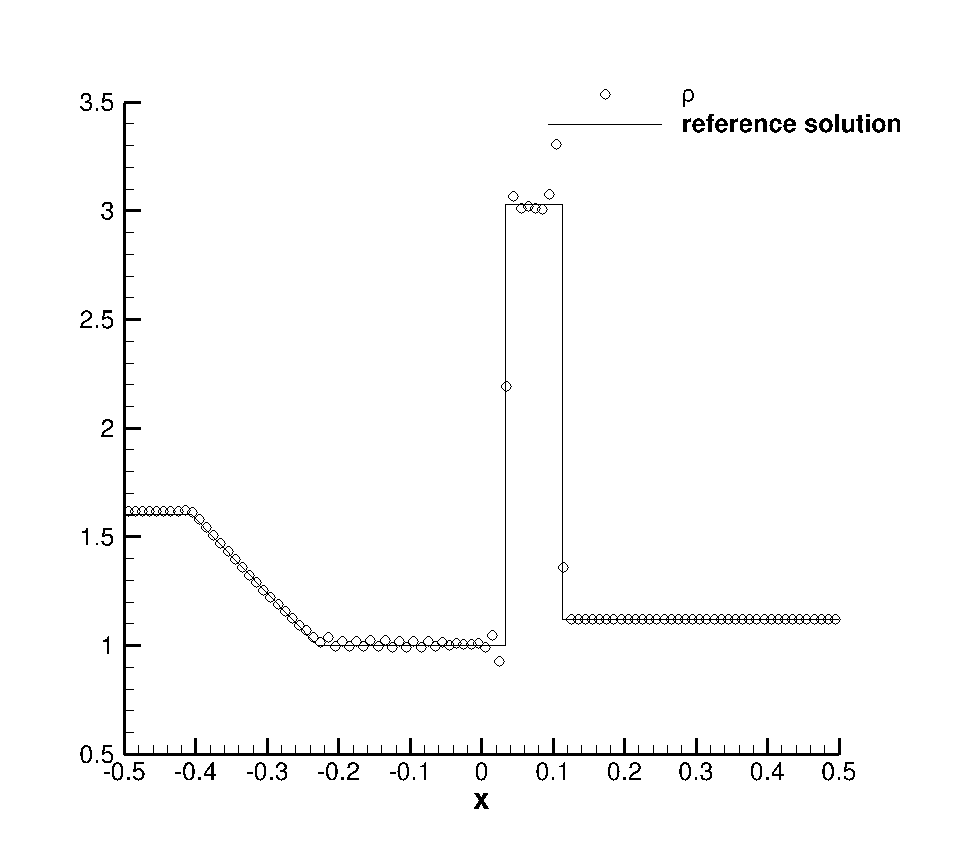 ,width=4.5cm}}
\subfloat[$10^{-3}\times u$]{\epsfig{figure=\LOCALFIGPATHRP 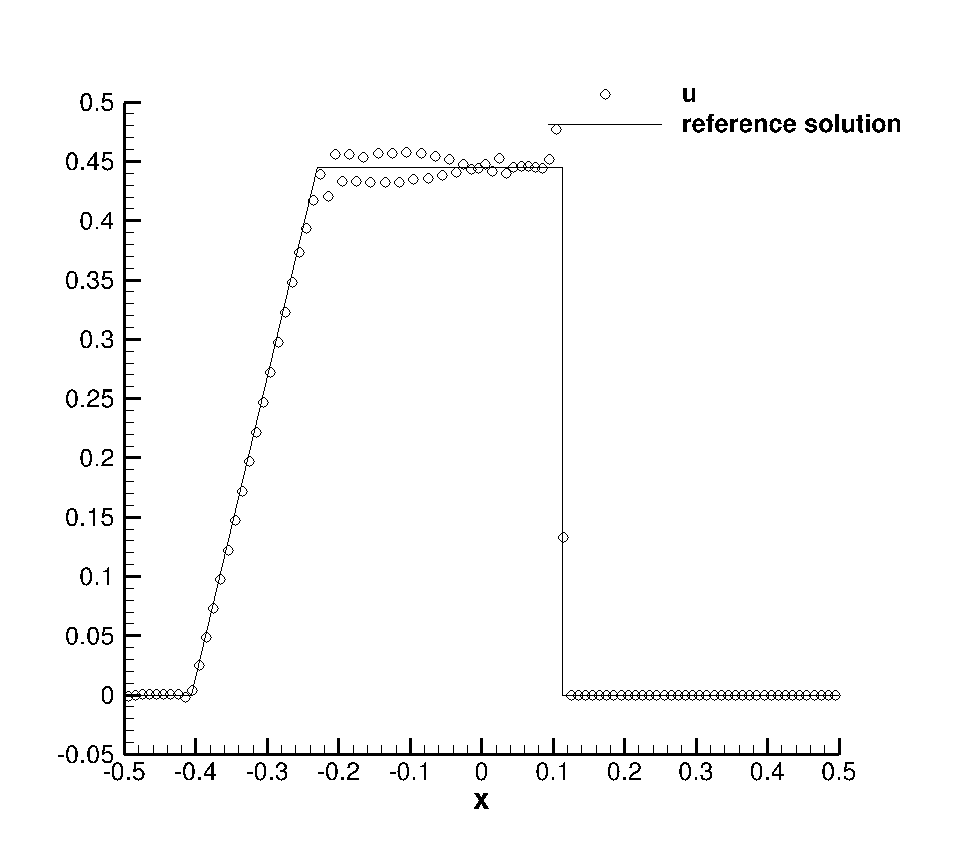 ,width=4.5cm}}
\subfloat[$10^{-6}\times \p$]{\epsfig{figure=\LOCALFIGPATHRP 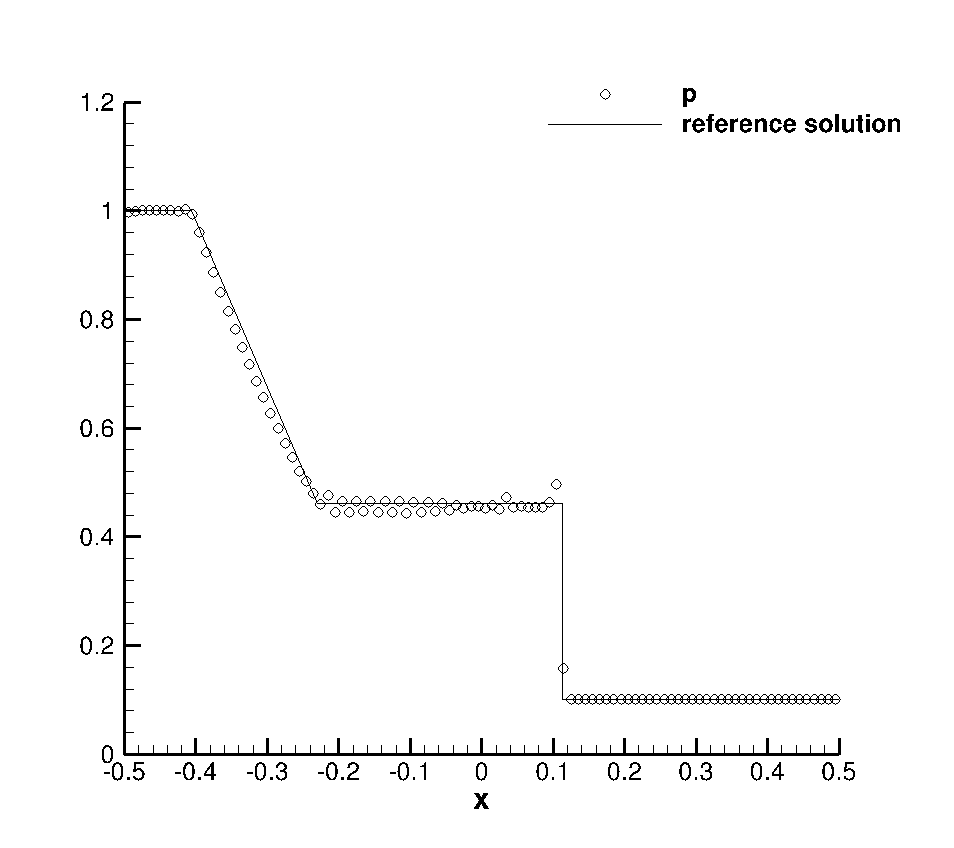 ,width=4.5cm}}
\caption{Riemann problem RP2 obtained with the Roe-like ES numerical flux \cref{eq:ES_flux} (with $\nu_{AD}=0.5$) at interfaces, $p=3$, and $N=100$ cells.}
\label{fig:solution_RP2_esf}
\end{bigcenter}
\end{figure}

\begin{figure}
\begin{bigcenter}
\captionsetup[subfigure]{labelformat=empty}
%
\subfloat[$t=32\;\mu$s]{\epsfig{figure=\LOCALFIGPATHSBIKT 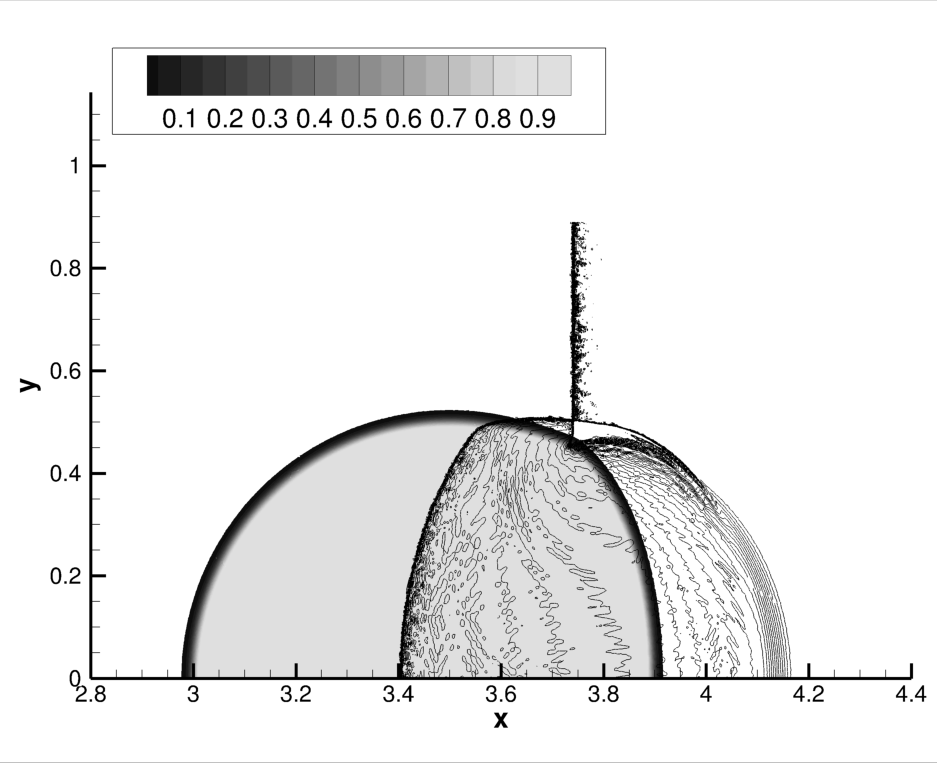,width=3.8cm}}
\subfloat{\epsfig{figure=\LOCALFIGPATHSBIKT 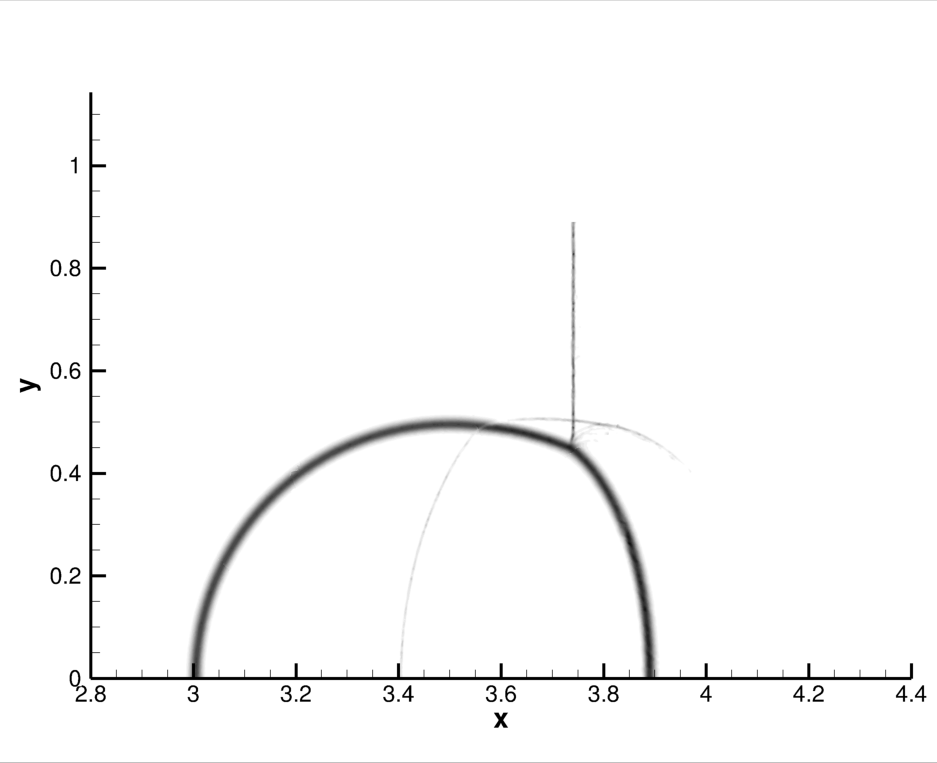,width=3.8cm}}
\subfloat[$t=102\;\mu$s]{\epsfig{figure=\LOCALFIGPATHSBIKT 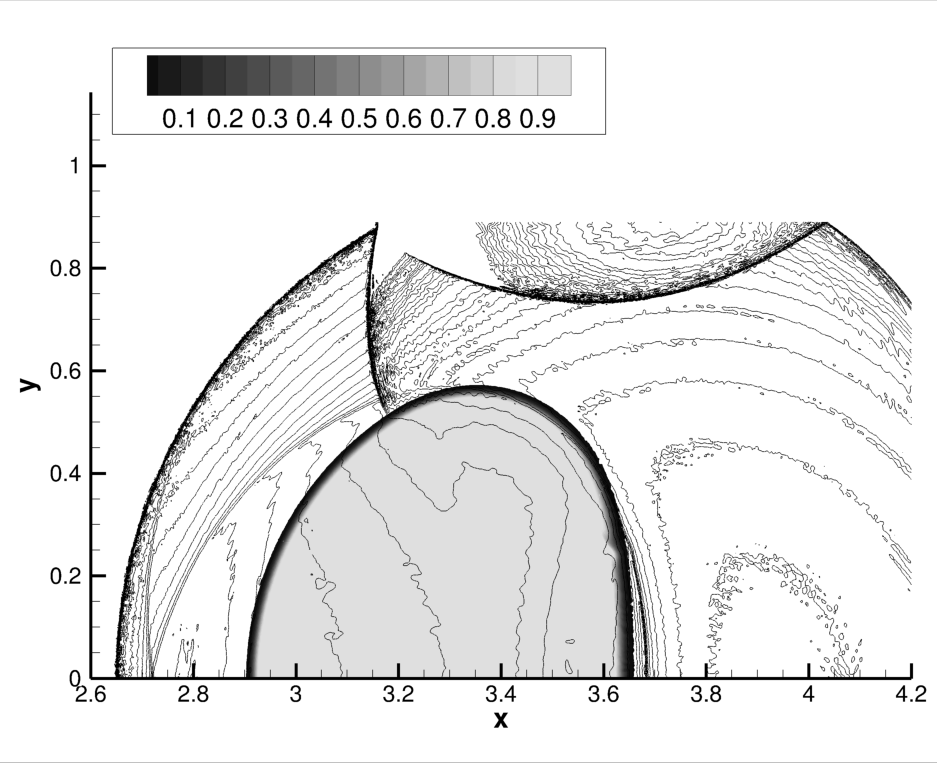,width=3.8cm}}
\subfloat{\epsfig{figure=\LOCALFIGPATHSBIKT 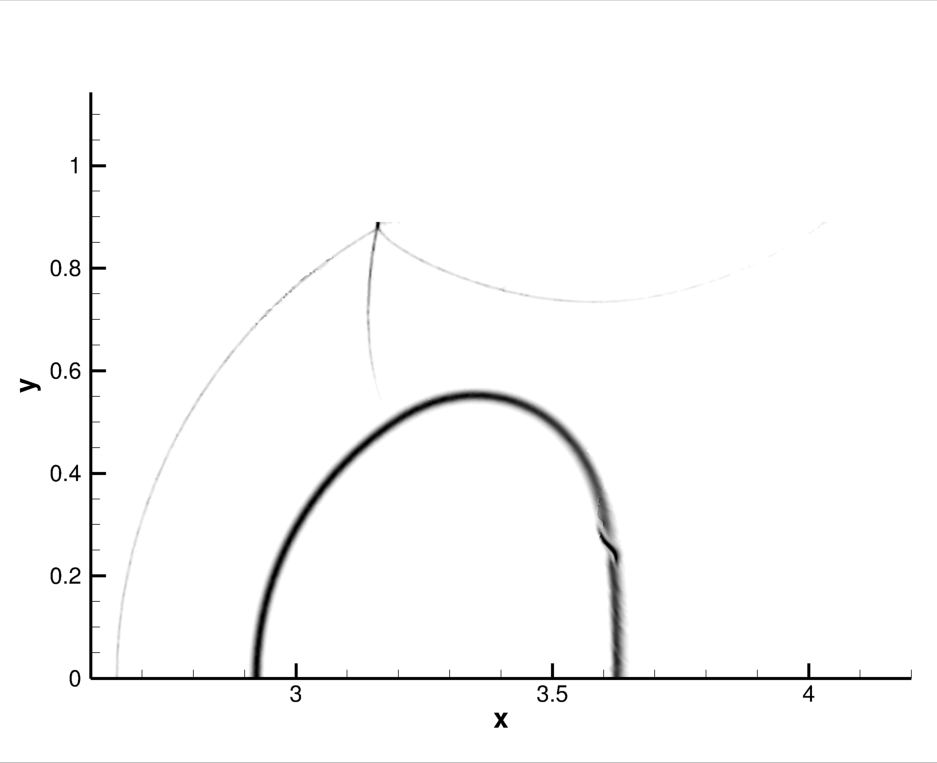,width=3.8cm}} \\
\subfloat[$t=62\;\mu$s]{\epsfig{figure=\LOCALFIGPATHSBIKT 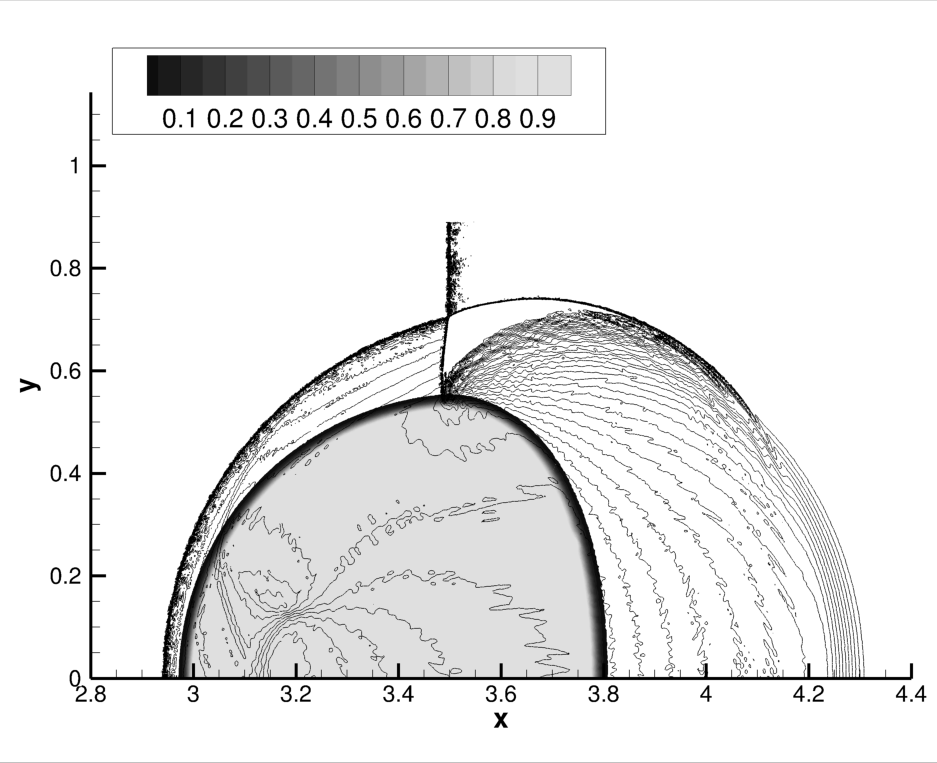,width=3.8cm}}
\subfloat{\epsfig{figure=\LOCALFIGPATHSBIKT 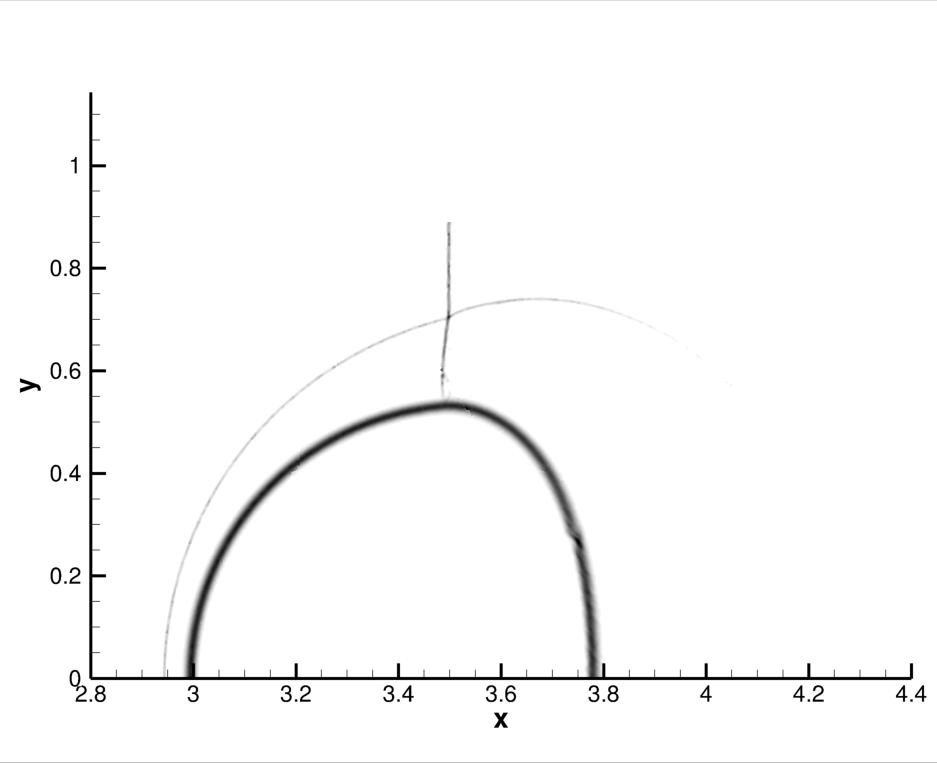,width=3.8cm}}
\subfloat[$t=240\;\mu$s]{\epsfig{figure=\LOCALFIGPATHSBIKT 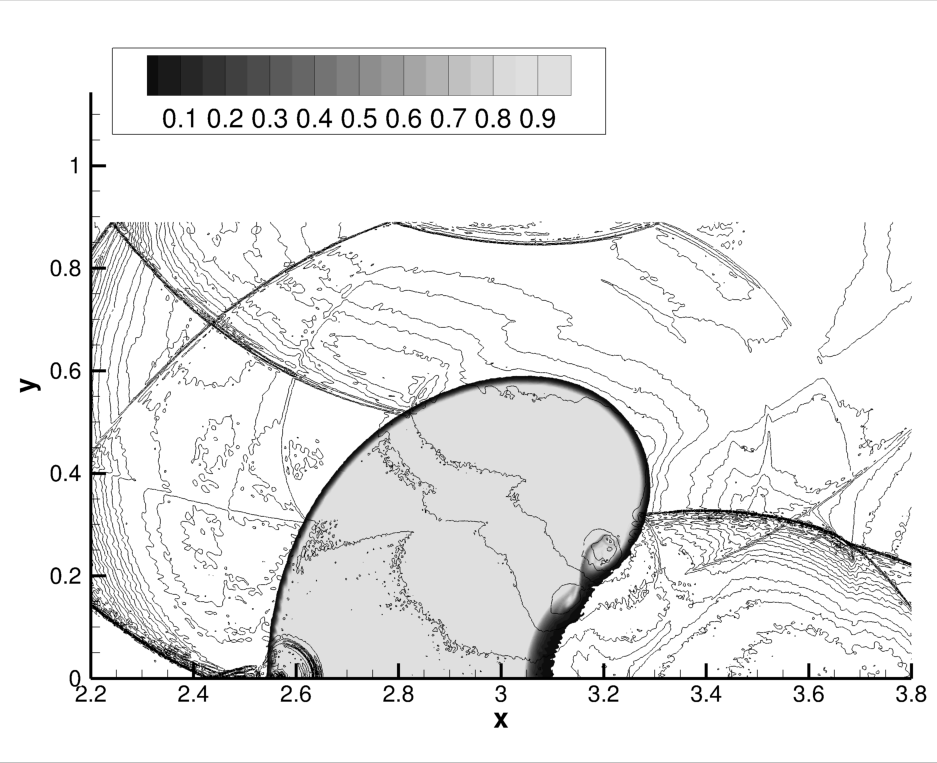,width=3.8cm}}
\subfloat{\epsfig{figure=\LOCALFIGPATHSBIKT 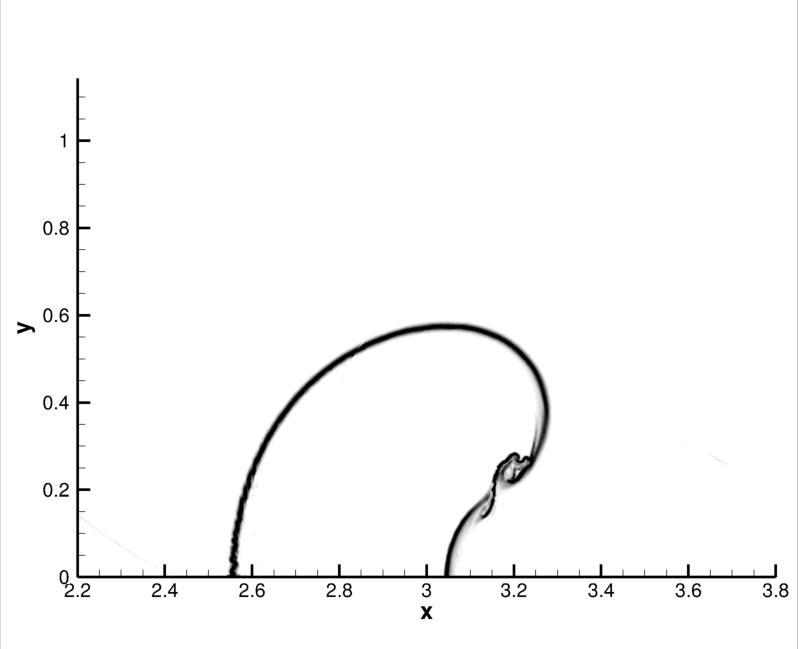,width=3.8cm}} 
\caption{Shock wave-helium bubble interaction: $70$ pressure contours (lines) and $20$ void fraction contours (colors), and Schlieren $|\nabla\rho|/\rho$ obtained with the Roe-like ES numerical flux \cref{eq:ES_flux} (with $\nu_{AD}=0.15$) at interfaces, $p=3$, and $N=238,673$ cells (see \cref{fig:solution_SBI} for details on the contours).
\label{fig:solution_SBI_esf}}
\end{bigcenter}
\end{figure}

%
%
\bibliographystyle{siam}
\bibliography{./biblio_generale}

\end{document}